\documentclass{daj}

\usepackage{amsmath}
\usepackage{amssymb}
\usepackage{mathtools}
\usepackage{mathcomp}
\usepackage{textcomp}
\usepackage{amsthm}
\usepackage{bm}
\usepackage{enumerate}
\usepackage{cite}
\usepackage{tikz}
\usepackage{caption}
\usepackage{subcaption}
\usetikzlibrary{arrows}
\usepackage{multicol}
\usepackage{tikz-qtree}
\usepackage{rotating}
\usepackage{url}
\usepackage{array}
\usepackage{appendix}

\usepackage{kbordermatrix}

\newtheorem{thm}{Theorem}[section]\theoremstyle{plain}
\newtheorem{theorem}[thm]{Theorem}\theoremstyle{plain}
\theoremstyle{plain}
\newtheorem{lemma}[thm]{Lemma}\theoremstyle{plain}
\theoremstyle{plain}

\theoremstyle{plain}
\newtheorem{claim}[thm]{Claim}\theoremstyle{plain}
\theoremstyle{plain}
\theoremstyle{plain}
\theoremstyle{plain}
\theoremstyle{plain}
\theoremstyle{plain}
\theoremstyle{plain}
\newtheorem{conjecture}[thm]{Conjecture}\theoremstyle{plain}
\theoremstyle{plain}

\DeclareMathOperator{\rank}{rank}

\DeclareMathOperator{\cl}{cl}

\newcommand{\bp}{{\bm p}}
\newcommand{\bq}{{\bm q}}
\newcommand{\bt}{{\bm t}}
\newcommand{\bQ}{{\bm Q}}

\newcommand{\bv}{{\bm v}}
\newcommand{\be}{{\bm e}}

\newcommand{\MM}{{\mathcal M}}

\newcommand{\R}{{\mathbb R}}
\newcommand{\F}{{\mathbb F}}
\newcommand{\rat}{{\mathbb Q}}
\newcommand{\bP}{{\mathbb P}}
\newcommand{\clo}{{\rm cl}}

\dajAUTHORdetails{%
  title = {Abstract 3-Rigidity and Bivariate $C_2^1$-Splines I: Whiteley's Maximality Conjecture}, 
  author = {Katie Clinch, Bill Jackson,  and Shin-ichi Tanigawa},
  plaintextauthor = {Katie Clinch, Bill Jackson,  and Shin-ichi Tanigawa},
  plaintexttitle = {Abstract 3-Rigidity and Bivariate C_2^1-Splines I: Whiteley's Maximality Conjecture}, 
  runningtitle = {Abstract 3-Rigidity and Bivariate $C_2^1$-Splines I: Whiteley's Maximality Conjecture}, 
  runningauthor = {Katie Clinch, Bill Jackson,  and Shin-ichi Tanigawa},
  copyrightauthor = {Katie Clinch, Bill Jackson,  and Shin-ichi Tanigawa},
  keywords = {graph rigidity, abstract rigidity matroid, bivariate spline, cofactor matroid},
}   

\dajEDITORdetails{%
   year={2022},
   number={2},
   received={10 November 2019},   
   revised={22 June 2021},    
   published={22 April 2022},  
   doi={10.19086/da.34691},       
}   

\begin{document}

\begin{frontmatter}[classification=text]


\author[katie]{Katie Clinch
}
\author[bill]{Bill Jackson
}
\author[shin]{Shin-ichi Tanigawa
}

\begin{abstract}
A conjecture of Graver from 1991 states that the generic 3-dimensional rigidity matroid is the unique maximal abstract 3-rigidity matroid with respect to the weak order on matroids. Based on a close similarity between the generic $d$-dimensional rigidity matroid and the generic $C_{d-2}^{d-1}$-cofactor matroid from approximation theory, Whiteley made an analogous conjecture in 1996 that the generic $C_{d-2}^{d-1}$-cofactor matroid is the unique maximal abstract $d$-rigidity matroid for all $d\geq 2$. We verify the case $d=3$ of Whiteley's conjecture in this paper. A key step in our proof is to verify a second conjecture of Whiteley that the `double V-replacement operation' 
preserves independence in the  generic $C_2^1$-cofactor matroid.
\end{abstract}
\end{frontmatter}


\section{Introduction}\label{sec:cofactor}
The statics of skeletal structures strongly depends on their underlying graphs. This connection has been studied since the seminal work of James Clerk Maxwell \cite{Max} in 1864, and combinatorial rigidity theory is now a fundamental topic in structural rigidity~\cite{GSS93,Wsurvey}.

We will consider a {\em $d$-dimensional (bar-joint) framework},  which is a pair $(G,\bp)$ consisting of a finite graph $G=(V,E)$ and a map $\bp:V\rightarrow \mathbb{R}^d$. 
Asimow and Roth~\cite{AR78} and Gluck~\cite{Glu} observed that the properties of rigidity and infinitesimal rigidity   coincide when $\bp$ is  generic (i.e., the set of coordinates in $\bp$ is algebraically independent over the rational field),  
and  are completely determined by $G$ and $d$. This fact enables us to define the generic $d$-dimensional rigidity matroid ${\cal R}_d(G)$ on the edge set of $G$, whose rank characterizes the rigidity of any generic framework $(G,\bp)$. 
(See, e.g.,~\cite{GSS93,SW,Wsurvey} for more details.)

The generic $1$-dimensional rigidity matroid ${\cal R}_1(G)$ is equal to the graphic matroid of $G$, and a celebrated theorem of Pollaczek-Geiringer \cite{P-G} and Laman~\cite{Lam} can be used to obtain a concise combinatorial formula for the rank function of the generic $2$-dimensional rigidity matroid ${\cal R}_2(G)$, see \cite{LY82}. Finding a combinatorial description for ${\cal R}_3(G)$ remains a fundamental open problem in the field which has inspired a significant amount of research.

In 1991, Graver~\cite{G91} suggested the approach of  abstracting representative properties of the $d$-dimensional rigidity matroid and recasting rigidity as a matroid property.
Suppose that $(G,\bp)$, $(G_1, \bp_1)$ and $(G_2, \bp_2)$ are generic $d$-dimensional frameworks such that $G=G_1\cup G_2$
and $\bp_i$ is the restriction of $\bp$ to  $V(G_i)$ for $i=1,2$. 
It is intuitively clear that: 
\begin{itemize}
\item if $|V(G_1)\cap V(G_2)|\leq d-1$ then adding any edge between $V(G_1)\setminus V(G_2)$ and $V(G_2)\setminus V(G_1)$ will restrict the `flexibility' of $(G,\bp)$, see  Figure~\ref{fig:glueing}(a); 
\item if $|V(G_1)\cap V(G_2)|\geq d$ and $(G_1,\bp_1)$ and $(G_2,\bp_2)$ are both rigid then $(G,\bp)$ will be rigid, see Figure~\ref{fig:glueing}(b).
\end{itemize}
Graver~\cite{G91} observed that these physical properties can be 
described in terms of matroid closure,  and proposed to investigate all matroids which satisfy these properties.  More precisely, 
he defined a matroid $M$ on the edge set $K(V)$ of the complete graph with vertex set $V$ 
to be an {\em abstract $d$-rigidity matroid}  if the following two properties hold:
\begin{description}
\item[(R1)] If $E_1, E_2\subseteq K(V)$ with $|V(E_1)\cap V(E_2)|\leq d-1$, then ${\rm cl}_M(E_1\cup E_2)\subseteq K(V(E_1))\cup K(V(E_2))$;
\item[(R2)] If $E_1, E_2\subseteq K(V)$ with ${\rm cl}_M(E_1)=K(V(E_1)), {\rm cl}_M(E_2)=K(V(E_2))$, and $|V(E_1)\cap V(E_2)|\geq d$, then ${\rm cl}_M(E_1\cup E_2)=K(V(E_1\cup E_2))$,
\end{description}
where 
${\rm cl}_M$ denotes the closure operator of $M$.
The generic $d$-dimensional rigidity matroid for $K(V)$, ${\cal R}_d(V)$,  is an example of an abstract $d$-rigidity matroid.
Graver~\cite{G91} conjectured that, for all $d\geq 1$, there is a unique maximal abstract $d$-rigidity matroid on $K(V)$ with respect to the weak order of matroids, and further that ${\cal R}_d(V)$ is this maximal matroid. 
He verified his conjecture for $d=1,2$ but N.~J.~Thurston (see, \cite[page 150]{GSS93}) subsequently showed that ${\cal R}_d(V)$ is not the unique maximal abstract $d$-rigidity matroid when $d\geq 4$.
We will verify the first part of Graver's conjecture when $d=3$.\footnote{The first part of Graver's conjecture, that there exists a unique maximal abstract $d$-rigidity matroid on $K(V)$, would follow for all $d$ as a special case of the main result in a preprint of Sitharam and Vince~\cite{SV}. Unfortunately this more general result  is false, see \cite{PAP, JT}.}
The second part 
remains as a long-standing open problem.

\begin{conjecture}[Maximality conjecture]\label{conj:max}
The generic $3$-dimensional rigidity matroid  on $K(V)$ is the unique maximal abstract $3$-rigidity matroid on $K(V)$.
\end{conjecture}

\begin{figure}[t]
\centering
\begin{minipage}{0.45\textwidth}
\centering
\includegraphics[scale=0.65]{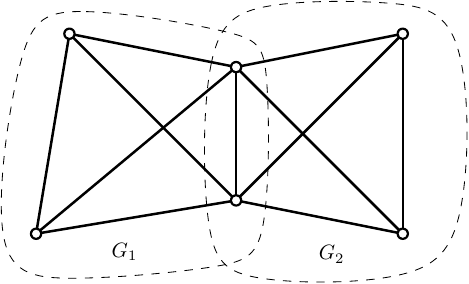}
\par
(a)
\end{minipage}
\begin{minipage}{0.45\textwidth}
\centering
\includegraphics[scale=0.65]{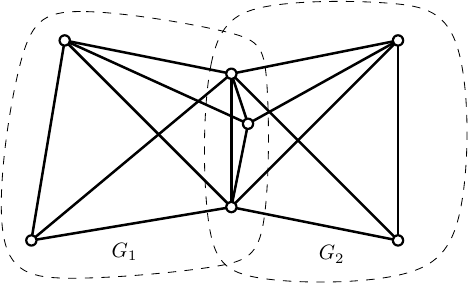}
\par
(b)
\end{minipage}
\caption{Examples in $d=3$: (a) When $|V(G_1)\cap V(G_2)|=2$, 
$(G_1,\bp_1)$ and $(G_2,\bp_2)$ can rotate about the line through the points of $V(G_1)\cap V(G_2)$, and adding any edge between $V(G_1)\setminus V(G_2)$ and $V(G_2)\setminus V(G_1)$ will restrict this motion.
(b)  Rigid frameworks $(G_1,\bp_1)$ and $(G_2,\bp_2)$ are glued together along three vertices, and the resulting framework is rigid.}
\label{fig:glueing}
\end{figure}


Whiteley~\cite{Wsurvey} found another candidate for a maximal abstract $d$-rigidity matroid 
from approximation theory.
Consider a plane polygonal domain $D$ which has been subdivided into  a set of polygonal faces $\Delta$. 
A {\em bivariate $C_s^r$-spline} over $\Delta$ is a function on $D$ which is continuously differentiable $r$ times and is given by a polynomial of degree $s$ over each face of $\Delta$. The set of all $C_s^r$-splines over $\Delta$ forms a vector space $S_s^r(\Delta)$, 
and 
determining the dimension of $S_s^r(\Delta)$ is a major question in this area.
When $r=0$ and $d=1$, $S_1^0(\Delta)$ is the space of piecewise linear functions, 
and {\em Maxwell's reciprocal diagram } gives a concrete correspondence between $S_1^0(\Delta)$ and  the self-stresses of the $1$-skeleton of $\Delta$, regarded  as a 2-dimensional framework $(G,\bp)$.
Whiteley~\cite{W90,W91,Wsurvey} extended this connection further and demonstrated  how the dimension of $S_s^r(\Delta)$ can be computed from the rank of the  {\em $C_s^r$-cofactor matrix $C_s^r(G,\bp)$},  which is a variant of the incidence matrix of $G$ in which the entries are replaced by block matrices whose entries are polynomials in $\bp$.
Billera~\cite{B88} used Whiteley's approach to solve a conjecture of Strang on the generic dimension of $S_s^1(\Delta)$.

The definition of  $C_s^r(G,\bp)$ is not dependent on the fact that  $(G,\bp)$ is a realisation of $G$ in the plane without crossing edges and is equally valid for  any 2-dimensional framework $(G,\bp)$. Moreover,
in the case when $r=s-1$, each row of $C_s^{s-1}(G,\bp)$ is associated with a distinct edge of $G$,
and hence the row matroid of $C_s^{s-1}(G,\bp)$ is a matroid defined on $E(G)$. We refer to this matroid  as the {\em $C_s^{s-1}$-cofactor  matroid}
for $(G,\bp)$ and denote it by 
${\cal C}_s^{s-1}(G,\bp)$.
Since the matroids ${\cal C}_s^{s-1}(G,\bp)$ are the same for all generic $\bp$, the {\em generic $C_s^{s-1}$-cofactor  matroid} of $G$, denoted  by ${\cal C}_s^{s-1}(G)$, is defined to be ${\cal C}_s^{s-1}(G,\bp)$ for any generic $\bp$. (See Section \ref{subsec:generic_cofactor} for a formal definition).
Let ${\cal C}_{d-1}^{d-2}(V)={\cal C}_{d-1}^{d-2}(K(V))$.
Whiteley~\cite{Wsurvey} showed that  ${\cal C}_{d-1}^{d-2} (V)$ is an example of an abstract $d$-rigidity matroid and pointed out that  ${\cal C}_{d-1}^{d-2} (V)={\cal R}_{d} (V)$ when $d=2$. He also showed that ${\cal C}_{d-1}^{d-2} (V)$ is a counterexample to Graver's original maximality conjecture for all $d\geq 4$ and conjectured further that  ${\cal C}_{d-1}^{d-2} (V)$  is the maximal abstract $d$-rigidity matroid for all $d\geq 2$.

In this paper we verify Whiteley's conjecture for $d=3$, and hence prove the cofactor counterpart of Conjecture~\ref{conj:max}:
\begin{theorem}\label{thm:cofactor}
The generic $C_2^1$-cofactor matroid ${\cal C}_2^1(V)$
is the unique maximal abstract 3-rigidity matroid on $K(V)$.  
\end{theorem}

The most difficult part of the proof of Theorem~\ref{thm:cofactor} is to show that a certain graph operation (called the double V-replacement operation) preserves independence in the generic $C_2^1$-cofactor  matroid. This is related to another long-standing conjecture on the generic 3-dimensional rigidity matroid, known as the {\em Henneberg construction conjecture},  which asks if every base of the generic 3-dimensional rigidity matroid ${\cal R}_3(V)$ can be constructed from $K_4$  by a sequence of simple graph operations (see~\cite{TW85,C14} for more details). 
Based on a strong similarity between rigidity matroids and cofactor matroids, Whiteley~\cite[page 55]{W90} posed a corresponding Henneberg construction conjecture for the generic $C_2^1$-cofactor  matroid. 
We will verify this conjecture.

In his survey on generic rigidity and cofactor matroids, Whiteley gave a table of properties and conjectures for the generic $3$-dimensional rigidity and $C_2^1$-cofactor matroids
~\cite[page 65]{Wsurvey}. The table below updates Whiteley's table with results from  \cite{JJ05} and the current paper. 

\begin{center}
\begin{tabular}{|l||c|c|} \hline
 & Generic $3$-dim rigidity &Generic $C_2^1$-cofactor \\ \hline\hline
 rank $K_n$, $n \geq 3$ & $3n-6$ & $3n-6$ \\\hline
0-extension (vertex addition) & YES & YES \\\hline
1-extension (edge split) & YES & YES \\\hline
abstract 3-rigidity & YES & YES  \\\hline
vertex splitting & YES & YES  \\\hline
simplicial 2-surfaces & Rigid
~\cite{Glu} 
& Rigid
~\cite{B88} 
\\\hline
coning & YES & YES  \\\hline
X-replacement & Conjectured & YES 
\cite{Wsurvey} 
\\\hline
double V-replacement & Conjectured & YES (Theorem~\ref{thm:doubleVconj}) \\\hline
Dress conjecture & NO \cite{JJ05} & NO \cite{JJ05} \\\hline
$K_{4,6}$ & Base & Base \\\hline
$K_{5,5}$ & Circuit & Circuit \\\hline
maximal abstract 3-rigidity matroid & Conjectured & YES  (Theorem~\ref{thm:cofactor}) \\\hline
\end{tabular}
\end{center}

We will give a combinatorial characterization of the rank function of the maximal abstract 3-rigidity matroid ${\cal C}_2^1(V)$ in a companion paper~\cite{CJT1} to this one, and hence solve the cofactor counterpart to  the combinatorial characterization problem for $3$-dimensional rigidity. 
Theorem~\ref{thm:cofactor} is a key ingredient of the proof in this companion paper.

This paper focuses on proving Theorem~\ref{thm:cofactor}, and hence we restrict our attention to $C_2^1$-cofactor matroids. 
In Section~\ref{sec:preliminaries}, we first provide a brief introduction to $C_2^1$-cofactor matroids, and then give a detailed roadmap for the proof of Theorem~\ref{thm:cofactor}. 

\section{Preliminaries}\label{sec:preliminaries}
In this section we give a formal definition of $C_2^1$-cofactor matroids and describe their fundamental properties. 
We then give a roadmap for the proof of Theorem~\ref{thm:cofactor}.

 We use the following notation throughout this paper.
For a finite set $V=\{v_1,v_2,\ldots,v_n\}$, let $K(V)$, or $K(v_1,v_2,\ldots,v_n)$, denote the edge set of the complete graph with vertex set  $V$.
For a graph $G=(V,E)$ and $v\in V$, let $N_G(v)$ be the set of neighbors of $v$ in $G$ and $\hat{N}_G(v)=N_G(v)\cup \{v\}$ be the closed neighborhood of $v$ in $G$.
For an edge set $F$ and a vertex $v$, let $d_F(v)$ denote the number of edges of $F$ incident to $v$. 
A {\em star on $n$ vertices} is a tree with $n$ vertices in which one vertex is adjacent to all other vertices.

We often regard a map $\bp:V\to \R^k$ as a $k|V|$-dimensional vector. 
The inner product $\bp\cdot \bq$ of two maps $\bp, \bq:V\to \R^k$ is given by this identification.

For vectors $\bv_1,\dots, \bv_k$ in a Euclidean space, let $\langle \bv_1,\dots, \bv_k\rangle$ be their linear span.

For a set $Z$ of real numbers, let $\mathbb{Q}(Z)$ denote the smallest subfield of $\mathbb{R}$ that contains the rationals and $Z$.

The closure operator and the rank function of a matroid ${\cal M}$ are 
denoted by ${\rm cl}_{{\cal M}}$ and $r_{{\cal M}}$, respectively.
The {\em weak order for matroids} is a partial order over all matroids with the same groundset $E$, where
for two matroids ${\cal M}_i=(E, r_{{\cal M}_i})\ (i=1,2)$ 
we have 
${\cal M}_1\preceq {\cal M}_2$ if $r_{{\cal M}_1}(X)\leq r_{{\cal M}_2}(X)$ holds
for all $X\subseteq E$. 

\subsection{\boldmath The $C_2^1$-cofactor Matrix}\label{subsec:generic_cofactor}

Given two points $p_i=(x_i,y_i)$ and $p_j=(x_j,y_j)$ in $\R^2$, we define $D(p_i,p_j)\in \mathbb{R}^{3}$ by
\[D(p_i,p_j)=((x_i-x_j)^2, (x_i-x_j)(y_i-y_j), (y_i-y_j)^2).
\] 
For a 2-dimensional framework $(G,\bp)$ with vertex set $V=\{v_1,v_2,\ldots,v_n\}$, we simply write $D(v_i,v_j)$ for  $D(\bp(v_i),\bp(v_j))$ when $\bp$ is clear from the context.

We define the {\em $C^{1}_2$-cofactor matrix} of a 2-dimensional framework $(G,\bp)$  to be the matrix $C_2^{1}(G,\bp)$ of size 
$|E|\times 3|V|$ in which each vertex is associated with a set of three consecutive  columns,
each edge is associated with a row,
and the row associated with the  edge $e=v_iv_j$ with $i<j$ is 
\[
\kbordermatrix{
 & & v_i & & v_j & \\
 e=v_iv_j & 0\cdots 0 & D(v_i,v_j) & 0 \cdots 0 & -D(v_i,v_j) & 0\cdots 0 
}.
\]
For example, for $(K_4,\bp)$ with $\bp(v_1)=(0,0)$, $\bp(v_2)=(1,0)$, $\bp(v_3)=(0,1)$ and $\bp(v_4)=(-1,-1)$, 
\[
C_2^1(K_4,\bp)=
\kbordermatrix{
 & & v_1 &  & & & v_2 & & & & v_3 & & & & v_4 & \\
v_1v_2 & 1 & 0 & 0 & & -1 & 0 & 0 & & 0 & 0 & 0 & & 0 & 0 & 0 \\ 
v_1v_3 & 0 & 0 & 1 & & 0 & 0 & 0 & & 0 & 0 & -1 & & 0 & 0 & 0 \\ 
v_1v_4 & 1 & 1 & 1 & & 0 & 0 & 0 & & 0 & 0 & 0 & & -1 & -1 & -1 \\
v_2v_3 & 0 & 0 & 0 & & 1 & -1 & 1 & & -1 & 1 & -1 & & 0 & 0 & 0 \\
v_2v_4 & 0 & 0 & 0 & & 4 & 2 & 1 & & 0 & 0 & 0 & & -4 & -2 & -1 \\
v_3v_4 & 0 & 0 & 0 & & 0 & 0 & 0 & & 1 & 2 & 4 & & -1 & -2 & -4 
}.
\]
(Our definition is slightly different to that given by Whiteley \cite{Wsurvey}, but the two definitions are equivalent up to elementary column operations.)
It is known that 
the space $S_2^{1}(\Delta)$ of bivariate $C_2^{1}$-splines over a subdivision $\Delta$ of a polygonal domain in the plane is linearly isomorphic to the left kernel of  $C_2^{1}(G,\bp)$ if $(G,\bp)$ is the 1-skeleton of $\Delta$.
See, \cite{Wsurvey} for more details.

\subsection{\boldmath $C_2^1$-motions} \label{subsec:C_2^1-motions}
As mentioned above, the left kernel of the $C_2^1$-cofactor matrix $C_2^1(G,\bp)$ has an important role in the analysis of $C_2^1$-splines. An analogous situation occurs in 
rigidity theory
where the left kernel of the rigidity matrix is the space of {\em self-stresses} of the framework. 
The right kernel of the rigidity matrix plays an equally important role in 
rigidity theory since it is the space of {\em infinitesimal motions} of the framework. 
In order to apply techniques from rigidity theory to cofactor matrices, 
we will also consider the right kernel of the $C_2^1$-cofactor matrix.

Let $G=(V,E)$ be a graph and $\bp:V\to \R^2$ such that $\bp(v_i)=(x_i,y_i)\in \mathbb{R}^2$ for all $v_i\in V$.
A {\em $C_2^1$-motion} (or simply, a {\em motion}) of $(G,\bp)$ is a map $\bq:V\rightarrow \mathbb{R}^3$ satisfying
\begin{equation}\label{eq:cofactor_inf}
D(v_i,v_j)\cdot ( \bq(v_i)-\bq(v_j))=0\qquad (v_iv_j\in E).
\end{equation}
The $C_2^1$-cofactor matrix
$C_2^1(G,\bp)$ defined in Section~\ref{subsec:generic_cofactor} is the matrix of coefficients of this system 
of linear equations in the variable $\bq$,
and hence each $C_2^1$-motion $\bq$ is a vector in the right kernel $Z(G,\bp)$ of $C_2^1(G,\bp)$.

Whiteley \cite{Wsurvey} showed that $Z(G,\bp)$  
has dimension at least six  when $\bp(V)$ affinely spans $\R^2$ by showing that  the $C_2^1$-motions $\bq_i^*:V\rightarrow \mathbb{R}^3$, $1\leq i\leq 6$, defined by
\begin{align}
\bq_1^*(v_i)&=(1,0,0), &  \bq_2^*(v_i)&=(0,1,0), &  \bq_3^*(v_i)&=(0,0,1),\label{eq:rigid_motions1} \\ 
\bq_4^*(v_i)&=(y_i,-x_i,0), & \bq_5^*(v_i)&=(0,-y_i,x_i) &  \bq_6^*(v_i)&=(y_i^2,-2x_iy_i,x_i^2) \label{eq:rigid_motions2}
\end{align}  
for each $v_i\in V$,
are linearly independent  vectors in $Z(G,\bp)$ for all $G$. (We have adjusted the values of the $\bq_i^*$ given in \cite{Wsurvey} to fit our modified definition of $C_2^1(G,\bp)$.)
We will refer to a $C_2^1$-motion that can be described as a linear combination of the $\bq_i^*, 1\leq i\leq 6$ as a {\em trivial} $C_2^1$-motion,
and let $Z_0(G,\bp)$ be the space of all trivial $C_2^1$-motions.
 
 The following terminologies are $C_2^1$-cofactor analogues of standard terminologies in rigidity theory.
 We say that a framework $(G,\bp)$ is: {\em $C_2^1$-rigid} if $Z(G,\bp)= Z_0(G,\bp)$;  {\em $C_2^1$-independent} if
 the rows of $C_2^1(G,\bp)$ are linearly independent;  {\em minimally $C_2^1$-rigid} if $(G,\bp)$ is both $C_2^1$-rigid and $C_2^1$-independent; a {\em $k$-degree of freedom framework}, or {\em $k$-dof framework} for short, if $\dim Z(G,\bp)=6+k$.
We will use the same terms for the graph $G$ if  $(G,\bp)$ has the corresponding properties for some (or equivalently, every) generic $\bp$.
 
%

\subsection{\boldmath Generic ${C}_2^1$-cofactor Matroids}
The generic {\em $C_2^{1}$-cofactor matroid}, ${\cal C}_{2,n}^{1}$, is the matroid on $E(K_n)$ in which independence is given by the linear independence of the  rows  of $C_2^{1}(K_n,\bp)$ for any generic $\bp$. 
We will sometimes simplify ${\cal C}_{2,n}^1$ to ${\cal C}_n$ or even ${\cal C}$ when it is clear from the context.
Note that a graph $G$ with $n$ vertices is minimally $C_2^1$-rigid (resp., $C_2^1$-independent) if and only if $E(G)$ is a base (resp., an independent set) of 
${\cal C}_{2,n}^1$. 

Since $\dim Z(K_n,\bp)\geq \dim Z_0(K_n,\bp)=6$ holds when $n\geq 3$ and  $\bp$ is generic, the rank of ${\cal C}_{2,n}^1$ is at most $3n-6$ for all $n\geq 3$.
Whiteley~\cite[Corollary 11.3.15]{Wsurvey} showed that the rank of ${\cal C}_{2,n}^1$ is equal to $3n-6$ when $n\geq 3$, and, more significantly,  that ${\cal C}_{2,n}^{1}$ is an abstract $3$-rigidity matroid.
Our main result, Theorem~\ref{thm:cofactor},  verifies that ${\cal C}_{2,n}^{1}$ is the unique maximal abstract $3$-rigidity matroid on $E(K_n)$.

\subsection{Pinning}\label{subsec:pinning}
 
The technique of {\em pinning}, i.e.\ specifying the value of a motion at one or more vertices, is frequently used when dealing with a $k$-dof framework $(G,\bp)$ in rigidity theory,
and we will also use this technique in our analysis. 
The following technical lemmas will be used only in Sections~\ref{sec:bad} and~\ref{sec:well-behaved}, so the reader may wish to skip this subsection and refer back to it later.

For $v_s\in V(G)$ and $t\in \{1,2,3\}$, let ${\bf e}_{s,t}: V\rightarrow \R^3$ be such that ${\bf e}_{s,t}(v_s)$ is the unit vector with $1$ at the $t$-coordinate and zeros elsewhere, and ${\bf e}_{s,t}(v_i)=0$ for  $i\neq s$.

\begin{lemma}\label{lem:canonical}
Let $(G,\bp)$ be a 
$k$-dof framework such that $\bp(v_i)=(x_i,y_i)\in \mathbb{R}^2$ for all $v_i\in V(G)$; $v_a, v_b, v_c$ be three distinct vertices such that $(y_a-y_b)(y_a-y_c)(y_b-y_c)\neq 0$;  $\bq_0, \bq_1, \dots, \bq_k$ be motions of $(G,\bp)$ such that 
$\bq_1, \dots, \bq_k$ are linearly independent and 
\begin{equation}
\label{eq:canonical1}
\bq_i\cdot {\bf e}_{s,t}=0 \text{ for all } (s,t)\in \{(a,1), (a, 2), (a,3), (b,1), (b,2), (c,1)\}
\end{equation}
for every $i\in \{0,1,\ldots,k\}$.
Then  $Z(G,\bp)=Z_0(G,\bp)\oplus \langle \bq_1, \dots, \bq_k\rangle$ and $\bq_0\in \langle \bq_1, \dots, \bq_k\rangle$.
\end{lemma}
\begin{proof}
Suppose $\sum_{i=1}^k \mu_i \bq_i =\sum_{j=1}^6 \lambda_j\bq^*_j$ for some  $\mu_i, \lambda_j\in \mathbb{R}$. Then (\ref{eq:canonical1}) gives
\[
\sum_{j=1}^6 \lambda_j(\bq^*_j \cdot {\bf e}_{s,t})=0 \text{ for every } (s,t)\in \{(a,1), (a, 2), (a,3), (b,1), (b,2), (c,1)\}.
\] By the definition of $q_i^*$ in (\ref{eq:rigid_motions1}) and (\ref{eq:rigid_motions2}), this system of equations can be written as 
\[
\begin{pmatrix}
1 & 0 & 0 & y_a & 0 & y_a^2 \\ 
0 & 1 & 0 & -x_a & -y_a & -2x_ay_a \\ 
0 & 0 & 1 & 0 & x_a & x_a^2 \\
1 & 0 & 0 & y_b & 0 & y_b^2 \\ 
0 & 1 & 0 & -x_b & -y_b & -2x_by_b \\
1 & 0 & 0 & y_c & 0 & y_c^2 
\end{pmatrix}
\begin{pmatrix}
\lambda_1 \\
\lambda_2 \\
\lambda_3 \\
\lambda_4 \\
\lambda_5 \\
\lambda_6 
\end{pmatrix}
=
\begin{pmatrix}
0 \\
0 \\
0 \\
0 \\
0 \\
0
\end{pmatrix}.
\]
Let $Q$ denote the matrix of coefficients of this system. 
Then $\det Q=(y_a-y_b)^2(y_c-y_a)(y_b-y_c)\neq 0$.
Hence $Q$ is non-singular, and   $\lambda_j=0$ for all $1\leq j\leq 6$. Thus $Z_0(G,\bp)\cap \langle \bq_1, \dots, \bq_k\rangle=\{0\}$. The fact that $(G,\bp)$ is a $k$-dof framework now gives $Z(G,\bp)=Z_0(G,\bp)\oplus \langle \bq_1, \dots, \bq_k\rangle$.

Since $\bq_0$ is a motion, we have
$\bq_0=\sum_{i=1}^k \mu_i' \bq_i +\sum_{j=1}^6 \lambda_j'\bq^*_j$ for some  $\mu_i', \lambda_j'\in \mathbb{R}$.
We can  use (\ref{eq:canonical1}) to obtain 
\[
\sum_{j=1}^6 \lambda_j'(\bq^*_j \cdot {\bf e}_{s,t})=0 \text{ for every } (s,t)\in \{(a,1), (a, 2), (a,3), (b,1), (b,2), (c,1)\}.
\] The same argument as in the previous paragraph now gives 
$\lambda_j'=0$ for all $1\leq j\leq 6$. 
In other words, $\bq_0$ is a linear combination of $\bq_1, \dots, \bq_k$.
\end{proof}

Suppose $v_a, v_b, v_c$ are distinct vertices in a framework $(G,\bp)$. 
We define the {\em extended $C_2^1$-cofactor matrix} $\tilde{C}(G,\bp)$ (with respect to  $(v_a, v_b, v_c)$) 
to be the matrix of size $3|V|\times (|E|+6)$ obtained from $C(G,\bp)$ by adding the six rows,
${\bf e}_{a,1}, {\bf e}_{a,2}, {\bf e}_{a,3}, {\bf e}_{b,1}, {\bf e}_{b,2}, {\bf e}_{c,1}$.

For example, for $(K_4,\bp)$ with $\bp(v_1)=(0,0)$, $\bp(v_2)=(1,0)$, $\bp(v_3)=(0,1)$ and $\bp(v_4)=(-1,-1)$, 
the extended $C_2^1$-cofactor matrix with respect to $(v_1,v_2,v_3)$ is
\[
\kbordermatrix{
 & & v_1 &  & & & v_2 & & & & v_3 & & & & v_4 & \\
 & 1 & 0 & 0 & & 0 & 0 & 0 && 0 & 0 & 0 && 0 & 0 & 0 \\ 
 & 0 & 1 & 0 & & 0 & 0 & 0 && 0 & 0 & 0 && 0 & 0 & 0 \\ 
 & 0 & 0 & 1 & & 0 & 0 & 0 && 0 & 0 & 0 && 0 & 0 & 0 \\ 
 & 0 & 0 & 0 & & 1 & 0 & 0 && 0 & 0 & 0 && 0 & 0 & 0 \\ 
 & 0 & 0 & 0 & & 0 & 1 & 0 && 0 & 0 & 0 && 0 & 0 & 0 \\ 
 & 0 & 0 & 0 & & 0 & 0 & 0 && 1 & 0 & 0 && 0 & 0 & 0 \\ 
v_1v_2 & 1 & 0 & 0 & & -1 & 0 & 0 & & 0 & 0 & 0 & & 0 & 0 & 0 \\ 
v_1v_3 & 0 & 0 & 1 & & 0 & 0 & 0 & & 0 & 0 & -1 & & 0 & 0 & 0 \\ 
v_1v_4 & 1 & 1 & 1 & & 0 & 0 & 0 & & 0 & 0 & 0 & & -1 & -1 & -1 \\
v_2v_3 & 0 & 0 & 0 & & 1 & -1 & 1 & & -1 & 1 & -1 & & 0 & 0 & 0 \\
v_2v_4 & 0 & 0 & 0 & & 4 & 2 & 1 & & 0 & 0 & 0 & & -4 & -2 & -1 \\
v_3v_4 & 0 & 0 & 0 & & 0 & 0 & 0 & & 1 & 2 & 4 & & -1 & -2 & -4 
}.
\]

\medskip

\begin{lemma}
\label{lem:1dof_motion1}
Let $(G,\bp)$ be a framework with $|E(G)|=3|V(G)|-(6+k)$, and denote $\bp(v_i)=(x_i,y_i)$ for each $v_i\in V$. Let $\tilde{C}(G,\bp)$ be the
extended $C_2^1$-cofactor matrix  with respect to three vertices $(v_a, v_b, v_c)$ and suppose that 
$(y_a-y_b)(y_a-y_c)(y_b-y_c)\neq 0$.
Then $\tilde{C}(G,\bp)$ is row independent if and only if $\dim Z(G,\bp)=6+k$.
\end{lemma}
\begin{proof}
If $\tilde{C}(G,\bp)$ is row independent, then $C(G,\bp)$ is row independent and hence $$\dim Z(G,\bp)=3|V(G)|-|E(G)|=6+k.$$
To see the converse, suppose $\dim Z(G,\bp)=6+k$ but $\tilde{C}(G,\bp)$ is row dependent.
We can choose $k+1$ linearly independent $C_2^1$-motions $\bq_0, \bq_1,\dots, \bq_k$ of $(G,\bp)$ from the kernel of $\tilde{C}(G,\bp)$. 
Then each $\bq_i$ 
satisfies (\ref{eq:canonical1}). Hence, by Lemma~\ref{lem:canonical}, $\bq_0\in\langle \bq_1,\dots, \bq_k\rangle$, which is a contradiction.
\end{proof}

\subsection{Roadmap for  the proof of Theorem~\ref{thm:cofactor}}
The proof of Theorem~\ref{thm:cofactor} is structured as follows:
{\footnotesize
\[
\text{Theorem~\ref{thm:cofactor} $\overset{\mathrm{Sec~\ref{subsec:K_d matroid}}}\Longleftarrow$
Theorem~\ref{thm:K_d+2matroid} $\overset{\mathrm{Sec~\ref{subsec:proof_of_thm12}}}\Longleftarrow$
Theorem~\ref{thm:doubleVconj} $\overset{\mathrm{Sec~\ref{sec:doubleV}}}\Longleftarrow$
Theorem~\ref{thm:doubleV} $\overset{\mathrm{Sec~\ref{sec:bad}}}\Longleftarrow$
Theorem~\ref{thm:well-behaved0}$\overset{\mathrm{Sec~\ref{sec:well-behaved}}}\Longleftarrow$
Theorem~\ref{thm:well-behaved}}.
\]
}
Below we briefly explain each theorem and step.

\paragraph{Section~\ref{subsec:K_d matroid}:} 
It is known that the edge set of $K_5$ is a circuit  in any abstract 3-rigidity matroid (see, e.g., \cite{GSS93} or Theorem~\ref{thm:hang}).
A matroid satisfying this circuit property is called a $K_5$-matroid.
The class of $K_5$-matroids is slightly larger  than the set of abstract 3-rigidity matroids.
Theorem~\ref{thm:K_d+2matroid} states that the generic $C_2^1$-cofactor matroid ${\cal C}_2^1(K_n)$ is the unique maximal matroid in the poset of all  $K_5$-matroids on $E(K_n)$.
Thus Theorem~\ref{thm:K_d+2matroid} is a stronger version of Theorem~\ref{thm:cofactor}.

\paragraph{Section~\ref{subsec:proof_of_thm12}:}
 Graver~\cite{G91} proved that the generic 2-dimensional rigidity matroid is the unique maximal abstract 2-rigidity matroid based on that fact that 
 any base of the generic 2-dimensional rigidity matroid  can be constructed from $K_3$ by a sequence of two simple graph operations, called the Henneberg constructions. 
 Following the same approach as that of Graver~\cite{G91}, 
 we show that Theorem~\ref{thm:K_d+2matroid} will follow from an inductive construction of the bases of the  ${\cal C}_2^1$-cofactor matroid. 
For this purpose, Whiteley~\cite{Wsurvey} already verified that  various graph operations preserve $C_2^1$-independence. 
However, there remained one operation, double V-replacement, that Whiteley could not show preserved independence. (See Figure~\ref{fig:doubleV} for an example of double V-replacement.).
We confirm that it does preserve independence in Theorem~\ref{thm:doubleVconj}. The remainder of the paper is devoted to proving this result.

\paragraph{Section~\ref{sec:doubleV}:}
By combining existing results from~\cite{Wsurvey} and a simple combinatorial argument in  Section~\ref{sec:doubleV}, we prove that Theorem~\ref{thm:doubleVconj} holds in all but  one special case.
This remaining special case is when ${\rm cl}(E(G-v_0))\cap K(N_G(v_0))$ forms a star on five vertices as shown in Figure~\ref{fig:type_star},
where $v_0$ denotes the vertex created by the double $V$-replacement.
We formulate a new result, Theorem~\ref{thm:doubleV}, to  deal with  this remaining special case.

\paragraph{Section~\ref{sec:bad}:}

 In Section~\ref{sec:bad}, we first show that, if the statement of  Theorem~\ref{thm:doubleV} fails, then the framework created by  the double V-replacement operation has a very special $C_2^1$-motion.
 Theorem~\ref{thm:well-behaved0} states that such a motion cannot exist in any generic $1$-dof framework.

 \paragraph{Section~\ref{sec:well-behaved}:}
 We prove Theorem~\ref{thm:well-behaved0}. Our inductive proof requires a more general inductive statement, Theorem~\ref{thm:well-behaved}, which  concerns frameworks with at most two degrees of freedom. The proof is rather long so we first outline the main ideas in Section~\ref{subsec:outline}.

\section{Inductive Construction and Proof of Theorem~\ref{thm:cofactor}}
\label{sec:inductive}
\subsection{Inductive Construction}\label{subsec:inductive}
Our proof of Theorem~\ref{thm:cofactor} 
uses the following  3-dimensional versions of standard graph operations from rigidity theory. 
Given a graph $H$:
\begin{itemize}
\item the \emph{$0$-extension} operation adds a new vertex $v$ and three new edges from $v$ to vertices in $H$; 
\item the {\em $1$-extension} operation chooses an edge $e$ of $H$ and adds a new vertex $v$ and four new edges from $v$ to vertices in $H-e$ with the proviso that two of the new edges join $v$ to the end-vertices of $e$;
\item the {\em X-replacement} operation chooses two non-adjacent edges $e$ and $f$ of $H$ and adds a new vertex $v$ and five new edges from $v$ to vertices in $H-\{e,f\}$ with the proviso that four of the new edges join $v$ to the end-vertices of $e$ and $f$;
\item the {\em  V-replacement} operation chooses two adjacent edges $e$ and $f$ of $H$ and adds a new vertex $v$ and five new edges from $v$ to vertices in $H-\{e,f\}$ with the proviso that three of the new edges join $v$ to the end-vertices of $e$ and $f$;
\item the {\em  vertex splitting} operation chooses a vertex $u$ of $H$ and pairwise disjoint sets  $U_1,U_2,U_3$ with $U_1\cup U_2\cup U_3=N_H(u)$ and $|U_2|=2$, deletes all edges from $u$ to $U_3$, and then adds a new vertex $v$ and $|U_3|+3$ new edges from $v$ to each vertex in $U_2\cup U_3\cup \{u\}$.
\end{itemize}
Whiteley~\cite[Lemmas 10.1.5 and 10.2.1, Theorems 10.2.7 and 10.3.1]{Wsurvey} showed that all but one of these operations preserve  $C_2^1$-independence:
\begin{lemma}
\label{lem:cofactor_ind}
The  $0$-extension, $1$-extension, X-replacement, and vertex splitting operations  preserve $C_{2}^{1}$-independence.
\end{lemma}

In general V-replacement may not preserve $C_2^1$-independence, but there is an important special case when it does.

\begin{lemma}
\label{lem:V} Suppose that $v_1,v_2,v_3,v_4,v_5$ are vertices of a $C_2^1$-independent graph $H=(V,E)$, $e=v_1v_2,f=v_1v_3$ are edges of $H$ and  $v_1v_4,v_1v_5$ belong to the closure of $E-e-f$ in the $C^1_{2}$-cofactor matroid on $K(V)$. Then the graph $G$ obtained from $H-e-f$ by adding a new vertex $v$ and new edges $vv_1,vv_2,vv_3,vv_4,vv_5$ is $C_2^1$-independent.
\end{lemma}
\begin{proof}
 Let $\cl(.)$ denote the closure operator in the $C^1_{2}$-cofactor matroid on $K(V)$. Since  $v_1v_4,v_1v_5\in \cl(E-e-f)$,  we may choose a base $B$ of $\cl(E-e-f)$ with $v_1v_4,v_1v_5\in B$. Then the graph $H'=(B,B+e+f)$ is $C_2^1$-independent and we can now apply the vertex splitting operation at $v_1$ to deduce that the graph $G'$ obtained from $H'-e-f$ by adding a new vertex $v$ and new edges $vv_1,vv_2,vv_3,vv_4,vv_5$ is $C_2^1$-independent, see Figure~\ref{fig:vertexsplitting}. This and the fact that $\cl(B)=\cl(E-e-f)$ imply that $G$ is $C_2^1$-independent.
 \end{proof}
 
 \noindent
Lemma \ref{lem:V} will be used several times in our proofs 
 
Whiteley~\cite{Wsurvey} conjectured that another special case of V-replacement preserves $C_2^1$-independence: if $H\cup \{e_1,e_2\}$ and $H\cup \{e_1',e_2'\}$ are both $C_2^1$-independent for two pairs of adjacent edges $e_1,e_2$ and $e_1',e_2'$ with the property that the common endvertex of $e_1,e_2$ is distinct from that of $e_1',e_2'$, then the graph $G$ obtained by adding a new vertex $v_0$ of degree five to $H$ in such a way that the endvertices of $e_1,e_2,e_1',e_2'$ are all neighbours of $v_0$, is $C_2^1$-independent. See Figure~\ref{fig:doubleV}. This operation is referred to as  {\em double V-replacement} since $G$ can be constructed from both $G-v_0+e_1+e_2$ and $G-v_0+e_1'+e_2'$ by a V-replacement.
%
We will verify Whiteley's conjecture:

\begin{figure}[t]
\centering
\includegraphics[scale=0.8]{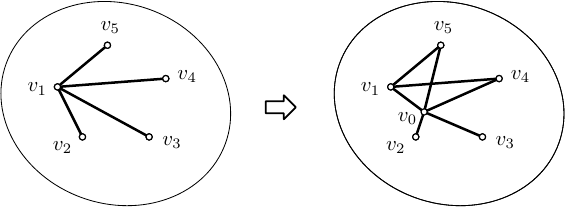}
\caption{An extension operation which can be considered as a special case of both V-replacement and vertex splitting. }
\label{fig:vertexsplitting}
\end{figure}

\begin{figure}[t]
\centering
\includegraphics[scale=0.8]{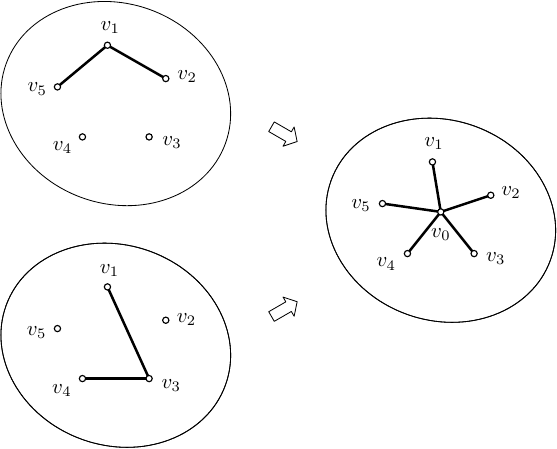}
\caption{An example of double V-replacement.}
\label{fig:doubleV}
\end{figure}

\begin{theorem}\label{thm:doubleVconj}
The  double V-replacement operation  preserves $C_{2}^{1}$-independence.

\end{theorem}

Lemma \ref{lem:cofactor_ind} and Theorem~\ref{thm:doubleVconj} imply that we can construct all minimally $C_2^1$-rigid graphs from copies of $K_4$ with a `Henneberg tree construction' using the 0-, 1-, X- and double V-extension operations. We refer the reader to \cite{TW85,Wsurvey} for more details.

The proof of Theorem \ref{thm:doubleVconj} will be spread over  Sections \ref{sec:doubleV}, \ref{sec:bad} and \ref{sec:well-behaved} of this paper. To motivate this rather long and technical proof we show in the remainder of this section how (a generalisation of) our main result, Theorem ~\ref{thm:cofactor}, can be deduced easily from Lemma \ref{lem:cofactor_ind} and Theorem \ref{thm:doubleVconj}. 

\subsection{\boldmath $K_t$-matroids
}
\label{subsec:K_d matroid}
Let $n$ and $t$ be positive integers with $t\leq n$.
A matroid ${\cal M}$ on the edge set of a complete graph $K_n$ is said to be a {\em $K_t$-matroid} if  the edge set of every copy of $K_t$ in $K_n$ is a circuit in ${\cal M}$.  The following simple characterization of abstract $d$-rigidity matroids due to Nguyen{\cite[Theorem 2.2]{N10}} 
implies that every abstract $d$-rigidity matroid is a $K_{d+2}$-matroid.
\begin{theorem}
\label{thm:hang}
Let $n,d$ be positive integers with $n\geq d+2$ and ${\cal M}$ be a matroid on 
$E(K_n)$. 
Then ${\cal M}$ is an abstract $d$-rigidity matroid if and only if 
${\cal M}$ is a $K_{d+2}$-matroid with rank $dn-{d+1\choose 2}$.
\end{theorem}
We will prove a stronger version of Theorem~\ref{thm:cofactor} by showing that the $C^1_2$-cofactor matroid is the unique maximal matroid in the family of  $K_5$-matroids. (This strengthening will be needed in our companion paper \cite{CJT1}.)

Our next lemma shows that $dn-{d+1\choose 2}$ is an upper bound on the rank of any $K_{d+2}$-matroid on $E(K_n)$ when $n$ is sufficiently large. 

\begin{lemma}\label{lem:rank_upper_bound} Suppose that ${\cal M}$ is a $K_{d+2}$-matroid on $E(K_n)$ for $n\geq d+2$. Then ${\cal M}$ has rank at most  $dn-{{d+1}\choose{2}}$.
\end{lemma}
\begin{proof}
We proceed  by induction on $n$.
The claim is trivial when $n=d+2$.

Suppose $n>d+2$, and denote   the vertex set of $K_n$ by $\{v_1,\dots, v_n\}$. 
Let $K$ be the edge set of the complete graph on $\{v_1,\dots, v_{n-1}\}$.
 Since ${\cal M}|_{K}$ is a $K_{d+2}$-matroid, the rank of $K$ in ${\cal M}$ is at most $d(n-1)-{d+1\choose 2}$ by induction.
 Let $K'=K+\{v_1v_n,v_2v_n,\dots, v_dv_n\}$. We show $K'$ spans $E(K_n)$.
 
For each $i=d+1,\dots, n-1$, let $C_i$ be the edge set of the complete graph on $\{v_1, v_2,\dots, v_{d-1}, v_d, v_n,v_i\}$.
Note that $K'$ contains all edges of $C_i$ except $v_iv_n$. Since ${\cal M}$ is a $K_{d+2}$-matroid, $v_iv_n$ is spanned by $K'$. 
Hence $K'$ spans $E(K_n)$, and the rank of ${\cal M}$ is at most $|K'| =|K|+d=dn-{d+1\choose 2}$.
\end{proof}

Theorem~\ref{thm:hang} and Lemma~\ref{lem:rank_upper_bound} imply that an abstract $d$-rigidity matroid is a $K_{d+2}$-matroid which attains the maximum possible rank. 

\subsection{\boldmath Maximality of the generic $C_2^1$-cofactor matroid} \label{subsec:proof_of_thm12}
We show that 
${\cal C}_{2,n}^1$ is the unique maximal element, with respect to the weak order, in the set of all $K_5$-matroids on $E(K_n)$. Since every abstract $3$-rigidity matroid is a $K_{5}$-matroid, this immediately implies 
%
%
Theorem~\ref{thm:cofactor}.

 \begin{theorem}\label{thm:K_d+2matroid}
  The generic $C_2^1$-cofactor matroid ${\cal C}_{2,n}^{1}$ is the unique maximal $K_5$-matroid on $E(K_n)$ for all $n\geq 1$.
  \end{theorem}

\begin{proof}
The theorem is trivially true when $n\leq 5$ so we may assume that $n\geq 6$.
   Let  ${\cal M}$ be an arbitrary $K_5$-matroid on $E(K_n)$ and $F\subseteq E(K_n)$ be an independent set in $\cal M$. We will show that $F$ is independent in ${\cal C}_{2,n}^{1}$ by induction on $|F|$.  
 We will abuse notation throughout this proof and use the same letter for both a subgraph of $K_n$ and its edge set. 

 Since ${\cal M}$ is a $K_5$-matroid,  Lemma \ref{lem:rank_upper_bound} implies that $|F|=r(F)\leq 3|V(F)|-6$, and hence  
 %
 there exists a vertex $v_0$ of degree at most five in $F$. 
 Let $N_F(v_0)=\{v_1,\dots, v_k\}$ (where $k$ is the degree of $v_0$) and
  let $K$ be the complete graph on $N_F(v_0)$.
  
  Suppose $d_F(v_0)\leq 3$. Since $F-v_0$ is independent in ${\cal M}$, $F-v_0$ is independent in ${\cal C}^1_{2,n}$ by induction.  Then $F$ is independent in ${\cal C}^1_{2,n}$ by the $0$-extension property in Lemma~\ref{lem:cofactor_ind}.
  Hence we may suppose $d_F(v_0)\in\{4,5\}$.
  
 \begin{claim}\label{claim:spline1} 
 ${\rm cl}_{{\cal M}}(F-v_0)\cap K$
does not contain a copy of $K_4$.
  \end{claim}
  \begin{proof}
  Suppose, for a contradiction, that  ${\rm cl}_{{\cal M}}(F-v_0)$ contains a copy of $K_4$.
  We may assume that  ${\rm cl}_{{\cal M}}(F-v_0)$ contains $K(v_1, v_2, v_3, v_4)$. 
  Let $e_i =v_0v_i$ for $1\leq i\leq 4$.
  Since ${\cal M}$ is a $K_5$-matroid, $e_4\in \cl_{\cal M}(F-v_0+e_1+e_2+e_3)$. This contradicts the fact that $F$ is independent in ${\cal M}$.
%
  \end{proof}
 
%
%

 Suppose that  $d_F(v_0)=4$.  Claim \ref{claim:spline1} 
 implies that $F-v_0+v_1v_2$ is independent in ${\cal M}$ for two non-adjacent neighbors $v_1v_2$ of $v_0$ in $F$.
 We can now apply induction to deduce that $F-v_0+v_1v_2$ is independent in ${\cal C}^1_{2,n}$ and then use the 1-extension property in Lemma~\ref{lem:cofactor_ind} to deduce that $F$ is independent in ${\cal C}^1_{2,n}$.
Hence we may assume that $d_F(v_0)=5$.

\begin{claim}\label{claim:spline2-0}
$K\not\subseteq {\rm cl}_{\MM}(F-v_0+e)$ for all $e\in K$.
  \end{claim}
  \begin{proof} Suppose, for a contradiction, that $K\subseteq {\rm cl}_{\MM}(F-v_0+e)$ for some $e\in K$. 
  Then  $F-v_0+e$ is independent in ${\cal M}$ by Claim \ref{claim:spline1}.
  Combining this with the fact  ${\cal M}$ is a $K_5$-matroid, we obtain
  $r_{\cal M}(F)\leq r_{\cal M}(F-v_0+e)+3$. 
  On the other hand, the fact that $F$ is independent in ${\cal M}$ tells us that $r_{\cal M}(F)= r_{\cal M}(F-v_0)+5= r_{\cal M}(F-v_0+e)+4$.
  This is a contradiction.
  \end{proof}

Claims \ref{claim:spline1} and \ref{claim:spline2-0} imply that we can choose $e_1,e_2\in K\setminus F$ such that $(F-v_0)+e_1+e_2$ is independent in $\MM$. The induction hypothesis now tells us that $(F-v_0)+e_1+e_2$ is independent in ${\cal C}^1_{2,n}$.
If $e_1,e_2$ have no common endvertices, then  $F$ can be obtained from $F-v_0+e_1+e_2$ by X-replacement and $F$ would be independent in ${\cal C}^1_{2,n}$ by Lemma~\ref{lem:cofactor_ind}, a contradiction. Hence we may suppose that $v_1$ is a common endvertex of $e_1,e_2$.  Claim  \ref{claim:spline1} now tells us that $F-v_0+e_1'$ is independent in ${\cal M}$ for some $e_1'\in K(v_2,v_3,v_4,v_5)\setminus F$, and Claim  \ref{claim:spline2-0} in turn gives an edge $e_2'\in K\setminus F$ such that $F-v_0+e_1'+e_2'$ is independent in ${\cal M}$. Then $F-v_0+e_1'+e_2'$ is independent in ${\cal C}^1_{2,n}$ by induction. 
Since $F$ can be obtained from $F-v_0+e_1+e_2$ and $F-v_0+e_1'+e_2'$ by a double V-replacement,   Theorem \ref{thm:doubleVconj} tells us that  $F$ is independent in ${\cal C}^1_{2,n}$.
This contradiction completes the proof.
\end{proof}

\section{Double V-replacement}\label{sec:doubleV}
The remainder of the paper is dedicated to proving Theorem~\ref{thm:doubleVconj}. Since we will only be concerned with the $C_2^1$-cofactor matroid we will often suppress specific mention of this matroid. In particular we  will 
say that a set of edges $F\subseteq E(K_n)$  is {\em independent} if it is independent in ${\cal C}^1_{2,n}$  and 
use $\cl(F)$ to denote the closure of $F$ in ${\cal C}^1_{2,n}$.
We will continue to use the terms $C^1_2$-independent and $C^1_2$-rigid for graphs.

In this section, we will formulate a special case of Theorem \ref{thm:doubleVconj},  Theorem~\ref{thm:doubleV} below, 
and then show that the general result will follow from this special case. The proof of Theorem~\ref{thm:doubleV} is delayed until Sections~\ref{sec:bad} and~\ref{sec:well-behaved}. 

We first need to establish a structural result for $C^1_2$-independent graphs.

\begin{lemma}\label{lem:combinatorial}
Let $H=(V,E)$ be a $C^1_2$-independent graph, $U=\{v_1,\dots, v_5\}$ be a set of five vertices in $G$, and $K=K(U)$. 
Then at least one of the following holds:
\begin{itemize}
\item[(i)] ${\rm cl}(E)\cap K$ contains a copy of $K_4$. 
\item[(ii)] $K\subseteq {\rm cl}(E+e)$  for some edge $e\in K$.
\item[(iii)] there are two non-adjacent edges $e_1$ and $e_2$ in $K$ such that 
$H+e_1+e_2$ is $C^1_2$-independent.
\item[(iv)] there are two adjacent edges $e_1$ and $e_2$ in $K$ such that 
$H+e_1+e_2$ is $C^1_2$-independent and 
the common end-vertex of $e_1$ and $e_2$ is incident to two edges in ${\rm cl}(E)\cap K$. 
\item[(v)] ${\cl}(E)\cap K$ forms a star on five vertices, and ${\rm cl}(E+e)\cap K$ contains no copy of $K_4$ for all $e\in K$.
\end{itemize}
\end{lemma}
\begin{proof}
We assume that  (i) - (iv) do not hold and prove that (v) must hold.

Since (iii) does not hold, we have:
\begin{equation}\label{eq:combinatorial1}
\text{for all $e=v_iv_j\in K\setminus {\rm cl}(E)$, ${\rm cl}(E+e)$ contains a triangle on $U\setminus \{v_i,v_j\}$. }  
\end{equation}

  \begin{claim}\label{claim:spline2}
 For all $e\in K$, ${\rm cl}(E+e)\cap K$ contains no copy of $K_4$.
  \end{claim}
  \begin{proof}
  Since (i) does not occur, the claim follows if $e\in {\rm cl}(E)$.
  Thus we suppose that $e\notin {\rm cl}(E)$, and hence $E+e$ is independent.
  
  Relabeling if necessary, we may suppose $e=v_1v_2$.
 By (\ref{eq:combinatorial1}),
 ${\rm cl}(E+e)$ contains the triangle $K(v_3,v_4,v_5)$ on $\{v_3,v_4,v_5\}$.
 Since (ii) does not hold,  ${\rm cl}(E+e)\cap K$ does not contain two copies of $K_4$ (since the union of two distinct copies of $K_4$ on $U$ would form a $C_2^1$-rigid graph on $U$.)
We consider two cases.

\medskip
\noindent \emph{Case 1:} 
 Suppose that
 ${\rm cl}(E+e)\cap K$ contains a $K_4$ which includes $e=v_1v_2$. 
 By symmetry we may assume that this is a $K_4$ on $\{v_1, v_2, v_3, v_4\}$.
 Since ${\rm cl}(E+e)\cap K$ contains $K(v_3,v_4,v_5)$ and contains at most one $K_4$, 
 $v_1v_5, v_2v_5\notin {\rm cl}(E+e)$.
We will show that:
\begin{equation}
\label{eq:claim3}
\{v_2v_3,v_2v_4\}\subset {\rm cl}(E).
 \end{equation}

Suppose  $v_2v_3\notin {\rm cl}(E)$. 
Since   $v_2v_3\in {\rm cl}(E+e)\setminus {\rm cl}(E)$,
$E+v_2v_3$ is independent and ${\rm cl}(E+e)={\rm cl}(E+v_2v_3)$.
Since $v_1v_5 \notin {\rm cl}(E+e)$, this in turn implies $v_1v_5\notin {\rm cl}(E+v_2v_3)$. Hence $E+v_2v_3+v_1v_5$ is independent, contradicting that (iii) does not hold.
Thus $v_2v_3\in {\rm cl}(E)$. 
The same argument for $v_2v_4$ implies (\ref{eq:claim3}).

Since $v_2v_5\notin {\rm cl}(E+e)$, $E+v_1v_2+v_2v_5$ is independent. Combined with  (\ref{eq:claim3}), this contradicts the assumption that (iv) does not hold.

\medskip
\noindent \emph{Case 2:} Suppose that
 ${\rm cl}(E+e)\cap K$ contains a $K_4$ which avoids $e=v_1v_2$. 
 By symmetry we may assume that this is a $K_4$ on $\{v_1, v_3, v_4,v_5\}$. Since ${\rm cl}(E+e)\cap K$ contains at most one $K_4$, there is at most one edge from $v_2$ to $\{v_3,v_4,v_5\}$ in ${\rm cl}(E+e)\cap K$ and we may assume by symmetry that
  $v_2v_4, v_2v_5\notin {\rm cl}(E+e)$.
We will show that this case cannot occur by showing that
\begin{equation}
\label{eq:claim3.5}
\text{all edges on $\{v_1,v_3,v_4,v_5\}$  are in ${\rm cl}(E)$}
 \end{equation}
and hence that (i) holds.

Suppose $v_1v_3\notin {\rm cl}(E)$.
Then, since  $v_1v_3\in {\rm cl}(E+e)\setminus {\rm cl}(E)$,
$E+v_1v_3$ is independent and ${\rm cl}(E+v_1v_2)={\rm cl}(E+v_1v_3)$.
Since $v_2v_4\notin {\rm cl}(E+v_1v_2)$, we have  $v_2v_4\notin {\rm cl}(E+v_1v_3)$. Hence $E+v_1v_3+v_2v_4$ is independent and (iii) holds.
This contradiction implies  that $v_1v_3\in {\rm cl}(E)$. 
The same argument holds for all the other edges on $\{v_1,v_3,v_4,v_5\}$ except $v_4v_5$.
If $v_4v_5\notin {\rm cl}(E)$ then, since  $v_4v_5\in {\rm cl}(E+v_1v_2)\setminus {\rm cl}(E)$,
$E+v_4v_5$ is independent and ${\rm cl}(E+v_1v_2)={\rm cl}(E+v_4v_5)$.
This in turn implies $v_2v_4\notin {\rm cl}(E+v_4v_5)$,
and $E+v_2v_4+v_4v_5$ is independent. 
This and the fact $v_1v_4,v_3v_4\in {\rm cl}(E)$ contradicts the assumption that (iv) does not hold.
Hence $v_4v_5\notin {\rm cl}(E)$ and (\ref{eq:claim3.5}) holds.
%
\end{proof}

Claim \ref{claim:spline2} implies that the second part of condition (v) in the statement holds. It remains to show that ${\rm cl}(E)\cap K$ forms a star on $U$.
This will follow by combining Claims~\ref{claim:spline5} and \ref{claim:spline6}, below.
We first derive an auxiliary claim.
 \begin{claim}
 \label{claim:spline4}
 Suppose that $E+v_1v_2$ is independent.
 Then ${\rm cl}(E)$ has at least two edges on $\{v_3, v_4, v_5\}$.
 \end{claim}
 \begin{proof}
 Recall that ${\rm cl}(E+v_1v_2)$ contains a complete graph  on $\{v_3,v_4,v_5\}$ by (\ref{eq:combinatorial1}).
 
 We first consider the complete graph on $\{v_1, v_2, v_3, v_4\}$.
 By Claim~\ref{claim:spline2}, we may assume without loss of generality that $v_1v_3\not\in {\rm cl}(E+v_1v_2)$.
 If $v_4v_5\notin {\rm cl}(E)$, then $v_4v_5\in {\rm cl}(E+v_1v_2)\setminus {\rm cl}(E)$, and so
 ${\rm cl}(E+v_4v_5) = {\rm cl}(E+v_1v_2)$. Hence $v_1v_3\notin {\rm cl}(E+v_4v_5)$.
 This gives a contradiction as  $E+v_4v_5+v_1v_3$ would be independent  and (iii) would hold.
 Thus we obtain $v_4v_5\in {\rm cl}(E)$.

We next  consider the complete graph on $\{v_1,v_2,v_4,v_5\}$.
 By Claim~\ref{claim:spline2},   we may assume without loss of generality that either $v_1v_4\notin {\rm cl}(E+v_1v_2)$ or $v_2v_4\notin {\rm cl}(E+v_1v_2)$.
  If $v_3v_5\notin {\rm cl}(E)$, then $v_3v_5\in {\rm cl}(E+v_1v_2)\setminus {\rm cl}(E)$, and so ${\rm cl}(E+v_1v_2)={\rm cl}(E+v_3v_5)$. Hence $v_1v_4\notin {\rm cl}(E+v_3v_5)$ or $v_2v_4\notin {\rm cl}(E+v_3v_5)$ respectively.
 This gives a contradiction as (iii) would hold.
 Thus we obtain $v_3v_5\in {\rm cl}(E)$.
 \end{proof}

 \begin{claim}
 \label{claim:spline5}
 For all $i$ with $1\leq i\leq 5$, $d_{{\rm cl}(E)\cap K}(v_i)\geq 1$.
 And if $d_{{\rm cl}(E)\cap K}(v_i)=1$, then the vertex $v_j$ adjacent to $v_i$ in ${\rm cl}(E)\cap K$ satisfies $d_{{\rm cl}(E)\cap K}(v_j)=4$.
 \end{claim}
 \begin{proof}
 Suppose that $d_{{\rm cl}(E)\cap K}(v_5)=0$.
 Since ${\rm cl}(E)$ has no copy of $K_4$, we may assume $v_1v_2\not\in {\rm cl}(E)$.
 Then $E+v_1v_2$ is independent, and ${\rm cl}(E)$ has at most one edge on $\{v_3, v_4, v_5\}$ since  $d_{{\rm cl}(E)\cap K}(v_5)=0$.
 This contradicts  Claim~\ref{claim:spline4}.

 Suppose that $d_{{\rm cl}(E)\cap K}(v_5)=1$.  We may assume without loss of generality that $v_5$ is adjacent to $v_1$ in ${\rm cl}(E)$.
 Suppose further that ${\rm cl}(E)$ does not contain $v_1v_2$.
 Then $E+v_1v_2$ is independent. On the other hand,  ${\rm cl}(E)$ has at most one edge on $\{v_3, v_4, v_5\}$ because $v_5$ is incident only to $v_1$ in ${\rm cl}(E)\cap K$.
 This again contradicts  Claim~\ref{claim:spline4}.
 \end{proof}

 \begin{claim}
 \label{claim:spline6}
 ${\rm cl}(E)\cap K$ is cycle free.
 \end{claim}
 \begin{proof}
 Suppose that ${\rm cl}(E)\cap K$ contains a cycle of length five.
 By Claim \ref{claim:spline2} we may choose two chords $e_1, e_2$ of this cycle such that $E+e_1+e_2$ is independent.
 This would contradict the assumption that (iii) and (iv) do not hold since either $e_1$ and $e_2$ are non-adjacent or  their common end-vertex has degree two in ${\rm cl}(E)\cap K$.

 Suppose that ${\rm cl}(E)\cap K$ contains a cycle of length four, say $v_1v_2v_3v_4v_1$.
 Since ${\rm cl}(E)\cap K$ has no $K_4$, we may suppose that $v_1v_3\notin {\rm cl}(E)$. Then Claim~\ref{claim:spline2} implies that $E+v_1v_3+v_2v_4$ is independent and (iii) holds, a contradiction. 

 Suppose that ${\rm cl}(E)\cap K$ contains a triangle, say $K(v_1,v_2,v_3)$.
 Claim~\ref{claim:spline5} and the fact that ${\rm cl}(E)\cap K$ contains no cycle of length four tells us that there is exactly one edge in ${\rm cl}(E)\cap K$ from each of $v_4,v_5$  to $\{v_1,v_2,v_3\}$, and that both of these edges have a common end-vertex, say $v_1$. Hence $E+v_2v_4$ is independent. 
 We can now use (\ref{eq:combinatorial1}) to deduce that $v_3v_5\in \clo(E+v_2v_4)$.  Claim~\ref{claim:spline2} now implies that $v_2v_5\not\in \clo(E+v_2v_4)$ and hence $E+v_2v_4+v_2v_5$ is independent. This gives a contradiction since $v_2$ has degree two in ${\rm cl}(E)\cap K$ and hence (iv) holds.
 \end{proof}

 Claims \ref{claim:spline5} and \ref{claim:spline6} imply that
 ${\rm cl}(E)\cap K$ forms a star on $\{v_1,v_2,\dots, v_5\}$ and hence (v) holds.
 \end{proof}

\medskip

Let $G=(V,E)$ be a graph.
A vertex $v_0$ in $G$ is said to be of {\em type $(\star)$} if 
\begin{itemize}
\item $v_0$ has degree five with  $N_G(v_0)=\{v_1,v_2,v_3,v_4,v_5\}$,  
\item $G-v_0+v_1v_2+v_1v_3$ and $G-v_0+v_1v_3+v_3v_4$ are both $C_2^1$-rigid, and
\item ${\rm cl}(E(G-v_0))\cap K(N_G(v_0))$ forms a star on five vertices centered on $v_5$,
\end{itemize}
See Figure~\ref{fig:type_star}.

\begin{figure}[t]
\centering
\includegraphics[scale=0.8]{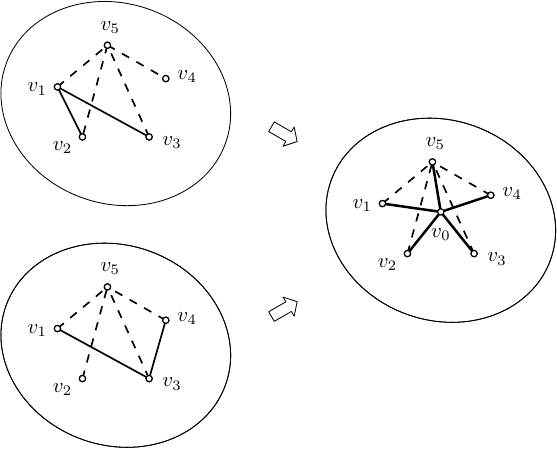}
\caption{The double V-replacement that creates a vertex $v_0$ of type ($\star$). 
The dashed lines represent the edges in ${\rm cl}(E(G-v_0))\cap K(N_G(v_0))$, which form a star on five vertices centered on $v_5$.}
\label{fig:type_star}
\end{figure}

The following theorem implies that double V-replacement preserves 
minimal $C_2^1$-rigidity in the special case when the new vertex is of type $(\star)$.

\begin{theorem}\label{thm:doubleV}
Let $G=(V,E)$ be a graph with $|E|=3|V|-6$ 
and suppose that $G$ has a vertex  $v_0$  of type $(\star)$.
Then $G$ is minimally $C_2^1$-rigid.
\end{theorem}

As noted above, we delay the proof of this theorem to Sections~\ref{sec:bad} and~\ref{sec:well-behaved} and instead use it to  verify Theorem \ref{thm:doubleVconj}.

%

\paragraph{Proof of Theorem \ref{thm:doubleVconj}} Let $G=(V,E)$ be a graph that has a vertex  $v_0$ of degree five,
and let $\{e_1,e_2\}$ and  $\{e_1',e_2'\}$ be two pairs of adjacent edges on $N_G(v_0)$.
Suppose that  $G-v_0+e_1+e_2$ and $G-v_0+e_1'+e_2'$ are both $C^1_2$-independent,
and the common endvertex of $e_1$ and $e_2$ is distinct from that of $e_1'$ and $e_2'$.
Let $N_G(v_0)=\{v_1,\dots, v_5\}$.
We need to show that $G$ is $C_2^1$-independent.

We apply Lemma~\ref{lem:combinatorial} to $H:=G-v_0$ and $U:=N_G(v_0)$.
(Note that $H$ is $C^1_2$-independent.)
Neither (i) nor (ii) of Lemma~\ref{lem:combinatorial}  hold since otherwise $G-v_0+e_1+e_2$ or $G-v_0+e_1'+e_2'$ would be dependent.
If (iii)  holds, then we can construct $G$ from a $C^1_2$-independent graph by X-replacement, and hence $G$ is $C^1_2$-independent by Lemma~\ref{lem:cofactor_ind}.
Similarly, if  (iv) holds, then $G$ can be  constructed from a $C^1_2$-independent graph by a V-replacement satisfying the hypotheses of 
Lemma~\ref{lem:V} so is $C^1_2$-independent.
Hence we may assume (v) of Lemma~\ref{lem:combinatorial} holds,
and, without loss of generality, that  ${\rm cl}(E(H))\cap K(U)$ forms a star whose center is $v_5$.
Then $E(H)+v_1v_2$ and $E(H)+v_3v_4$ are both independent.
If $v_3v_4\notin {\rm cl}(E(H)+v_1v_2)$, then $G$ can be obtained from the $C^1_2$-independent graph $H+v_1v_2+v_3v_4$ by X-replacement.
Hence we may assume that $v_3v_4\in {\rm cl}(E(H)+v_1v_2)$, 
and thus ${\rm cl}(E(H)+v_1v_2)={\rm cl}(E(H)+v_3v_4)$ by the independence of $H+v_3v_4$.
Since (v) of Lemma~\ref{lem:combinatorial} holds, 
we may further suppose that $v_1v_3\notin {\rm cl}(E(H)+v_1v_2)$.
Therefore both $H+v_1v_2+v_1v_3$ and $H+v_3v_4+v_1v_3$ are $C^1_2$-independent.

Let $k=3|V|-6-|E|$. We complete the proof by induction on $k$.
If $k=0$ then $H+v_1v_2+v_1v_3$ and $H+v_3v_4+v_1v_3$ are both minimally $C_2^1$-rigid,
and $v_0$ is of type $(\star)$.
Theorem~\ref{thm:doubleV} now implies that $G$ is minimally $C_2^1$-rigid.

For the induction step, we assume that $k>0$.
Since  ${\rm cl}(E(H)+v_1v_2)={\rm cl}(E(H)+v_3v_4)$,
we have  ${\rm cl}(E(H)+v_1v_2+v_1v_3)={\rm cl}(E(H)+v_3v_4+v_1v_3)$.
As $H+v_1v_2+v_1v_3$ is a $k$-dof graph with $k>0$, there is an edge $f\in K(V-v_0)$ such that 
$H+v_1v_2+v_1v_3+f$ and $H+v_3v_4+v_1v_3+f$ are both $C^1_2$-independent.
Applying the induction hypothesis to $H+f$, we find that
$G+f$ is $C^1_2$-independent.
Therefore $G$ is $C^1_2$-independent.
\qed

\section{\boldmath Bad $C_2^1$-Motions} \label{sec:bad}

In this section, we show that if Theorem~\ref{thm:doubleV} does not hold for some graph $G$, then every generic realization of $G$ in $\mathbb{R}^2$ has a special kind of motion, and  derive some properties of such `bad' motions.
We use these properties in Section~\ref{sec:well-behaved}, to prove that no 
generic framework can have a bad motion.

The following notation will be used throughout the remainder of this paper.
Suppose $p_i=(x_i, y_i)$ is a point in the plane for $i=0,1,2,3$. 
For any three of these points, we let 
$$\Delta(p_i, p_j, p_k)=\begin{vmatrix} x_i & y_i & 1 \\ x_j & y_j & 1 \\ x_k & y_k & 1 \end{vmatrix},$$
which is twice the signed area of the triangle defined by the three points.
The formula for the determinant of a Vandermonde matrix implies that
\begin{align}
\left|
\begin{array}{c}
D(p_0, p_1) \\ D(p_0 , p_2) \\ D(p_0, p_3)
\end{array}\right|
&=-\left|
\begin{array}{ccc}
x_0 & y_0 & 1 \\
x_1 & y_1 & 1 \\
x_2 & y_2 & 1
\end{array}\right|
\left|
\begin{array}{ccc}
x_0 & y_0 & 1 \\
x_2 & y_2 & 1 \\
x_3 & y_3 & 1
\end{array}\right|
\left|
\begin{array}{ccc}
x_0 & y_0 & 1 \\
x_3 & y_3 & 1 \\
x_1 & y_1 & 1
\end{array}\right| \\
&=-\Delta(p_0, p_1, p_2) \Delta(p_0, p_2, p_3) \Delta(p_0, p_3, p_1). \label{eq:Vandermonde}
\end{align}
See, e.g., \cite[page~15]{W90} for further details.

\subsection{Bad local behavior}
%

The following key lemma states that if Theorem~\ref{thm:doubleV} does not hold for some vertex $v_0$ of type ($\star$) in a  graph $G$, then every generic framework $(G,\bp)$ has a non-trivial $C_2^1$-motion $\bq$ with the property that each coordinate  of $\bq(v_i)$,  for $v_i\in \hat N_G(v_0)$, can be expressed as a polynomial in the coordinates of the points $\bp(v_j)$, for $v_j\in \hat N_G(v_0)$.  
We  note that the precise formulae of Lemma~\ref{lem:motion} are not important in the remaining proof, and only the fact  that each coordinate is written as a polynomial in the coordinates of $\bp(v_0), \bp(v_1),\dots, \bp(v_5)$ will be important.
\begin{lemma}
\label{lem:motion}
Let $(G,\bp)$ be a generic framework,
 $v_0$ be a vertex of type $(\star)$ in $G$
with $N_G(v_0)=\{v_1,\dots, v_5\}$,  $H=G-v_0$
and $\bp|_H$ be the restriction of $\bp$ to $H$. 
Suppose that 
$G$ is not $C_2^1$-rigid. Then there exist non-trivial $C_2^1$-motions $\bq_1$ of $(H+v_1v_2, \bp|_H)$ and 
$\bq_2$ of $(H+v_1v_3, \bp|_H)$ satisfying
\begin{align*}
\bq_1(v_1)&=\bq_1(v_2)=\bq_1(v_5)=0, \\ 
\bq_1(v_3)&=\Delta(\bp(v_1),\bp(v_2),\bp(v_3)) D(v_3, v_4) \times D(v_3, v_5),  \\
\bq_1(v_4)&=\Delta(\bp(v_1),\bp(v_2),\bp(v_4)) D(v_3, v_4)\times D(v_4, v_5),
\end{align*}
and
\begin{align*}
\bq_2(v_1)&=\bq_2(v_3)=\bq_2(v_5)=0, \\
\bq_2(v_2)&=\Delta(\bp(v_1), \bp(v_3), \bp(v_2)) D(v_2, v_4)\times D(v_2, v_5), \\
\bq_2(v_4)&=\Delta(\bp(v_1), \bp(v_3), \bp(v_4)) D(v_2, v_4)\times D(v_4, v_5),
\end{align*}
where $\times$ denotes the 
cross product of two vectors.
Moreover, $G$ has a non-trivial motion $\bq$ such that $\bq|_{V-v_0} = \alpha \bq_1 + \beta \bq_2$, where
\begin{align*}
\alpha=\begin{vmatrix} D(v_0, v_2) \\ D(v_2, v_4) \\ D(v_2,v_5) \end{vmatrix} \begin{vmatrix} D(v_0, v_3) \\ D(v_0, v_1) \\ D(v_0, v_5) \end{vmatrix}, 
\quad
\beta=-\begin{vmatrix} D(v_0, v_3) \\ D(v_3, v_4) \\ D(v_3, v_5) \end{vmatrix} \begin{vmatrix} D(v_0, v_2) \\ D(v_0, v_1) \\ D(v_0, v_5) \end{vmatrix},
\end{align*}
and $\bq(v_0)$ is given by 
\[
\bq(v_0)=\Delta(\bp(v_1), \bp(v_2) \bp(v_3))\begin{vmatrix} D(v_0, v_3) \\ D(v_3, v_4) \\ D(v_3, v_5) \end{vmatrix} 
\begin{vmatrix} D(v_0, v_2) \\ D(v_2, v_4) \\ D(v_2, v_5) \end{vmatrix}
 D(v_0, v_1)\times D(v_0, v_5).
\]
\end{lemma}
\begin{proof}
For simplicity, put $D_{i,j}=D(v_i,v_j)$ and $\Delta_{i, j, k}=\Delta(\bp(v_i), \bp(v_j), \bp(v_k))$.

Since $v_0$ is of type $(\star)$ in $G$, 
$(H+v_1v_2+v_1v_3,\bp|_H)$ and $(H+v_1v_3+v_3v_4,\bp|_H)$ are both  $C_2^1$-rigid.
Also ${\rm cl}(E(H))$ forms a star on $N_G(v_0)$ centered at $v_5$.
 We first prove a result concerning the $C_2^1$-motions of $(H,\bp|_H)$.

\begin{claim}\label{claim:Hbase}  There exist non-trivial $C_2^1$-motions $\bq_1$ of $(H+v_1v_2,\bp|_H)$ and $\bq_2$ of $(H+v_1v_3,\bp|_H)$ such that:\\
(a) $\bq_1(v_1)=\bq_1(v_2)=\bq_1(v_5)=0$, $D_{3,4}\cdot [\bq_1(v_3)-\bq_1(v_4)]=0$, and
$D_{1,3}\cdot \bq_1(v_3)\neq 0$;\\
(b) $\bq_2(v_1)=\bq_2(v_3)=\bq_2(v_5)=0$, $D_{2,4}\cdot [\bq_2(v_2)-\bq_2(v_4)]=0$, and
$D_{1,2}\cdot \bq_2(v_2)\neq 0$;\\
(c) $Z(H,\bp|_H)=Z_0(H,\bp|_H)\oplus\langle \bq_1,\bq_2 \rangle$.
\end{claim}
\begin{proof}
Since  ${\rm cl}(E(H))$ forms a star on $N_G(v_0)$ centered at $v_5$, 
${\rm cl}(E(H)+v_1v_2)$ contains the triangle $K(v_1,v_2,v_5)$.
Since $G$ contains a $1$-extension of $H+v_1v_2$ and  is not $C_2^1$-rigid, $H+v_1v_2$ is not $C_2^1$-rigid. Hence we can choose a non-trivial motion $\bq_1$ of $(H+v_1v_2,\bp|_H)$ such that  $\bq_1(v_1)=\bq_1(v_2)=\bq_1(v_5)=0$ (by pinning the triangle on $\{v_1,v_2,v_5\}$).
The fact that $(H+v_1v_2+v_1v_3,\bp|_H)$ is  $C_2^1$-rigid implies that
$D_{1,3}\cdot [\bq_1(v_3)-\bq_1(v_1)]=D_{1,3}\cdot \bq_1(v_3)\neq 0$. 
If  $D_{3,4}\cdot [\bq_1(v_3)-\bq_1(v_4)]\neq 0$ then $H+v_1v_2+v_3v_4$ would be $C_2^1$-rigid, and $G$ would also be $C_2^1$-rigid since it can be obtained from $H+v_1v_2+v_3v_4$ by an X-replacement. Thus $D_{3,4}\cdot [\bq_1(v_3)-\bq_1(v_4)]= 0$. 
This completes the proof of (a). Part (b) can be proved analogously.

The facts that  $D_{1,2}\cdot [\bq_1(v_1)-\bq_1(v_2)]=0\neq D_{1,2}\cdot [\bq_2(v_1)-\bq_2(v_2)]$ and
$D_{1,3}\cdot [\bq_1(v_1)-\bq_1(v_3)]\neq 0= D_{1,3}\cdot [\bq_2(v_1)-\bq_2(v_3)]$ imply that $\bq_1,\bq_2$ are linearly independent and can be extended to a base of $Z(H,\bp|_H)$ by adding the vectors in a base of $Z_0(H,\bp|_H)$. Hence $Z(H,\bp|_H)=Z_0(H,\bp|_H)\oplus\langle \bq_1,\bq_2 \rangle$.
\end{proof}


\begin{claim}\label{claim:Gmotion}  There exists a non-trivial $C_2^1$-motion $\bq$ of $(G,\bp)$ such that $\bq|_H=\alpha \bq_1+\beta \bq_2$ for some $\alpha, \beta\in \R$ with $\beta\neq 0$.
\end{claim}
\begin{proof}
Since $G$ is not $C_2^1$-rigid, we may choose a non-trivial $C_2^1$-motion $\bq$ of $(G,\bp)$. Since $d_G(v_0)=5$, $\bq|_H$ is a non-trivial $C_2^1$-motion of $(H,\bp|_H)$. By Claim \ref{claim:Hbase}(c),
$\bq|_H=\bq_0+\alpha \bq_1+\beta \bq_2$ for some $\bq_0\in Z_0(H,\bp|_H)$ and $\alpha,\beta\in \R$. We may choose $\tilde \bq\in Z_0(G,\bp)$ such that $\tilde \bq|_H=\bq_0$. Then $\hat \bq=\bq-\tilde \bq$ is a non-trivial $C_2^1$-motion of $(G,\bp)$ and  $\hat \bq|_H=\alpha \bq_1+\beta \bq_2$.
Since $(G+v_1v_2,\bp)$ can be obtained from  $(H+v_1v_2+v_1v_3,\bp|_H)$ by first performing a 1-extension operation and then adding the edge $v_0v_5$, $(G+v_1v_2,\bp)$ is $C_2^1$-rigid. This tells us that $D_{1,2}\cdot [\hat \bq(v_1)- \hat \bq(v_2)]\neq 0$ and hence $\beta\neq 0$. 
\end{proof}

Let  $\alpha, \beta$ be as described in Claim~\ref{claim:Gmotion}, and put $\gamma=\alpha/\beta$.
For all $1\leq i\leq 5$, $G$ contains the edge $v_0v_i$, and so
$D_{0,i}\cdot [\bq(v_0)-\bq(v_i)]=0$. Since
$\bq(v_i)=\alpha \bq_1(v_i)+\beta \bq_2(v_i)$ for $1\leq i\leq 5$,
we obtain the system of equations
$$ D_{0,i}\cdot [\bq(v_0)/\beta-\gamma \bq_1(v_i)]=D_{0,i}\cdot \bq_2(v_i) \mbox{ for $1\leq i\leq 5$}.$$
Putting $\bq(v_0)/\beta=(a_0,b_0,c_0)$, and using the facts that $\bq_1$ is zero on $v_1,v_2,v_5$ and $\bq_2$ is zero on $v_1,v_3,v_5$, we may rewrite this system as the matrix equation

\begin{equation}
\label{eq:dV2}
\left(
\begin{array}{ccc}
D_{0,1} &0 \\
D_{0,2} & 0\\
D_{0,3} & -D_{0,3}\cdot \bq_1(v_3)\\
D_{0,4} & -D_{0,4}\cdot \bq_1(v_4)\\
D_{0,5} & 0
\end{array}
\right) \left(
\begin{array}{c}
a_0\\ b_0\\ c_0\\ \gamma
\end{array}
\right) =
\left(
\begin{array}{c}
0 \\
D_{0,2}\cdot \bq_2(v_2)\\
0\\
D_{0,4}\cdot \bq_2(v_4)\\
0
\end{array}
\right)\,.
\end{equation}

\bigskip\noindent
Since this equation has a solution for $a_0,b_0,c_0,\gamma$ we have
\begin{equation}\label{eq:dV3}
\left|
\begin{array}{ccc}
D_{0,1} &0 & 0\\
D_{0,2} & 0 & D_{0,2}\cdot \bq_2(v_2)\\
D_{0,3} & -D_{0,3}\cdot \bq_1(v_3) & 0\\
D_{0,4} & -D_{0,4}\cdot \bq_1(v_4) & D_{0,4}\cdot \bq_2(v_4)\\
D_{0,5} & 0 & 0
\end{array}
\right| = 0\,.
\end{equation}

Let $\bp(v_0)=(x_0,y_0)$.
The left hand side of (\ref{eq:dV3}) is a polynomial in $x_0,y_0$ with coefficients in $\mathbb{Q}(\bp|_H, \bq_1, \bq_2)$. Since it is equal to zero for all generic values of $x_0,y_0$, each coefficient is zero. This implies that (\ref{eq:dV3}) will hold for {\em all} values of $x_0$ and $y_0$, including  non-generic choices of $\bp(v_0)$.

\begin{claim}\label{cllm:dV4-1}
$D_{3,4}\cdot \bq_1(v_3)=0.$
\end{claim}
\begin{proof}
We choose $\bp(v_0)$ to be a point on the interior of the line segment joining  $\bp(v_3)$ and $\bp(v_4)$.
Then $D_{0,3}$ and $D_{0,4}$ are scalar multiples of $D_{3,4}$.
We can now use elementary row and column operations to convert (\ref{eq:dV3}) to
\begin{equation}
\label{eq:dV3-2}
\left|
\begin{array}{ccc}
D_{0,1} &0 & 0\\
D_{0,2} & 0 & D_{0,2}\cdot \bq_2(v_2)\\
D_{3,4} & D_{3,4}\cdot \bq_1(v_3) & 0\\
0 & D_{3,4}\cdot (\bq_1(v_4)-\bq_1(v_3)) & D_{3,4}\cdot \bq_2(v_4)\\
D_{0,5} & 0 & 0
\end{array}
\right| = 0.\,
\end{equation}
Since $D_{3,4}\cdot (\bq_1(v_4)-\bq_1(v_3))=0$ by Claim \ref{claim:Hbase}(a), this gives
\begin{equation}
\label{eq:dV3-3}
\left|
\begin{array}{ccc}
D_{0,1} &0 & 0\\
D_{0,2} & 0 & D_{0,2}\cdot \bq_2(v_2)\\
D_{3,4} & D_{3,4}\cdot \bq_1(v_3) & 0\\
0 &0 & D_{3,4}\cdot \bq_2(v_4)\\
D_{0,5} & 0 & 0
\end{array}
\right| = 0.\,
\end{equation}
Hence
\begin{equation}\label{eq:dV3-4}
\left|
\begin{array}{c}
D_{0,1} \\ D_{0,2} \\ D_{0,5}
\end{array}\right|
(D_{3,4}\cdot \bq_1(v_3)) \,
(D_{3,4}\cdot \bq_2(v_4))=0.
\end{equation}
As long as $\bp(v_0)$ is generic on the line through $\bp(v_3)$ and $\bp(v_4)$, the first term is non-zero. Since $\bq_2$ is a non-trivial $C_2^1$-motion of $(H+v_1v_3,\bp|_H)$ and $(H+v_1v_3+v_3v_4,\bp|_H)$ is  $C_2^1$-rigid, we also have
$D_{3,4}\cdot \bq_2(v_4)=D_{3,4}\cdot (\bq_2(v_4)-\bq_2(v_3))\neq 0$.
Hence (\ref{eq:dV3-4}) gives $D_{3,4}\cdot \bq_1(v_3)=0$. 
\end{proof}

Since the edge $v_3v_5$ exists in ${\rm cl}(E(H))$ and $\bq_1(v_5)=0$,
we also have
\begin{equation}
\label{eq:dV4-2}
D_{3,5}\cdot \bq_1(v_3)=0.
\end{equation}
Claim~\ref{cllm:dV4-1} and (\ref{eq:dV4-2}) imply that $\bq_1(v_3)$ is a vector in the orthogonal complement of the span of $D_{3,4}$ and $D_{3,5}$.
Hence, scaling $\bq_1$ appropriately, we may assume that 
\begin{equation}
\label{eq:dV4-3}
\bq_1(v_3) = \Delta_{1,2,3} \,D_{3,4}\times D_{3,5}.
\end{equation}
By an analogous argument on the line through $\bp(v_2)$ and $\bp(v_4)$, we obtain
\begin{equation}
\label{eq:dV4-4}
\bq_2(v_2) = \Delta_{1,3,2} D_{2,4}\times D_{2,5}.
\end{equation}

We next calculate $\bq_1(v_4)$. 
Since the edge $v_4v_5$ exists in ${\rm cl}(E(H))$ and $\bq_1(v_5)=0$, we have 
$D_{4,5}\cdot \bq_1(v_4)=0$.
Also the edge $v_3v_4$ exists in ${\rm cl}(E(H)+v_1v_2)$, since otherwise we could perform an X-replacement on $H+v_1v_2+v_3v_4$ and deduce that $G$ is $C_2^1$-rigid.
Hence $D_{3,4}\cdot (\bq_1(v_4)-\bq_1(v_3))=0$. Since $D_{3,4}\cdot \bq_1(v_3)=0$ by Claim \ref{cllm:dV4-1}, we have $D_{3,4}\cdot \bq_1(v_4)=0$.
Thus $\bq_1(v_4)$ is in the orthogonal complement of the span of $D_{4,5}$ and $D_{3,4}$,
and $\bq_1(v_4)=s D_{3,4}\times D_{4,5}$ for some $s\in\mathbb{R}$.
The next claim determines the value of $s$:
\begin{claim}
$s=\Delta_{1,2,4}$.
\end{claim}
\begin{proof}
We put $\bp(v_0)$ on the interior of the line segment joining  $\bp(v_1)$ and $\bp(v_2)$.
Then $D_{0,1}$ is a scalar multiple of  $D_{0,2}$.
With this choice of $\bp(v_0)$, we may expand the determinant in (\ref{eq:dV3})
to obtain
\begin{equation}
\label{eq:dV5}
\left|
\begin{array}{c}
D_{0,1} \\ D_{0,5} \\ D_{0,3}
\end{array}\right|
(D_{0,4}\cdot \bq_1(v_4) )
(D_{0,2}\cdot \bq_2(v_2))
-
\left|
\begin{array}{c}
D_{0,1} \\ D_{0,5} \\ D_{0,4}
\end{array}\right|
(D_{0,3}\cdot \bq_1(v_3) )
(D_{0,2}\cdot \bq_2(v_2))
=0.
\end{equation}
We may use the fact that $D_{1,2}\cdot \bq_2(v_2)\neq 0$ by Claim \ref{claim:Hbase}(b) to deduce that $D_{0,2}\cdot \bq_2(v_2)\neq 0$ and  then rewrite (\ref{eq:dV5}) as
\begin{equation}
\label{eq:dV6}
\left|
\begin{array}{c}
D_{0,1} \\ D_{0,5} \\ D_{0,3}
\end{array}\right|
(D_{0,4}\cdot \bq_1(v_4) )
-
\left|
\begin{array}{c}
D_{0,1} \\ D_{0,5} \\ D_{0,4}
\end{array}\right| 
(D_{0,3}\cdot \bq_1(v_3) )
=0.
\end{equation}
Using $\bq_1(v_4)=s D_{3,4}\times D_{4,5}$, and substituting $\bq_1(v_3)$ with \eqref{eq:dV4-3}, we obtain 
\begin{equation}\label{eq:dV7}
s = \frac{
\left|
\begin{array}{c}
D_{0,1} \\ D_{0,5} \\ D_{0,4}
\end{array}\right|
\left|
\begin{array}{c}
D_{0,3} \\ D_{3,4} \\ D_{3,5}
\end{array}\right|}{ 
\left|
\begin{array}{c}
D_{0,1} \\ D_{0,5} \\ D_{0,3}
\end{array}\right|
\left|
\begin{array}{c}
D_{0,4} \\ D_{3,4} \\ D_{4,5}
\end{array}\right|}\Delta_{1,2,3}.
\end{equation}

By using the identity (\ref{eq:Vandermonde}) and the fact $D_{i,j} = D_{j,i}$, (\ref{eq:dV7}) becomes 
\begin{equation}
s=\frac{\Delta_{0, 1, 5} \Delta_{0, 5, 4} \Delta_{0, 4, 1} \Delta_{3, 0, 4} \Delta_{3, 4, 5} \Delta_{3, 5, 0}}{\Delta_{0, 1, 5} \Delta_{0, 5, 3} \Delta_{0, 3, 1} \Delta_{4, 0, 3} \Delta_{4, 3, 5} \Delta_{4, 5, 0}}\Delta_{1,2,3}
=\frac{\Delta_{0,4,1}}{\Delta_{0,3,1}}\Delta_{1,2,3}.
\end{equation}
Finally since $\bp(v_0)$ is on the interior of the line segment joining $\bp(v_1)$ and $\bp(v_2)$, 
we get 
\[
s=\frac{\Delta_{0,4,1}}{\Delta_{0,3,1}}\Delta_{1,2,3}=\frac{\Delta_{1,2,4}}{\Delta_{1,2,3}}\Delta_{1,2,3}=\Delta_{1,2,4}.\qedhere
\]
\end{proof}
This completes the derivation  of the formula for $\bq_1$  in the statement of the lemma. 
The formula for $\bq_2$ can be obtained using a symmetric argument by changing the role of $v_2$ and $v_3$.

It remains to compute $\bq(v_0)$. By (\ref{eq:dV2}), $\bq(v_0)/\beta$ is in the orthogonal complement of the span of $D_{0,1}$ and $D_{0,5}$. Hence $\bq(v_0)/\beta = k D_{0,1} \times D_{0,5}$ for some $k\in \mathbb{R}$.
By the second equation in (\ref{eq:dV2}), we also have $D_{0,2}\cdot \bq_2(v_2) = D_{0,2}\cdot (\bq(v_0)/\beta) = k D_{0,2}\cdot(D_{0,1} \times D_{0,5})$. 
Combining this with (\ref{eq:dV4-4}), we get 
\begin{equation}
\label{eq:dV8}
k = \Delta_{1,3,2} \frac{D_{0,2}\cdot (D_{2,4}\times D_{2,5})}{D_{0,2}\cdot(D_{0,1} \times D_{0,5})}.
\end{equation}
By the third equation in (\ref{eq:dV2}), $(D_{0,3}\cdot \bq_1(v_3))\gamma = D_{0,3}\cdot \bq(v_0)/\beta= k D_{0,3}\cdot (D_{0,1} \times D_{0,5})$. Hence by  (\ref{eq:dV8}) and \eqref{eq:dV4-3} we get 
\begin{align*}
\frac{\alpha}{\beta}=\gamma &= \frac{k D_{0,3}\cdot (D_{0,1} \times D_{0,5})}{D_{0,3}\cdot \bq_1(v_3)}
=\frac{\Delta_{1,3,2} D_{02}\cdot(D_{24}\times D_{25}) D_{03}\cdot (D_{01}\times D_{05})}{\Delta_{1,2,3} D_{03}\cdot(D_{34}\times D_{35}) D_{02}\cdot (D_{01}\times D_{05})} \\
&=-\frac{
\left|
\begin{array}{c}
D_{0,2} \\ D_{2,4} \\ D_{2,5}
\end{array}\right|
\left|
\begin{array}{c}
D_{0,3} \\ D_{0,1} \\ D_{0,5}
\end{array}\right|}{
\left|
\begin{array}{c}
D_{0,3} \\ D_{3,4} \\ D_{3,5}
\end{array}\right|
\left|
\begin{array}{c}
D_{0,2} \\ D_{0,1} \\ D_{0,5}
\end{array}\right|}.
\end{align*}
By scaling $\bq$ appropriately, we may take
$$\alpha = \left|
\begin{array}{c}
D_{0,2} \\ D_{2,4} \\ D_{2,5}
\end{array}\right|
\left|
\begin{array}{c}
D_{0,3} \\ D_{0,1} \\ D_{0,5}
\end{array}\right|
\mbox{ and }
\beta = -
\left|
\begin{array}{c}
D_{0,3} \\ D_{3,4} \\ D_{3,5}
\end{array}\right|
\left|
\begin{array}{c}
D_{0,2} \\ D_{0,1} \\ D_{0,5}
\end{array}\right|.$$ 
Since $\bq(v_0)= \beta k D_{0,1} \times D_{0,5}$, we can now use (\ref{eq:dV8}) to obtain the formula for $\bq(v_0)$ given in the statement of the lemma.
\end{proof}

\subsection{Bad motions and proof of Theorem \ref{thm:doubleV}}
Let $(G,\bp)$ be a generic framework and $v_0$ be a vertex of type $(\star)$ in $G$ with $N_G(v_0)=\{v_1,\dots, v_5\}$. Lemma~\ref{lem:motion} shows that, if $G$ is not $C_2^1$-rigid, then $(G,\bp)$ is a 1-dof framework and has a non-trivial motion $\bq$ with the property that the coordinates of $\bq$ at each vertex  of $\hat{N}_G(v_0)$ are three polynomials in the coordinates of $\bp(\hat{N}_G(v_0))$ with coefficients in $\rat$.
This seems unlikely as the graph on $N_G(v_0)$ induced by $\cl(E(G-v_0))$ 
is a star on the five vertices  (since $v_0$ is of type $(\star)$), which has a large degree of freedom as a subframework.
Our goal is to give a rigorous proof that this cannot happen. 
The following concepts will be useful 
for this purpose.

Let $G=(V,E)$ be a graph with $V=\{v_1,v_2,\ldots, v_n\}$. For $U\subseteq V$, let 
\[
b:U\to \rat[X_1,Y_1,X_2, Y_2,\ldots,X_n,Y_n]^3
\]
be a map which associates a 3-tuple $b_{i}$ of polynomials in $2n$ variables to each $v_i\in U$.

For a given framework $(G,\bp)$ with $\bp(v_i)=(x_i,y_i)$ for all $1\leq i\leq n$, the substitution of $(X_i, Y_i)$ with $(x_i,y_i)$ in $b_{i}$ for all $1\leq i\leq n$ gives a vector $b_{i}(\bp)$ in $\R^3$
for each $v_i\in U$.
We say that a $C_2^1$-motion $\bq$ of  $(G,\bp)$ is a {\em $b$-motion}  if $\bq(v_i)=b_{i}(\bp)$  for all $v_i\in U$.

The following lemma states that the property of having  a $b$-motion is stable over a large set $S$ of realisations. (In particular, it will imply that this  property is generic). Since we will need the realisations in $S$ to satisfy the hypotheses of Lemma \ref{lem:1dof_motion1} it will be convenient to say that the realisation $\bp$ is {\em non-degenerate on $U$} if there exist three distinct vertices in $U$ which have distinct `$y$-coordinates' in $(G,\bp)$.

\begin{lemma}\label{lem:bad_pinning}
Let $G=(V,E)$ be a $C_2^1$-independent graph with $V=\{v_1,\dots,v_n\}$, 
$U
\subseteq V$, $F
\subseteq K(U)$, 
and
$
S=\{\bp:\mbox{$(G+F,\bp)$ is minimally $C_2^1$-rigid and $\bp$ is non-degenerate on $U$}\}.
$
Suppose that $b:U\to \rat[X_1,Y_1,\ldots,X_n,Y_n]^3$ and $(G,\bp)$ has a $b$-motion for some generic $\bp$. Then $(G, \bp)$ has a $b$-motion for all $ \bp\in S$.
\end{lemma}
%

Lemma~\ref{lem:bad_pinning} can be understood as follows.
Since each entry of $C(G,\bp)$ is a polynomial function of the coordinates of $\bp$, we can choose a base of the space of $C_2^1$-motions of $(G,\bp)$ such that 
each entry of the vectors in the base is a polynomial function of the coordinates of $\bp$.
By the assumption of the lemma, $(G,{\bp})$ has a $b$-motion for some generic ${\bp}$, which is spanned by  this base.
Since this holds for some generic ${\bp}$ and $b$ is a polynomial function, it will hold for all generic $\bp$. 
The statement of Lemma~\ref{lem:bad_pinning} follows for all generic $\bp$.
We need a slightly more involved argument to verify the statement for {all} $\bp\in S$. This is given in Appendix~\ref{sec:bad_pinning}.

Suppose $v_0$ is a vertex of type $(\star)$ in a generic framework $(G,\bp)$ and $(G,\bp)$ is not $C_2^1$-rigid. Then 
Lemma \ref{lem:motion} tells us that $(G,\bp)$ has a 
non-trivial motion $\bq$ 
with the property that the value of $\bq$ at each vertex  of $\hat{N}_G(v_0)$ can be described by three polynomials in the coordinates of $\bp(\hat{N}_G(v_0))$ with coefficients in $\rat$. We will see in Lemma~\ref{lem:bad_motion} below that $\bq$ also has the property that the graph on $N_G(v_0)$ with edge set $\{v_iv_j\,:\,D(v_i,v_j)\cdot (\bq(v_i)-\bq(v_j))=0\}$ is a star. 
We will concentrate on $b$-motions which share these properties.
 
Formally, let $(G,\bp)$ be a framework and  $v_0$ be a vertex of degree five in $G$ with $N_G(v_0)=\{v_1,v_2,\ldots,v_5\}$. We say that a $C_2^1$-motion $\bq$ of $(G,\bp)$ is {\em bad at $v_0$} if $\bq$ is a $b$-motion for some $b:\hat N_G(v_0)\to \rat[X_0,Y_0,X_1, Y_1,\ldots,X_5,Y_5]^3$ for which 
\begin{equation}\label{eq:bad}
\text{the graph on $N_G(v_0)$ with edge set 
$\{v_iv_j: 
D_{i,j}\cdot (b(v_i)-b(v_j))=0 \}$
 is a star,}
\end{equation}
where $D_{i,j}=((X_i-X_j)^2, (X_i-X_j)(Y_i-Y_j), (Y_i-Y_j)^2)$.
(Here $D_{i,j}\cdot (b(v_i)-b(v_j))=0$ means that polynomial $D_{i,j}\cdot (b(v_i)-b(v_j))$ is identically zero.)

Our next result verifies that the motion defined in Lemma \ref{lem:motion} is indeed a bad motion.

\begin{lemma}
\label{lem:bad_motion}
Let $(G,\bp)$ be a generic framework and $v_0$ be a vertex of type $(\star)$ 
with $N_G(v_0)=\{v_1,\dots, v_5\}$.
Suppose that $G$ is not $C_2^1$-rigid. 
Then $(G,\bp)$ has a bad motion at $v_0$.
\end{lemma}
\begin{proof}
Motivated by Lemma~\ref{lem:motion}, we define $b:\hat{N}_G(v_0)\rightarrow\rat[X_0,Y_0,X_1,\ldots,Y_5]^3$ by
\begin{align}
b(v_0)&=\Delta_{1,2,3} \left|
\begin{array}{c} D_{0,3} \\ D_{3,4} \\ D_{3,5}
 \end{array} \right|
\left|
\begin{array}{c} D_{0,2} \\ D_{2,4} \\ D_{2,5}
 \end{array} \right|
 D_{0,1}\times D_{0,5}, \label{eq:bad1}\\
b(v_1)&=0,\\
b(v_2)&=\beta \Delta_{1,3,2} D_{2,4}\times D_{2,5}, \\
b(v_3)&=\alpha \Delta_{1,2,3} D_{3,4} \times D_{3,5}, \\
b(v_4)&=\alpha\Delta_{1,2,4} D_{3,4}\times D_{4,5} +\beta \Delta_{1,3,4}D_{2,4}\times D_{4,5}, \\
b(v_5)&=0;  \label{eq:bad6}
\end{align}
where
 $\Delta_{i,j,k}=\left|
\begin{array}{ccc}
X_i & Y_i & 1 \\ X_j & Y_j & 1 \\ X_k & Y_k & 1
\end{array}\right|$,
$\alpha = \left|
\begin{array}{c}
D_{0,2} \\ D_{2,4} \\ D_{2,5}
\end{array}\right|
\left|
\begin{array}{c}
D_{0,3} \\ D_{0,1} \\ D_{0,5}
\end{array}\right|$,
$\beta = -
\left|
\begin{array}{c}
D_{0,3} \\ D_{3,4} \\ D_{3,5}
\end{array}\right|
\left|
\begin{array}{c}
D_{0,2} \\ D_{0,1} \\ D_{0,5}
\end{array}\right|$
are 
regarded as 
polynomials in $X_0,Y_0,X_1,\ldots,Y_5$.
Then, by Lemma~\ref{lem:motion}, $(G,\bp)$ has a $b$-motion at $v_0$.

It remains to show that $b$ satisfies (\ref{eq:bad}) symbolically.
This can be done by a hand computation using the explicit formulae,
see Appendix \ref{app:bad} for the details.
\end{proof}

Our next result asserts that  no generic framework with one degree of freedom can have a bad motion.
\begin{theorem}\label{thm:well-behaved0}
Let $(G,\bp)$ be a generic 1-dof framework and $v_0$ be a vertex of degree five.
Then $(G,\bp)$ has no bad motion  at $v_0$.
\end{theorem}

We will delay the proof of Theorem \ref{thm:well-behaved0} until  Section \ref{sec:well-behaved}. Instead we show how it can be used to deduce Theorem~\ref{thm:doubleV}.

\medskip
\noindent
{\bf Proof of Theorem~\ref{thm:doubleV}.}
Suppose the theorem is false. Let $G$ be a counterexample and $\bp$ be a generic realisation of $G$. 
Then Lemma~\ref{lem:bad_motion} implies that $(G,\bp)$ has a bad motion at $v_0$ and this contradicts Theorem~\ref{thm:well-behaved0}. 
(Note that $(G,\bp)$ is a 1-dof framework since $v_0$ is type $(\star)$.)
\qed

\subsection{Projective transformations}
Recall that each point $p=(x,y)\in \R^2$ can be associated to a point $p^\uparrow=[x,y,1]^T$ in 2-dimensional real projective space. 
A map $f:\R^2\to \R^2$ is a {\em projective transformation} if there exists a $3\times 3$ non-singular matrix $M$ such that, for all $p\in\R^2$, 
$f(p)^\uparrow=\lambda_pMp^\uparrow$ for some non-zero $\lambda_p\in \R$.
We say that a framework $(G,\bp')$ is a {\em projective image} of a framework $(G,\bp)$ if there is a 
projective transformation $f:\mathbb{R}^2\rightarrow \mathbb{R}^2$ such that $\bp'(v_i)=f(\bp(v_i))$ for all $v_i\in V(G)$.
Whiteley~\cite[Theorem 11.3.2.]{Wsurvey} showed that $C_2^1$-rigidity is invariant under projective transformations. 
We shall use the following lemma, which implies that a certain projective transformation preserves the existence of a bad motion, to simplify our calculations in the proof of Theorem \ref{thm:well-behaved0}. 
Its statement  implicitly uses the fact that, for any two ordered sets $S$ and $T$ of four points in general position in $\R^2$, 
there is a unique projective transformation that maps $S$ onto $T$.
\begin{lemma}\label{lem:projective_bad}
Let $(G,\bp)$ be a generic framework, $v_0$ be a vertex of degree five with $N_G(v_0)=\{v_1,\dots, v_5\}$ 
and $(G, \bp')$ be a projective image of $(G,\bp)$ such that 
$\bp'(v_1)=(1, 0)$, $\ \bp'(v_2)=(0, 0)$, $\bp'(v_3)=(0,1)$, and  $\bp'(v_4)=(1,1)$.
If $(G, \bp')$ has a bad motion at $v_0$, then $(G, \bp)$ has a bad motion at $v_0$.
\end{lemma}
Lemma~\ref{lem:projective_bad} will follow from the fact that the entries in the matrix which defines the projective transformation given in the lemma are  determined by  
the coordinates of $\bp(v_i)$ for $1\leq i\leq 4$. However, a rigorous proof seems to
require a precise understanding of how the projective transformation changes the motion space of the framework. Whiteley's proof of projective invariance does not provide this as it uses a  self-stress argument.
We develop a new proof technique for the projective invariance of $C_2^1$-rigidity in Appendix~\ref{sec:projective},
and then use it to derive Lemma~\ref{lem:projective_bad}.

\section{Proof of Theorem~\ref{thm:well-behaved0}}\label{sec:well-behaved}
\subsection{Locally $k$-dof parts}\label{subsec:kdof}\label{subsec:kdof_parts}
Let $G=(V,E)$ be a graph.
A set $X\subseteq V$ is said to be a {\em locally $C_2^1$-rigid part} in $G$ if every edge in $K(X)$ is in ${\rm cl}(E(G))$. 
The set $X$ is said to be a {\em locally $k$-dof part} in $G$ if 
there is a set $F$ of $k$ edges in $K(V)$ such that 
$X$ is a locally $C_2^1$-rigid part in $G+F$, but 
no smaller edge set has this property.

Our inductive proof of Theorem~\ref{thm:well-behaved0} requires the following more general  inductive hypothesis. 
\begin{theorem}\label{thm:well-behaved}
Let $(G,\bp)$ be a generic framework and $v_0$ be a vertex of degree five such that 
$\hat{N}_G(v_0)$ is a locally $k$-dof part in $G$ with $k=1$ or $k=2$.
Then $(G,\bp)$ has no bad motion  at $v_0$.
\end{theorem}

We give an outline of the proof of Theorem~\ref{thm:well-behaved} in Section~\ref{subsec:outline}, 
and delay the full proof to Section~\ref{subsec:final}.

\subsection{Outline of the proof of  Theorem~\ref{thm:well-behaved}}
 \label{subsec:outline}
Theorem~\ref{thm:well-behaved} is proved by contradiction. 
We choose a counterexample $(G,\bp)$ with $G=(V,E)$ and $|V|$ as small  as possible and, subject to this condition, $k$ as small as possible. 
The proof proceeds as follows.
\begin{description}
\item[Section~\ref{subsubsec:1}.] The first step  is to prove that $G$ is a 2-dof graph (Claim~\ref{claim:k=2}) and  $G$ has a vertex $u_0$ of degree five with $u_0\notin \hat{N}_G(v_0)$ (Claim~\ref{claim:2d4}).

\item[Sections~\ref{subsubsec:2} and~\ref{subsubsec:3}.] 
Let $G_0=(V_0, E_0)=G-u_0$.
We prove statement (46): ${\rm cl}(E_0)\cap K(N_G(u_0))$ is a star on five vertices.
The proof of (\ref{eq:star}) proceeds as in that of Lemma~\ref{lem:combinatorial}, but to apply a similar argument  we need to prepare several preliminary claims in Section~\ref{subsubsec:2}.

\item[Section~\ref{subsubsec:4}.] The final goal of the proof is to contradict the minimal choice of $(G,\bp)$ by showing that 
  $(G_0+e_1+e_2, \bp|_{V_0})$ is a 2-dof framework with a bad motion at $v_0$   
  for two edges $e_1,e_2\in K(N_G(u_0))$.
To do this we prove the existence of distinct  edges $f_1, f_2\in K(N_G(v_0))$ 
such that $G+f_1+f_2$ is $C_2^1$-rigid and distinct edges  $e_{j,1}, e_{j,2}\in K(N_G(u_0))$ for $j=1,2,3$ such that $(G_0+f_i+e_{j,1}+e_{j,2},\bp)$ has a non-trivial motion $\bq_j^i$ for all $i=1,2$ and $j=1,2,3$ so that   $Z(G_0+f_i, \bp|_{V_0})=Z_0(G_0+f_i, \bp|_{V_0})\oplus \langle \bq_1^i, \bq_2^i, \bq_3^i\rangle$.

\item[Section~\ref{subsubsec:5}.] Let $\bq^i$ be a nontrivial motion of $(G+f_i,\bp)$ for $i=1,2$. 
Since  $G+f_1+f_2$ is $C_2^1$-rigid, we have $Z(G, \bp)=Z_0(G, \bp)\oplus \langle \bq^1, \bq^2 \rangle$.
By $Z(G_0+f_i, \bp|_{V_0})=Z_0(G_0+f_i, \bp|_{V_0})\oplus \langle \bq_1^i, \bq_2^i, \bq_3^i\rangle$, we may suppose (by subtracting a suitable trivial motion) that, for both $i=1,2$,   
$\bq^i|_{V_0}=\sum_{j=1}^3\alpha_j^i\bq_j^i$ for some scalars $\alpha_j^i$.

Since  ${\rm cl}(E_0)\cap K(N_G(u_0))$ forms a star on five vertices,  
we can use the proof technique of Lemma~\ref{lem:motion} to derive an explicit formula for the scalars $\alpha^i_j$.

\item[Section~\ref{subsubsec:6}.] 
The original framework $(G,\bp)$ has a bad motion $\bq_{\rm bad}$ at $v_0$, and its restriction to $V_0$ is a motion of $(G_0, \bp|_{V_0})$.
Since  $Z(G, \bp)=Z_0(G, \bp)\oplus \langle \bq^1, \bq^2 \rangle$, we have 
\[
\bq_{\rm bad}|_{\hat{N}(v_0)}\in \left\langle \sum_{j=1}^3\alpha_j^1\, \bq_j^1|_{\hat{N}(v_0)}, \sum_{j=1}^3\alpha_j^2\, \bq_j^2|_{\hat{N}(v_0)}\right\rangle.
\]
We can use the explicit formulae for the scalars $\alpha_j^i$ derived in Section~\ref{subsubsec:5} and the facts that the entries of $\bq_{\rm bad}|_{\hat{N}(v_0)}$ are contained in $\mathbb{Q}(\bp(V_0))$ and $p(u_0)$ is generic over $\mathbb{Q}(\bp(V_0))$
to show  that
 \[
\bq_{\rm bad}|_{\hat{N}(v_0)}\in \left\langle  \bq_j^1|_{\hat{N}(v_0)},  \bq_j^2|_{\hat{N}(v_0)}\right\rangle
\]
for some $j\in \{1,2,3\}$.
Since any linear combination of $\bq_j^1$ and $\bq_j^2$ is a motion of $(G_0+e_{j,1}+e_{j,2},\bp|_{V_0})$,
this implies that  $(G_0+e_{j,1}+e_{j,2},\bp|_{V_0})$ has a bad motion at $v_0$. This contradiction to the minimal choice of $(G,\bp)$ completes the proof.
 \end{description}

\subsection{Proof of Theorem~\ref{thm:well-behaved}}\label{subsec:final}
We assume the statement of Theorem~\ref{thm:well-behaved} is false and choose a counterexample $(G,\bp)$ with $G=(V,E)$ and $v_0\in V$ such that the $G$ is minimal under the lexicographic ordering given by $(|V|,k, \mbox{dof } G,|E|)$.
Then $(G,\bp)$ has a bad motion at $v_0$,  
 $d_G(v_0)=5$ 
 and $\hat{N}_G(v_0)$ is a locally $k$-dof part in $G$.

\subsubsection{\boldmath 
Preliminary results on the structure of $G$}\label{subsubsec:1}
The first step of the proof is to show that $G$ is a  ${C}_2^1$-independent, 2-dof graph (statement (\ref{eq:Kv}) and Claims~\ref{claim:Gkdof} and \ref{claim:k=2}) and  $G$ has a vertex $u_0$ of degree five with $u_0\notin \hat{N}_G(v_0)$ (Claim~\ref{claim:2d4}).


Let $N_G(v_0)=\{v_1,\dots,v_5\}$ and denote for simplicity $K_v=K(N_G(v_0))$.
Since $(G,\bp)$ is a counterexample,  $(G,\bp)$ has a $b$-motion $\bq_{\rm bad}$ for some  $b:\hat{N}_G(v_0)\rightarrow \mathbb{Q}[X_0, Y_0, \dots,Y_5]$ which satisfies (\ref{eq:bad}).
Relabeling $N_G(v_0)$ if necessary, we may suppose that, for all $v_i,v_j\in N_G(v_0)$ with $i<j$,
\begin{equation}\label{eq:star1}
\text{$D(v_i,v_j)\cdot (\bq_{\rm bad}(v_i)-\bq_{\rm bad}(v_j))=0$ holds if and only if $j=5$.}
\end{equation}

By (\ref{eq:star1}),  $v_iv_j\notin {\rm cl}(E)$ for all $1\leq i<j\leq 4$.
If $v_iv_5\notin {\rm cl}(E)$ for some $1\leq i\leq 4$, then again by (\ref{eq:star1}), $(G+v_iv_5,\bp)$ would be a $(k-1)$-dof framework having the bad motion $\bq_{\rm bad}$, 
which would contradict either the fact that $(G+v_iv_5,\bp)$ is $C^1_2$-rigid (when $k=1$) or the minimality of $k$ (when $k=2$).
Hence, for all $v_i,v_j\in N_G(v_0)$,
 we have $v_iv_j\in {\rm cl}(E)$ if and only if $j=5$. 
 Then $\cl(E)\cap K_v$ induces a star on five vertices so is ${C}_2^1$-independent. This implies that
we may choose a base $E'$ of $\cl(E)$ in ${\cal C}_2^1$ with $\cl(E)\cap K_v\subseteq E'$ and $N_E(v_0)=N_{E'}(v_0)$. Let $G'$ be the subgraph of $K(V)$ induced by $E'$. Then $(G,\bp)$ and $(G',\bp)$ have the same space of motions. The minimality of $|E|$ now implies that $|E|=|E'|$ and $G$ is ${C}_2^1$-independent. Replacing $G$ by $G'$ if necessary, we may assume that
\begin{equation}\label{eq:Kv}
\text{$G$ is ${C}_2^1$-independent and $E\cap  K_v={\rm cl}(E)\cap K_v$ is a star on five vertices centred on $v_5$.}
\end{equation}



 \begin{claim}\label{claim:Gkdof} $G$ is a $k$-dof graph.
 \end{claim}
\begin{proof}
Suppose that $G$ has $k'$ degrees of freedom. We show that  there exists a set of $(k'-k)$ edges $D$ of $K(V-v_0)$ such that  $G+D$ is a $k$-dof graph,  $X:=\hat N_G(v_0)$ is a locally $k$-dof part in $G+D$ and, for any generic framework $(G,\bp)$ and any motion $\bq$ of $(G,\bp)$, $(G+D,\bp)$ has a motion $\bq'$ satisfying $\bq'(v)=\bq(v)$ for all $v\in X$.

Let 
$F$ be a set of $k$ edges in $K(V)$ such that $X$ is a locally $C_2^1$-rigid part in $G+F$, and $D$ be a minimal set of edges in $K(V-v_0)$ such that $G+F+D$ is $C^1_2$-rigid. (Note that $D$ exists since $d_G(v_0)=5$.) For each $e\in F\cup D$, $G+F+D-e$ is a 1-dof graph and we can choose a non-trivial $C^1_2$-motion $\bq_e$ of  $(G+F+D-e,\bp)$. 
Then $B_1:=\{\bq_i^*\,:\,1\leq i\leq 6\}\cup
\{\bq_e\,:\,e\in F\cup D\}$ is a base for $Z(G,\bp)$ and 
$B_2:=\{\bq_i^*\,:\,1\leq i\leq 6\}\cup
\{\bq_e\,:\,e\in F\}$ is a base for $Z(G+D,\bp)$,
where $\{q_i^*:1\leq i\leq 6\}$ is the basis for  $Z_0(G,p)$ given in Section~\ref{subsec:C_2^1-motions}.
 The facts that the motions in $B_1$ are linearly independent and that, for each $e\in F$, we have $D(u,v)\cdot (\bq_e(u)-\bq_e(v))\neq 0$ for some $u,v\in X$ now imply that we must add at least $k$ edges to $G+D$ to make $X$ locally $C_2^1$-rigid. Hence $X$ is a locally $k$-dof part in $G+D$.

Let $\bq$ be a motion of $(G,\bp)$. Since $B_1$ is a base for $Z(G,\bp)$, we have $\bq=\sum_{i=1}^6 \lambda_i \bq_i^*+\sum_{e\in  F\cup D}\mu_e \bq_e$ for some $\lambda_i, \mu_e\in \R$. Since $X$ is a locally $C_2^1$-rigid part of $G+F$, $\bq_e|_X$ is a trivial motion of $(K(X),\bp|_X)$ for all $e\in D$. Hence $\bq_e|_X=\bq_e^*|_X$ for some trivial motion $\bq_e^*$ of $(G,\bp)$. Let $\bq'=\sum_{i=1}^6 \lambda_i \bq_i^*+\sum_{e\in  F}\mu_e \bq_e+\sum_{e\in D}\mu_e \bq_e^*$. Then $\bq'\in Z(G+D,\bp)$.  Since $\bq_e|_X=\bq_e^*|_X$ for $e\in D$, we also have $\bq'|_X=\bq|_X$.

Hence $(G+D,\bp)$ is a counterexample. The minimality of dof $G$ now implies that $k'=k$ and $D=\emptyset$. Hence $G$ is a $k$-dof graph.
\end{proof}

As $G$ is a $k$-dof graph for some $1\leq k\leq 2$, (\ref{eq:Kv}) implies that $|V|\geq 7$.

\begin{claim}\label{claim:well-behaved2}
$G$ has a vertex $u_0$ of degree at most five with $u_0\notin\hat N_G(v_0)$.
\end{claim}
\begin{proof}
Put $G-\hat{N}_G(v_0)=G'=(V',E')$ and let
$F$ be the set of edges of $G$ from $\hat{N}_G(v_0)$ to $V'$.
Since $|V|\geq 7$, we have $|V'|\geq 1$.
By (\ref{eq:Kv}) we have 
\begin{equation}
\label{eq:degree1}
9+|F|+|E'|=|E|=3|V|-6-k=3(|V'|+6)-6-k.
\end{equation}

Suppose every vertex in $V'$ has degree at least six in $G$.
Then 
\begin{equation}
\label{eq:degree3}
\sum_{v'\in V'} d_G(v') = |F|+2|E'|\geq 6|V'|.
\end{equation}
We can now use (\ref{eq:degree1}) and (\ref{eq:degree3}) to deduce that $|E'|\geq 3|V'|-3+k\geq 3|V'|-2$. This contradicts the fact that $G$ is $C_2^1$-independent.
\end{proof}

Let $u_0\in V\setminus \hat{N}_G(v_0)$ be a vertex of degree at most five in $G$,
$K_u=K(N_G(u_0))$ and put $G-u_0=G_0=(V_0,E_0)$.
Let ${\cal C}_{n-1}$ be the generic $C_2^1$-cofactor matroid on $K(V_0)$ and let ${\cal C}_{n-1}/{\rm cl}(E_0)$ be the matroid obtained from ${\cal C}_{n-1}$ by contracting ${\rm cl}(E_0)$.

\begin{claim}
\label{claim:2mechanism1}
$F_v:=\{v_1v_2, v_1v_3, v_1v_4\}$ is a base of $K_v\setminus {\rm cl}(E_0)$ in ${\cal C}_{n-1}/{\rm cl}(E_0)$ (and hence
$K_v\setminus {\rm cl}(E_0)$ has rank three in ${\cal C}_{n-1}/{\rm cl}(E_0)$).
\end{claim}
\begin{proof}
By (\ref{eq:Kv}) the subgraph of $G+F_v$
induced by $\hat N_G(v_0)$ is $C_2^1$-rigid. Hence $F_v$ spans $K_v\setminus {\rm cl}(E_0)$ in ${\cal C}_{n-1}/{\rm cl}(E_0)$.
%
Choose
a base $F'$ of $F_v$ in ${\cal C}_{n-1}/{\rm cl}(E_0)$.
Since the closure of $F'$ contains $F_v$, 
${\rm cl}(E_0\cup F')$ contains a $C^1_2$-rigid subgraph on $\hat{N}_G(v_0)$.
Hence  $\hat{N}_G(v_0)$ is a locally $t$-dof part in $G_0$ with $t = |F'|$.
Since $\bq_{\rm bad}|_{V_0}$ is a  bad motion of $(G_0, \bp|_{V_0})$  at $v_0$ and $G_0$ is not a counterexample by the minimality of $|V|$, we must have  $t=|F'|=3$ and $F'=F_v$.
\end{proof}

\begin{claim}\label{claim:degree4or5}
$d_G(u_0)\in\{4,5\}$.
\end{claim}

\begin{proof}
If $d_G(u_0)\leq 3$, then $G_0$ would be a $t$-dof graph for some $t\leq 2$. This would contradict Claim~\ref{claim:2mechanism1}.
\end{proof}

\begin{claim}\label{claim:k=2} $k=2$.
\end{claim}
\begin{proof}
Suppose for a contradiction that $k=1$. If $d_G(u_0)=4$ then 
$G_0$ would be a $t$-dof graph for some $t\leq 2$ and this would contradict Claim \ref{claim:2mechanism1}.
Hence we may suppose $d_G(u_0)=5$. Let  $N_G(u_0)=\{u_1,u_2,\ldots,u_5\}$. We will obtain a contradiction by showing that $(G_0+u_iu_j,\bp|_{V_0})$ has a bad motion at $v_0$ for some $1\leq i<j\leq 5$.

Since $G$ is a 1-dof graph, $G_0$ is a 3-dof graph and hence, by Claim \ref{claim:2mechanism1}, $G_0+F_v$ is minimally $C_2^1$-rigid. 
As $G+F_v-u_0u_4-u_0u_5$  is obtained from $G_0+F_v$ by a 0-extension, 
Lemma \ref{lem:cofactor_ind} now implies that $G+F_v-u_0u_4-u_0u_5$ is minimally $C_2^1$-rigid. Since $G$ is $C_2^1$-independent,
$G+F-u_0u_5$ is minimally $C_2^1$-rigid for some $F\subset F_v$ with $|F|=2$. 
We may now deduce that  $G_0+F+u_iu_j$ is minimally $C_2^1$-rigid for some $1\leq i<j\leq 4$, and hence that $G_0+u_iu_j$ is a 2-dof graph. 
It remains to show that there is a bad motion for $(G_0+u_iu_j,\bp|_{V_0})$ at $v_0$. 

Let $G'=G-u_0u_5$. Observe that $G'$ is obtained from $G_0+u_iu_j$ by a 1-extension. 
We shall consider the following non-generic realisation $\bp'$ of $G'$, 
which was used in the proof of \cite[Theorem 10.2.1]{Wsurvey} to show that 
1-extension preserves independence. For $w\in V$, let
\[
\bp'(w)=\begin{cases}
\bp(w) & (\text{if } w\in V\setminus \{u_0\}) \\
\text{the mid-point of the line segment joining $\bp(u_i)$ and $\bp(u_j)$} & (\text{if } w=u_0).
\end{cases}
\]
See Figure~\ref{fig:claim5-3}.
\begin{figure}[h]
\centering
\begin{minipage}{0.45\textwidth}
\centering
\includegraphics[scale=0.6]{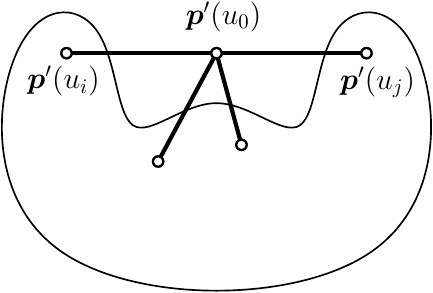}
\par(a)
\end{minipage}
\begin{minipage}{0.45\textwidth}
\centering
\includegraphics[scale=0.6]{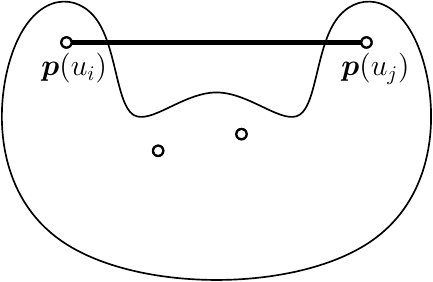}
\par(b)
\end{minipage}
\caption{(a) $(G',\bp')$ and (b) $(G_0+u_iu_j,\bp_{|V_0})$.}
\label{fig:claim5-3}
\end{figure}

We now use Lemma~\ref{lem:bad_pinning} to show that $(G',\bp')$ has a bad motion at $v_0$. Let
\[
S=\left\{\bar \bp: 
 \text{$(G'+F,\bar \bp)$ is minimally $C_2^1$-rigid and $\bar \bp$ is non-degenerate on $\hat N_{G'}(v_0)$}
\right\}.
\]
We may use  the minimal $C_2^1$-rigidity of $(G_0+F+u_iu_j,\bp|_{V_0})$, 
and the fact that $\{u_0,u_i,u_j\}$ is collinear in $(G'+F+u_iu_j,\bp')$ to deduce that 
$(G'+F,\bp')$ is minimally $C_2^1$-rigid, see the proof of \cite[Theorem 10.2.1]{Wsurvey} for more details. 
Since $\bp$ is generic, $\bp'$ is non-degenerate  on $\hat N_{G'}(v_0)$ and hence $\bp'\in S$.
Since $\bq_{\rm bad}$ is a bad motion of $(G',\bp)$  at $v_0$,
Lemma~\ref{lem:bad_pinning} now implies that $(G',\bp')$ has a bad motion $\bq_{\rm bad}'$ at $v_0$.
%
The definition of $\bp'$ implies that 
$\bq_{\rm bad}'|_{V_0}$ is a motion of $(G_0+u_iu_j,\bp|_{V_0})$.  
Since $(G_0+u_iu_j,\bp|_{V_0})$ is a generic 2-dof framework, $\bq_{\rm bad}'|_{V_0}$ is bad at $v_0$ and this
contradicts the minimality of $G$.
\end{proof}

Recall that $F_v=\{v_1v_2,v_1v_3,v_1v_4\}$.
We will denote the set of all pairs of edges in $F_v$ by ${F_v\choose 2}$.
We say that $F\in {F_v\choose 2}$ is {\em good} if $G+F$ is $C_2^1$-independent (and hence minimally $C_2^1$-rigid by $k=2$).
We say that an edge $f\in F_v$ is {\em very good} if $\{f, f'\}$ is good for any $f'\in F_v\setminus \{f\}$.
\begin{claim}\label{claim:good}
There is at least one very good edge in $F_v$. (Equivalently,  at most one pair in ${F_v\choose 2}$ is not good).
\end{claim}
\begin{proof}
Let $F_v=\{f_1, f_2, f_3\}$, and suppose that $\{f_1, f_2\}$ and $\{f_1, f_3\}$ are not good.
By  (\ref{eq:Kv}),
 $f_1\notin {\rm cl}(E)$. Hence $G+f_1$ is $C_2^1$-independent and $f_2,f_3\in {\rm cl}(E+f_1)$.
This in turn implies that $K_v\subseteq {\rm cl}(E+f_1)$ 
and hence $\hat{N}_G(v_0)$ is a locally 1-dof part in $G$. 
This is a contradiction as $\hat{N}_G(v_0)$ is a locally $k$-dof part with $k=2$ by Claim~\ref{claim:k=2}.
\end{proof}
%

\begin{claim}\label{claim:2mechanism2}
Let $F\in {F_v\choose 2}$ be a good pair.
Then ${\rm cl}(E_0\cup F)\cap K_u$ does not contain a copy of $K_4$.
\end{claim}
\begin{proof}
Suppose that ${\rm cl}(E_0\cup F)\cap K_u$  contains a copy of $K_4$, say on $\{u_1,\dots, u_4\}\subseteq N_G(u_0)$.
Then $u_0u_4\in \cl((E_0\cup F)+u_0u_1+u_0u_2+u_0u_3)$.
This in turn implies that $E\cup F$ is dependent and  contradicts the hypothesis that $F$ is good.
\end{proof}

\begin{claim}\label{claim:2d4}
$d_G(u_0)=5$.
\end{claim}
\begin{proof}
Assume this is false. Then by Claim \ref{claim:degree4or5}, $d_G(u_0)=4$.
Choose a good pair $F\in {F_v\choose 2}$. 
By Claim~\ref{claim:2mechanism2}, there is an edge $e=u_iu_j\in K_u$ such that 
$G_0+F+u_iu_j$ is $C_2^1$-independent. Since $d_G(u_0)=4$, it is minimally $C_2^1$-rigid.
We can now proceed as in the last paragraph of the proof of Claim \ref{claim:k=2} (with $G'$ replaced by $G$) to show that $(G_0+u_iu_j,\bp|_{V_0})$ has a bad motion at $v_0$ and hence
%
%
%
%
contradict the minimality of $|V|$.
\end{proof}

\subsubsection{\boldmath Preliminary results on $\cl(E_0)\cap K_u$}\label{subsubsec:2}
 Our next goal is to show that 
 ${\rm cl}(E_0)\cap K_u$ is a star on five vertices
 (statement (\ref{eq:star}) in the next section). 
In this section we verify some preliminary claims.
Let $N_G(u_0)=\{u_1,u_2,\ldots,u_5\}$.




\begin{claim}\label{claim:2mechanism3}
$G_0+F+e_1+e_2$ is $C_2^1$-dependent for all $F\in {F_v\choose 2}$
and all  pairs of non-adjacent edges  $e_1, e_2\in K_u\setminus E_0$.
\end{claim}
\begin{proof}
Suppose that 
$G_0+F+e_1+e_2$ is $C_2^1$-independent (and hence  minimally $C_2^1$-rigid) for $e_1=u_1u_2$ and $e_2=u_3u_4$. We will contradict the minimality of $G$ by showing that $(G_0+e_1+e_2,\bp|_{V_0})$ has a bad motion at $v_0$. 


Observe that $G+F$ is obtained from $G_0+F+e_1+e_2$ by an X-replacement. 
We shall consider the following non-generic realisation $\bp'$ of $G$, 
which was used in the proof of \cite[Theorem 10.3.1]{Wsurvey} to show that 
X-replacement preserves $C_2^1$-independence. For $w\in V$, let 
$\bp'(w)=\bp(w)$ for $w\in V_0$ and $\bp'(u_0)$ be the point of intersection of the  line  through $\bp(u_1)$ and $\bp(u_2)$, and the line  through $\bp(u_3)$ and $\bp(u_4)$.
We first use Lemma~\ref{lem:bad_pinning} to show that $(G,\bp')$ has a bad motion at $v_0$. To do this, let
\[
S=\left\{\bar \bp: 
 \text{$(G+F,\bar \bp)$ is minimally $C_2^1$-rigid and $\bar \bp$ is non-degenerate on $\hat N_G(v_0)$}
\right\}.
\]
We may use  the minimal $C_2^1$-rigidity of $(G_0+F+e_1+e_2,\bp|_{V_0})$, 
and the fact the vertex sets $\{u_0,u_1,u_2\}$ and $\{u_0,u_3,u_4\}$ are collinear in $(G+F+e_1+e_2,\bp')$ to deduce that 
$(G+F,\bp')$ is minimally $C_2^1$-rigid, see the proof of \cite[Theorem 10.3.1]{Wsurvey} for more details. Since $\bp$ is generic, $\bp'$ is non-degenerate  on $\hat N_G(v_0)$ and hence $\bp'\in S$.
Since $(G,\bp)$ has a bad motion at $v_0$,
Lemma~\ref{lem:bad_pinning} now implies that $(G,\bp')$ has a bad motion $\bq_{\rm bad}'$ at $v_0$.
%
The definition of $\bp'$ implies  that 
$\bq_{\rm bad}'|_{V_0}$ is a motion of $(G_0+e_1+e_2,\bp|_{V_0})$.  
Since $(G_0+e_1+e_2,\bp|_{V_0})$ is a generic 2-dof framework, 
$\bq_{\rm bad}'|_{V_0}$ is bad at $v_0$ and this
contradicts the minimality of $G$.
\end{proof}

Claim~\ref{claim:2mechanism3} implies the following.

\begin{claim}\label{smallclaim2}
Let $F\in {F_v\choose 2}$ and let $e=u_iu_j\in K_u\setminus E_0$ such that $G_0+F+e$ is $C_2^1$-independent.
Then ${\rm cl}((E_0\cup F)+e)$ contains a triangle on $N_G(u_0)\setminus \{u_i, u_j\}$.
\end{claim}
\begin{proof}
If not, then we can choose an edge $e'$ on $N_G(u_0)\setminus \{u_i, u_j\}$ such that $G_0+F+e+e'$ is $C_2^1$-independent. 
This contradicts Claim~\ref{claim:2mechanism3}.
\end{proof}


\begin{claim}\label{claim:2mechanism4}
$G_0+F+e_1+e_2$ is $C_2^1$-dependent for all $F\in {F_v\choose 2}$
and all pairs of adjacent edges  $e_1, e_2\in K_u\setminus E_0$ whose common end-vertex has degree two in ${\rm cl}(E_0)\cap K_u$.
\end{claim}
\begin{proof}
Suppose that 
$G_0+F+e_1+e_2$ is $C^1_2$-independent (and hence minimally $C_2^1$-rigid). Our aim is to contradict the minimality of $G$ by showing that $(G_0+e_1+e_2,\bp|_{V_0})$ has a bad motion at $v_0$. 

Let $u_1$ be the common endvertex of $e_1$ and $e_2$. Since $u_1$ has degree two in $\cl(E_0)\cap K_u$ we may
choose a base $E_0'$ of $\cl(E_0)$ such that
$u_1$ has degree two in $E_0'\cap K_u$ and $v_0$ is incident with the same set of edges in $E_0$ and $E_0'$. Let $G_0'$ be the graph induced by $E_0'$ and $G'$ be the graph obtained from $G_0'$ by adding $u$ and the edges of $G$ incident to $u$. Then $G_0'+F+e_1+e_2$ is $C^1_2$-independent and hence $G'+F$ is $C^1_2$-independent by Lemma \ref{lem:V}. In addition, the fact that $\cl(E_0)=\cl(E_0')$ implies that $(G,p)$ and $(G',p)$ will have the same space of motions. Relabelling $G'$ by $G$, we may assume that  $u_1$  has degree two in $E_0\cap K_u$ and
$G+F$ is obtained from $G_0+F+e_1+e_2$ by a vertex splitting. 

Consider the realisation $\bp_0$ of $G$ 
defined by 
$\bp_0(w)=\bp(w)$ for $w\in V_0$
and $\bp_0(u_0)=\bp(u_1)$. 
Let
${C}^*(G+F,\bp_0)$ be obtained from the cofactor matrix ${C}(G+F,\bp_0)$
by replacing the zero-row indexed by the edge $u_0u_1$ with a row of the form 
$$
\kbordermatrix{
 & & u_0 & & u_1 & \\
 e=u_0u_1 & 0\dots 0 & D(\bp(u_0),\bp(u_1)) & 0 \dots 0 & -D(\bp(u_0),\bp(u_1)) & 0\dots 0 
}.
$$
The  argument used to show that vertex splitting  preserves $C^1_2$-independence in  \cite[Theorem 10.2.7]{Wsurvey} implies that ${C}^*(G+F,\bp_0)$ is row independent.
If ${C}^*(G+F,\bp_0)$ were the $C^1_2$-cofactor matrix of $(G+F,\bp_0)$ then we could find a bad motion for $(G_0+e_1+e_2,\bp|_{V_0})$ by applying Lemma~\ref{lem:bad_pinning} to $(G+F,\bp_0)$ as in the last paragraph of Claim \ref{claim:2mechanism3}, and this would contradict the minimality of $G$.  Since ${C}^*(G+F,\bp_0)$ is not a $C^1_2$-cofactor matrix 
we instead apply the proof technique of Lemma~\ref{lem:bad_pinning} directly to ${C}^*(G+F,\bp_0)$ to show that
\begin{equation}\label{eq:special}
\mbox{$z|_{V_0}$ is a bad motion of $(G_0+e_1+e_2,\bp|_{V_0})$ for some $z\in \ker {C}^*(G,\bp_0)$}
\end{equation}
and obtain the same contradiction. We give the proof of (\ref{eq:special}) immediately after the proof of  Lemma~\ref{lem:bad_pinning} in Appendix
\ref{sec:bad_pinning}.
\end{proof}


\subsubsection{\boldmath The structure of  ${\rm cl}(E_0)\cap K_u$}\label{subsubsec:3}
We show that ${\rm cl}(E_0)\cap K_u$ is a star on $N_G(u_0)$.
The proof proceeds as in that of Lemma~\ref{lem:combinatorial}.

\begin{claim}\label{claim:2mechanism5}
For every good pair $F\in {F_v\choose 2}$ and every $e\in K_u$ such that $(E_0\cup F)+e$ is independent,
${\rm cl}((E_0\cup F)+e)\cap K_u$ does not contain a $C_2^1$-rigid subgraph spanning $N_G(u_0)$.
\end{claim}
\begin{proof}
Suppose ${\rm cl}((E_0\cup F)+e)\cap K_u$ contains a $C_2^1$-rigid subgraph spanning $N_G(u_0)$. Then 
$K_u\subseteq {\rm cl}((E_0\cup F)+e)$
and hence $u_0u_4, u_0u_5\in \cl((E_0\cup F)+e+u_0u_1+u_0u_2+u_0u_3)$. 
Since $G$ is $C_2^1$-independent,  we can use two base exchanges to show that $\cl((E_0\cup F)+e+u_0u_1+u_0u_2+u_0u_3)=\cl(E+e')$ for some $e'\in F+e$. Then  $F\subseteq \cl(E+e')$. This contradicts the hypothesis that $F$ is good. 
\end{proof}

For $F\subseteq K(V_0)\setminus \cl(E_0)$, we use
$[F]$ to denote the closure of $F$ in ${\cal C}_{n-1}/{\rm cl}(E_0)$.
Thus an edge $e\in K(V_0)\setminus {\rm cl}(E_0)$ belongs to $[F]$ if and only if we have $e\in C\subseteq (E_0\cup F)+e$ for some circuit $C$ of ${\cal C}_{n-1}$.
If $F$ is a singleton $\{f\}$, then we denote $[\{f\}]$ by $[f]$.

We say that a set $F\subset F_v$ is {\em influential} if $[F]\cap K_u\neq \emptyset$.
A set $F+e$ with $F\in {F_v\choose 2}$ and $e\in K_u\setminus {\rm cl}(E_0)$ is said to be a {\em stabilizer} if $\hat{N}(v_0)$ is a locally $C_2^1$-rigid part of $G_0+F+e$.
By Claim~\ref{claim:2mechanism1}, $F+e$ is a stabilizer if and only if 
${\rm cl}((E_0\cup F)+e)={\rm cl}(E_0\cup F_v)$.
This fact will be used frequently.

\begin{claim}\label{claim:2mechanism6}
Let $F\in {F_v\choose 2}$ be a good pair and $e\in K_u\setminus {\rm cl}(E_0)$ such that $(E_0\cup F)+e$ is independent.
Suppose that  either $F$ is not influential or $F+e$ is a stabilizer. 
Then ${\rm cl}((E_0\cup F)+e)\cap K_u$ contains no copy of $K_4$.
\end{claim}
\begin{proof}
Without loss of generality, let $e=u_1u_2$.

Suppose that ${\rm cl}((E_0\cup F)+e)\cap K_u$ contains a copy of $K_4$.
If ${\rm cl}((E_0\cup F)+e)\cap K_u$ contains two distinct copies of $K_4$, then
their union would be a $C^1_2$-rigid subgraph of $K_u$. This would contradict 
Claim~\ref{claim:2mechanism5}.
Hence ${\rm cl}((E_0\cup F)+e)\cap K_u$ contains exactly one copy of $K_4$.
We consider two cases depending on the position of the copy of $K_4$.

\medskip
\noindent
\emph{Case 1:} The copy of $K_4$ in ${\rm cl}((E_0\cup F)+e)\cap K_u$ contains $e=u_1u_2$.\\
By symmetry we may assume that this is a $K_4$ on $\{u_1, u_2, u_3, u_4\}$.
By Claim~\ref{smallclaim2}, ${\rm cl}((E_0\cup F)+e)\cap K_u$ also contains $\{u_3u_5, u_4u_5\}$.
Hence, by Claim~\ref{claim:2mechanism5}, 
$$u_1u_5, u_2u_5\notin {\rm cl}((E_0\cup F)+e)\cap K_u.$$

We next show that:
\begin{equation}
\label{eq:2-6-1}
u_1u_3, u_1u_4\in {\rm cl}(E_0\cup F).
\end{equation}
To see this, suppose $u_1u_3\notin {\rm cl}(E_0\cup F)$. Since $u_1u_3\in {\rm cl}((E_0\cup F)+e)$, we have that
$(E_0\cup F)+u_1u_3$ is independent and $u_2u_5\not\in{\rm cl}((E_0\cup F)+e)={\rm cl}((E_0\cup F)+u_1u_3)$.
Hence $(E_0\cup F)+u_1u_3+u_2u_5$ is independent.
This contradicts Claim~\ref{claim:2mechanism3} so  $u_1u_3\in {\rm cl}(E_0\cup F)$. The same argument for $u_1u_4$ gives (\ref{eq:2-6-1}).  

Since $(E_0\cup F)+e+u_1u_5$ is independent,
Claim~\ref{claim:2mechanism4} implies that $\{u_1u_3,u_1u_4\}\not \subseteq {\rm cl}(E_0)$.
By symmetry, we may suppose that $u_1u_3\notin {\rm cl}(E_0)$, i.e., $u_1u_3\in {\rm cl}((E_0\cup F))\setminus {\rm cl}(E_0)=[F]$.
This implies that $F$ is influential, and hence (by an hypothesis of the claim) $F+e$ is a stabilizer.
Thus, for the unique edge  $f_e\in F_v\setminus F$, we have ${\rm cl}((E_0\cup F)+e)={\rm cl}((E_0\cup F)+f_e)$. 
Since $u_1u_3\in [F]$, we also have 
$u_2u_5\not\in{\rm cl}((E_0\cup F)+e)={\rm cl}((E_0\cup F)+f_e)={\rm cl}(E_0+f+u_1u_3+f_e)$ for some edge $f\in F$. 
Then $(E_0\cup \{f,f_e\})+u_1u_3+u_2u_5$ is independent.
This contradicts Claim~\ref{claim:2mechanism3}.

\medskip
\noindent
\emph{Case 2:} ${\rm cl}((E_0\cup F)+e)\cap K_u$ contains a copy of $K_4$ which avoids  $e=u_1u_2$.\\
By symmetry we may assume that this is a $K_4$ on $\{u_1, u_3, u_4, u_5\}$.
Since there is no other copy of $K_4$ in ${\rm cl}((E_0\cup F)+e)\cap K_u$, 
we may also assume 
$$u_2u_3, u_2u_4\notin {\rm cl}((E_0\cup F)+e).$$

We proceed in a similar way to Case 1. We first show: 
\begin{equation}
\label{eq:2-6-2}
\{u_1u_3, u_1u_4, u_1u_5, u_3u_5, u_4u_5\}\subseteq  {\rm cl}(E_0\cup F).
\end{equation}
To see this, suppose $u_1u_3\notin {\rm cl}(E_0\cup F)$. 
Since $u_1u_3\in {\rm cl}((E_0\cup F)+e)$, 
$(E_0\cup F)+u_1u_3$ is independent and  $u_2u_4\notin {\rm cl}((E_0\cup F)+e)={\rm cl}((E_0\cup F)+u_1u_3)$.
Hence $(E_0\cup F)+u_1u_3+u_2u_4$ is independent.
This contradicts Claim~\ref{claim:2mechanism3}.
The same proof can be applied to the other edges to give (\ref{eq:2-6-2}).

We next show that:
\begin{equation}
\label{eq:2-6-3}
\{u_1u_3, u_1u_4, u_1u_5, u_3u_5, u_4u_5\}\subseteq {\rm cl}(E_0).
\end{equation}
To see this, suppose $u_1u_3\notin {\rm cl}(E_0)$. Then $u_1u_3\in {\rm cl}((E_0\cup F))\setminus {\rm cl}(E_0)=[F]$.
This implies that $F$ is influential, and hence (by an hypothesis of the claim) $F+e$ is a stabilizer.
Thus, for the unique edge  $f_e\in F_v\setminus F$, ${\rm cl}((E_0\cup F)+e)={\rm cl}((E_0\cup F)+f_e)$ holds, and we have 
$u_2u_4\not\in {\rm cl}((E_0\cup F)+e)={\rm cl}((E_0\cup F)+f_e)={\rm cl}(E_0+f+u_1u_3+f_e)$ for some edge $f\in F$. 
Hence $(E_0\cup \{f,f_e\})+u_1u_3+u_2u_4$ is independent.
This contradicts Claim~\ref{claim:2mechanism3}.
 The same argument can be applied to the other edges to give  (\ref{eq:2-6-3}).
 
Suppose that $u_3u_4\notin {\rm cl}(E_0 \cup F)$.
As $u_3u_4\in {\rm cl}((E_0\cup F)+e)$, 
$(E_0\cup F)+u_3u_4$ is independent and $u_2u_4\not\in {\rm cl}((E_0\cup F)+e)={\rm cl}((E_0\cup F)+u_3u_4)$.
Hence $(E_0\cup F)+u_3u_4+u_2u_4$ is independent. Since $u_4$ has degree two in ${\rm cl}(E_0)\cap K_u$ by
(\ref{eq:2-6-3}),
this contradicts Claim~\ref{claim:2mechanism4}. Hence $u_3u_4\in {\rm cl}(E_0 \cup F)$.
This and (\ref{eq:2-6-2})  imply that ${\rm cl}(E_0 \cup F)$ contains a copy of $K_4$ on 
$\{u_1, u_3, u_4, u_5\}$, contradicting Claim~\ref{claim:2mechanism2}.
\end{proof}

\begin{claim}\label{claim:mechanism7}
Let $F\in {F_v\choose 2}$ be a good pair  and $e=u_iu_j\in K_u$ be such that $(E_0\cup F)+e$ is independent.
Suppose that either $F$ is not influential or $F+e$ is a stabilizer.
Then ${\rm cl}(E_0)$ contains at least two edges of the triangle on $N_G(u_0)\setminus \{u_i, u_j\}$.
\end{claim}
\begin{proof}
Without loss of generality suppose $e=u_1u_2$.
By Claim~\ref{smallclaim2}, ${\rm cl}(E_0+F+e)$ contains a triangle on $\{u_3,u_4,u_5\}$.

We first consider the complete graph on $\{u_1, u_2, u_3, u_4\}$.
By Claim~\ref{claim:2mechanism6} we may assume without loss of generality that 
${\rm cl}((E_0\cup F)+e)$ does not contain $u_1u_3$.
If $u_4u_5\in {\rm cl}((E_0\cup F)+e)\setminus {\rm cl}(E_0\cup F)$, then 
$(E_0\cup F)+u_4u_5$ is independent and  ${\rm cl}((E_0\cup F)+u_4u_5)={\rm cl}((E_0\cup F)+e)$. This would imply that $(E_0\cup F)+u_1u_3+u_4u_5$ is independent and 
contradict Claim~\ref{claim:2mechanism3}.
Hence $u_4u_5\in {\rm cl}(E_0\cup F)$.

Suppose next that $u_4u_5\in {\rm cl}(E_0\cup F)\setminus {\rm cl}(E_0)$, i.e., $u_4u_5\in [F]$.
Then $F$ is influential and so, by the hypothesis, $F+e$ is a stabilizer. 
Hence ${\rm cl}((E_0\cup F)+e)={\rm cl}(E_0\cup F_v)={\rm cl}(E_0+f_1+f_2+u_4u_5)$ holds for some edges  $f_1, f_2\in F_v$.
Then $E_0+f_1+f_2+u_4u_5+u_1u_3$ is independent and we again 
contradict Claim~\ref{claim:2mechanism3}.
Hence $u_4u_5\in {\rm cl}(E_0)$.

By applying the 
same argument to the complete graph on $\{u_1, u_2, u_4, u_5\}$, 
we also get $u_3u_4\in {\rm cl}(E_0)$ or $u_3u_5\in {\rm cl}(E_0)$.
(Specifically, we have $u_3u_4\in {\rm cl}(E_0)$ 
when  $u_1u_5\notin {\rm cl}(E_0+f+e)$ or $u_2u_5\notin {\rm cl}((E_0\cup F)+e)$ holds,
and we have $u_3u_5\in {\rm cl}(E_0)$ when  $u_1u_4\notin {\rm cl}((E_0\cup F)+e)$ or $u_2u_4\notin {\rm cl}((E_0\cup F)+e)$ holds.)
\end{proof}

Recall that an edge $f\in F_v$ is said to be very good if $\{f, f'\}$ is a good pair for every $f'\in F_v\setminus \{f\}$.
In order to apply Claim~\ref{claim:2mechanism5} for a given edge $e\in K_u$, we need to find 
a good pair $F\subset F_v$ such that $(E_0\cup F)+e$ is independent.
The following claim shows we can do this under a mild assumption.

\begin{claim}\label{claim:verygood}
Let  $e\in K_u\setminus {\rm cl}(E_0)$, and suppose that $e\notin [f]$ for some very good edge $f\in F_v$.
Then there is a good pair $F\in {F_v\choose 2}$ such that $(E_0\cup F)+e$ is independent.
\end{claim}
\begin{proof}
Let $F_v=\{f_1, f_2, f_3\}$, and suppose that an edge $e\in K_u\setminus {\rm cl}(E_0)$ satisfies $e\notin [f_1]$ for a very good edge $f_1$.
Since $e\notin [f_1]$, $E_0+f_1+e$ is independent.
As $f_1$ is very good, $\{f_1, f_2\}$ and $\{f_1, f_3\}$ are both good.
Hence, if no good pair has the desired property, then $\{f_2, f_3\}\subset {\rm cl}(E_0+f_1+e)$.
By Claim~\ref{claim:2mechanism1}, $E_0+f_1+f_2$ is independent. Since $f_2\in {\rm cl}(E_0+f_1+e)$,
this gives ${\rm cl}(E_0+f_1+e)={\rm cl}(E_0+f_1+f_2)$,
and $f_3\in {\rm cl}(E_0+f_1+f_2)$ follows since $f_3\in {\rm cl}(E_0+f_1+e)$.
This would contradict Claim~\ref{claim:2mechanism1}. 
\end{proof}

Our next claim extends  Claim~\ref{claim:2mechanism6}.
(Note that $F$ is not required to be a good pair in the following claim.)
\begin{claim}\label{claim:2K4}
${\rm cl}(E_0\cup F_v)\cap K_u$ has no copy of $K_4$ and, for any stabilizer $F+e$, 
${\rm cl}((E_0\cup F)+e)\cap K_u$ has no copy of $K_4$.
\end{claim}
\begin{proof}
Suppose that a copy of $K_4$ exists in ${\rm cl}(E_0\cup F_v)\cap K_u$.
By Claim~\ref{claim:good} we can choose a very good edge $f$ from $F_v$.
Then Claim~\ref{claim:2mechanism2} implies that ${\rm cl}(E_0+f)\cap K_u$ has no copy of $K_4$.
Hence there is an edge $e$ in the copy of $K_4$ with $e\notin {\rm cl}(E_0+f)$.
By Claim~\ref{claim:verygood}, there exists a good pair $F$ of edges in $F_v$ such that $(E_0\cup F)+e$ is independent.
Since $e\in {\rm cl}(E_0\cup F_v)$ and $(E_0\cup F)+e$ is independent,
we have $F_v\subset {\rm cl}(E_0\cup F_v)={\rm cl}((E_0\cup F)+e)$. (Note that $E_0\cup F_v$ is independent by Claim~\ref{claim:2mechanism1}.)
Now $F_v\subset {\rm cl}((E_0\cup F)+e)$ implies that $F+e$ is a stabilizer while 
${\rm cl}(E_0\cup F_v)={\rm cl}((E_0\cup F)+e)$ implies that  ${\rm cl}((E_0\cup F)+e)$ has a copy of $K_4$.
This contradicts Claim~\ref{claim:2mechanism6}, and completes the proof of the first part of the claim.
The second part now follows from the fact that ${\rm cl}(E_0\cup F_v)={\rm cl}((E_0\cup F)+e)$ when $F+e$ is a stabilizer. 
\end{proof}

\begin{claim}\label{claim:2mechanism8}
Let $e=u_iu_j\in K_u\setminus {\rm cl}(E_0)$ be such that $e\notin [f]$ for some very good edge $f\in F_v$.
Then ${\rm cl}(E_0)$ contains at least two edges of the triangle on $N_G(u_0)\setminus \{u_i, u_j\}$.
\end{claim}
\begin{proof}
We may assume, without loss of generality, that $e=u_1u_2$.
Suppose that two edges on $\{u_3,u_4,u_5\}$ are missing in ${\rm cl}(E_0)$.
By symmetry, we may assume that $u_3u_5$ and $u_4u_5$ are missing.

We first show that for any  $F\in {F_v\choose 2}$ with $(E_0\cup F)+e$ independent, we have
\begin{equation}\label{eq:2-8-1}
\{u_1u_2, u_1u_3, u_1u_4, u_2u_3, u_2u_4, u_3u_4, u_3u_5, u_4u_5\}\subset {\rm cl}((E_0\cup F)+e).
\end{equation}
To see this, choose an $F\in {F_v\choose 2}$ such that $(E_0\cup F)+e$ is independent.
By Claim~\ref{smallclaim2}, ${\rm cl}((E_0\cup F)+e)$ contains a triangle on $\{u_3,u_4,u_5\}$.
To show that $\{u_1u_2,u_1u_3, u_1u_4, u_2u_3, u_2u_4\}$ is contained in  ${\rm cl}((E_0\cup F)+e)$, we first prove:
\begin{equation}
\label{eq:2-8-2}
\begin{split}
&\text{ if $u_ku_l\notin {\rm cl}(E_0)$ for some  $k,l \in \{3,4,5\}$, then} \\
&\text{${\rm cl}((E_0\cup F)+e)$ contains a triangle on $N(u_0)\setminus \{u_k,u_l\}$.}
\end{split}
\end{equation}
This follows from Claim~\ref{smallclaim2} if $u_ku_l\not\in {\rm cl}(E_0\cup F)$.
Hence we assume $u_ku_l\in {\rm cl}(E_0\cup F)\setminus {\rm cl}(E_0)=[F]$.
By Claim~\ref{claim:2mechanism1}, the unique edge $f'$ in $F_v\setminus F$ satisfies $f'\notin [F]$.
Also, since $u_ku_l\in [F]$, we can choose $f''\in F$ such that  ${\rm cl}(E_0\cup F)={\rm cl}(E_0+f''+u_ku_l)$. Since $f'\notin {\rm cl}(E_0\cup F)$, this implies that $(E_0\cup \{f',f''\})+u_ku_l$ is independent.
If $e=u_1u_2\notin {\rm cl}((E_0\cup \{f',f''\})+u_ku_l)$, then $(E_0\cup \{f',f''\})+u_ku_l+e$ would be independent, which  contradicts Claim \ref{claim:2mechanism3}.
Hence $e\in  {\rm cl}((E_0\cup \{f',f''\})+u_ku_l)$. 
As $[\{f'',u_ku_l\}]=[F]$ and $(E_0\cup F)+e$ is independent, 
we have ${\rm cl}((E_0\cup \{f',f''\})+u_ku_l)={\rm cl}((E_0\cup F)+e)$.
By Claim~\ref{smallclaim2}, ${\rm cl}((E_0\cup \{f',f''\})+u_ku_l)$ contains a triangle on $N_G(u_0)\setminus \{u_k, u_l\}$. This
in turn implies (\ref{eq:2-8-2}), since ${\rm cl}((E_0\cup \{f',f''\})+u_ku_l)={\rm cl}((E_0\cup F)+e)$.

The assumption that $u_3u_5\not\in{\rm cl}(E_0)$ and  (\ref{eq:2-8-2}) imply that 
$u_1u_2, u_1u_4, u_2u_4\in {\rm cl}((E_0\cup F)+e)$. 
Similarly, $u_4u_5\not\in{\rm cl}(E_0)$ and  (\ref{eq:2-8-2}) imply that 
$u_1u_3, u_2u_3\in {\rm cl}((E_0\cup F)+e)$. 
Hence (\ref{eq:2-8-1}) holds.
In particular, ${\rm cl}((E_0\cup F)+e)$ contains a copy of $K_4$, and 
Claim~\ref{claim:2K4} now gives:
\begin{equation}\label{eq:2-8-3}
\mbox{for any $F\in {F_v\choose 2}$ such that $(E_0\cup F)+e$ is independent, $F+e$ is not a stabilizer.}
\end{equation}

Let $F_v=\{f_1,f_2,f_3\}$. 
By Claim~\ref{claim:verygood} there is a good pair $F$ such that $(E_0\cup F)+e$ is independent.
We may assume without loss of generality that  $F=\{f_1, f_2\}$.
We next prove:
\begin{equation}\label{eq:2-8-3-2}
\text{$E_0+f_1+f_3+e$ and $E_0+f_2+f_3+e$ are both independent.}
\end{equation}

Suppose $E_0+f_1+f_3+e$ is dependent. Then, since $E_0+f_1+e$ is independent,  $f_3\in {\rm cl}(E_0+f_1+e)$.
This in turn implies that $f_1+f_2+e$ is a stabilizer, contradicting (\ref{eq:2-8-3}).
Hence $E_0+f_1+f_3+e$ is independent. A similar argument for 
$E_0+f_2+f_3+e$ 
gives (\ref{eq:2-8-3-2}).

We next show that:
\begin{equation}\label{eq:2-8-4}
S:=\{u_1u_2, u_1u_3, u_1u_4, u_2u_3, u_2u_4, u_3u_4, u_3u_5, u_4u_5\}\subset {\rm cl}(E_0+e).
\end{equation}
To see this suppose $u_iu_j\notin {\rm cl}(E_0+e)$ for some edge $u_iu_j\in S$.
Since $u_iu_j\in {\rm cl}(E_0+f_1+f_2+e)$ by (\ref{eq:2-8-1})  but $u_iu_j\notin {\rm cl}(E_0+e)$, 
we may assume without loss of generality that ${\rm cl}(E_0+f_1+f_2+e)={\rm cl}(E_0+f_1+e+u_iu_j)$.
Then $f_2\in {\rm cl}(E_0+f_1+e+u_iu_j)$.
We also have $u_iu_j\in {\rm cl}(E_0+f_1+f_3+e)$ by (\ref{eq:2-8-1}) and (\ref{eq:2-8-3-2}).
This gives ${\rm cl}(E_0+f_1+f_3+e)={\rm cl}(E_0+f_1+f_3+e+u_iu_j)$.
Since $f_2\in {\rm cl}(E_0+f_1+f_3+e+u_iu_j)$, this in turn implies that $f_1+f_3+e$ is a stabilizer, contradicting (\ref{eq:2-8-3}).
Thus (\ref{eq:2-8-4}) holds.

Since $F$ is a good pair,
Claim~\ref{claim:2mechanism5} implies that ${\rm cl}((E_0\cup F)+e)\cap K_u$ cannot contain a $C_2^1$-rigid graph which spans $N_G(u_0)$. 
Combining this fact with (\ref{eq:2-8-4}), we have 
\begin{equation}\label{eq:2-8-5}
{\rm cl}((E_0\cup F)+e)\cap K_u= {\rm cl}(E_0+e)\cap K_u=S.
\end{equation}
In particular, $u_1u_5$ and $u_2u_5$ are missing in ${\rm cl}((E_0\cup F)+e)$.
If $u_1u_3, u_1u_4\in {\rm cl}(E_0)$, then the fact that
$(E_0\cup F)+e+u_1u_5$ is independent would 
contradict Claim~\ref{claim:2mechanism4}.
Thus we may assume $u_1u_3\notin {\rm cl}(E_0)$.
Then $u_1u_3\in [e]$ by (\ref{eq:2-8-5}).
This implies that ${\rm cl}((E_0\cup F)+e)={\rm cl}((E_0\cup F)+u_1u_3)$, and hence $(E_0\cup F)+u_1u_3+u_2u_5$ is independent as $u_2u_5\notin {\rm cl}((E_0\cup F)+e)$.
This  contradicts Claim~\ref{claim:2mechanism3}.
\end{proof}

\begin{claim}\label{claim:2mechanism9}
For all $i$ with $1\leq i\leq 5$, $d_{{\rm cl}(E_0)\cap K_u}(u_i)\geq 1$.
And if $d_{{\rm cl}(E_0)\cap K_u}(u_i)=1$, then the vertex $u_j$ adjacent to $u_i$ in ${\rm cl}(E_0)\cap K_u$ satisfies 
$d_{{\rm cl}(E_0)\cap K_u}(u_j)=4$.
\end{claim}
\begin{proof}
Let $F_v=\{f_1, f_2, f_3\}$. By Claim~\ref{claim:good}, we may assume that $f_1$ is a very good edge.

 Suppose that $d_{{\rm cl}(E_0)\cap K_u}(u_5)=0$.
 By Claim~\ref{claim:2mechanism2}, ${\rm cl}(E_0+f_1)\cap K_u$ has no copy of $K_4$.
 Hence we can choose an edge $e$ on  $\{u_1,u_2,u_3,u_4\}$ such that $e\notin {\rm cl}(E_0+f_1)$.
We can then apply Claim~\ref{claim:2mechanism8} to $e$ to deduce that
 ${\rm cl}(E_0)$ has an edge incident to $u_5$. This gives a contradiction and completes the proof of first part of the claim.

 To prove the second part we assume that $d_{{\rm cl}(E_0)\cap K_u}(u_5)=1$ and
 that $u_5$ is adjacent to $u_1$ in ${\rm cl}(E_0)$. We will show that 
 $u_1u_2,u_1u_3,u_1u_4\in {\rm cl}(E_0)$.
 Suppose,  for a contradiction, that 
 $u_1u_2\not\in {\rm cl}(E_0)$.
 If $u_1u_2\notin [f_1]$, then Claim~\ref{claim:2mechanism8} would imply that $u_5$ has degree at least two in ${\rm cl}(E_0)$, a contradiction.
 Hence $u_1u_2\in [f_1]$ which gives $u_1u_2\in {\rm cl}(E_0+f_1)$. 
 We aslo have $u_1u_3, u_1u_4\in {\rm cl}(E_0+f_1)$ since, if $u_1u_i\not\in {\rm cl}(E_0)$ for some $i=1,2$, then we can use the same argument as for $u_1u_2$ to deduce that $u_1u_i\in [f_1]$.
 The facts that $u_1u_2\in [f_1]$ and  $F_v$ is independent in ${\cal C}_{n-1}/{\rm cl}(E_0)$ imply that  $E_0+f_2+f_3+u_1u_2$ is independent and  ${\rm cl}(E_0\cup F_v)={\rm cl}(E_0+f_2+f_3+u_1u_2)$.
Since $u_3u_4, u_4u_5, u_5u_3\in {\rm cl}(E_0+f_2+f_3+u_1u_2)$  by Claim~\ref{smallclaim2},
${\rm cl}(E_0\cup F_v)$ contains a copy of $K_4$ on $\{u_1,u_3,u_4,u_5\}$.
This contradicts Claim~\ref{claim:2K4}.
 \end{proof}

 \begin{claim}
 \label{claim:2mechanism10}
 ${\rm cl}(E_0)\cap K_u$ is cycle free.
 \end{claim}
 \begin{proof}
We choose $F\in {F_v\choose 2}$ and $e\in K_u$ as follows.
\\[-6mm]
\begin{itemize}
\item When $[F_v]\cap K_u\neq \emptyset$, we first choose $e\in [F_v]\cap K_u$
and then choose $F$ such that $F+e$ is independent. Such a choice of $F$ is possible by Claim \ref{claim:2mechanism1} and will ensure that $F+e$ is a stabilizer.
\\[-4mm]
\item
When  $[F_v]\cap K_u= \emptyset$, we first choose $F$ to be good and then choose $e\in K_u\setminus {\rm cl}(E_0)$. This choice will ensure that  $(E_0\cup F)+e$ is independent and $F$ is good but not influential.
\\[-4mm]
\end{itemize}
This choice of $F$ and $e$  and Claims~\ref{claim:2mechanism6} and \ref{claim:2K4} imply  that 
\begin{equation}\label{eq:noK4}
\text{${\rm cl}((E_0\cup F)+e)\cap K_u$ contains no copy of $K_4$.}
\end{equation}
  
 Suppose that ${\rm cl}(E_0)\cap K_u$ contains a cycle of length five, say 
 $u_1u_2u_3u_4u_5u_1$.
 We may assume  without loss of generality that $e=u_1u_3$.
 By (\ref{eq:noK4}), either
 $u_2u_4\notin {\rm cl}((E_0\cup F)+e)$ or $u_1u_4\notin {\rm cl}((E_0\cup F)+e)$ holds. Both alternatives lead to a  contradiction by applying 
 Claim~\ref{claim:2mechanism3}
and \ref{claim:2mechanism4}, respectively.

 Suppose that ${\rm cl}(E_0)\cap K_u$ contains a cycle of length four, say 
$u_1u_2u_3u_4u_1$. 
If $e$ is a chord of the cycle, then the other chord is missing in ${\rm cl}((E_0\cup F)+e)$ by (\ref{eq:noK4}),
 and we can apply Claim~\ref{claim:2mechanism3} to get a contradiction.
 Hence $e$ is not a chord of the cycle.
Since  at least one diagonal of the cycle is missing in ${\rm cl}((E_0\cup F)+e)$ by (\ref{eq:noK4}),  we can again apply Claim~\ref{claim:2mechanism3} or Claim~\ref{claim:2mechanism4} to get a contradiction.

It remains to consider the case when  ${\rm cl}(E_0)\cap K_u$ contains a cycle of length three, say 
$u_1u_2u_3u_1$.
Claim~\ref{claim:2mechanism9} and (\ref{eq:noK4}) tell us that there is exactly one edge in ${\rm cl}(E_0)\cap K_u$ from each of $u_4, u_5$ to $\{u_1, u_2, u_3\}$, and that both of these edges have a common end-vertex, say $u_1$.
By symmetry, we may assume that either $e=u_2u_4$ or $e=u_4u_5$ holds.
If $e=u_2u_4$, then $u_3u_5\in {\rm cl}((E_0\cup F)+e)$ by Claim~\ref{smallclaim2},
and $u_2u_5\notin {\rm cl}((E_0\cup F)+e)$ by (\ref{eq:noK4}).
Then $(E_0\cup F)+e+u_2u_5$ is independent, 
contradicting Claim~\ref{claim:2mechanism4}.
Hence $e=u_4u_5$. 

Recall our selection of  $e$ at the beginning of the proof.
If $[F_v]\cap K_u=\emptyset$, then $e$ could have been chosen to be any edge in $K_u\setminus {\rm cl}(E_0)$. In particular we could choose $e=u_2u_4$, and obtain the contradiction in the previous paragraph.
Hence  $[F_v]\cap K_u\neq \emptyset$. Since $e$ can be any edge in $[F_v]\cap K_u$, we must have $[F_v]\cap K_u=\{u_4u_5\}$.

Choose any good pair $F'\in {F_v\choose 2}$, and consider $F'+u_2u_4$. 
Since $[F_v]\cap K_u=\{u_4u_5\}$, $F'+u_2u_4$ is independent.
Hence by Claims~\ref{claim:2mechanism3} and
\ref{claim:2mechanism4},
$u_2u_5, u_3u_5\in {\rm cl}((E_0\cup F')+u_2u_4)$.
Since $u_3u_5\notin {\rm cl}(E_0\cup F')$, we have ${\rm cl}((E_0\cup F')+u_2u_4)={\rm cl}((E_0\cup F')+u_3u_5)$.
Then by Claim~\ref{claim:2mechanism4},
${\rm cl}((E_0\cup F')+u_3u_5)$ also contains $u_3u_4$.
Summarizing, we have shown that ${\rm cl}((E_0\cup F')+u_2u_4)$ contains a $C_2^1$-rigid graph which spans $N_G(u_0)$. This contradicts Claim~\ref{claim:2mechanism5}.
\end{proof}

Combining Claim~\ref{claim:2mechanism9} and \ref{claim:2mechanism10}, we conclude that 
${\rm cl}(E_0)\cap K_u$ is a star on five vertices.
Without loss of generality, we may assume that 
\begin{equation}
\label{eq:star}
\text{${\rm cl}(E_0)\cap K_u$ is a star centered at $u_5$.}
\end{equation}

\subsubsection{\boldmath Further structural properties of $G$}\label{subsubsec:4}

We 
let $K$ be the edge set of the complete graph on $\{u_1, u_2, u_3, u_4\}$ and, 
for each edge $e\in K$,  we refer to the unique edge of $K$ which is not adjacent to $e$ as the {\em opposite edge to} $e$ and denote it by $\tilde{e}$. For $e\in K$ and $f\in F_v$, we say that   {\em $e$ is $f$-coupled} if $\tilde e\in [\{e,f\}]$.

\begin{claim}\label{claim:2opposite1}
Let $e\in [F_v]\cap K$. Then $\tilde{e}\in [F_v]$. 
\end{claim}
\begin{proof}
Since $e\in  [F_v]$, $e\notin {\rm cl}(E_0)$ and we can choose $F\in {F_v\choose 2}$ such that $(E_0\cup F)+e$ is independent. Together  with 
Claim~\ref{smallclaim2}, this implies that $\tilde{e}\in {\rm cl}((E_0\cup F)+e)={\rm cl}(E_0\cup F_v)$. Since $\tilde e\not\in \cl(E_0)$ by (\ref{eq:star}), we have $\tilde e\in [F_v]$.
\end{proof}

\begin{claim}\label{claim:2opposite2}
Suppose $e$ is not $f$-coupled for some $e\in K$ and $f\in F_v$.
Then $e, \tilde{e}\in [F_v]$.
\end{claim}
\begin{proof}
Let $F_v=\{f_1, f_2, f_3\}$ where $f=f_1$.

Suppose $e\notin [F_v]$.
Then, by Lemma~\ref{claim:2mechanism1}, $E_0+f_i+f_j+e$ is independent for all distinct $i,j \in \{1,2,3\}$.
Hence, for each $i\in\{2,3\}$, Claim~\ref{smallclaim2} implies $\tilde{e}\in [\{f_1,f_i,e\}]$.
Since $\tilde{e}\notin [\{f_1,e\}]$ (as $e$ is not $f_1$-coupled), we have $f_i\in {\rm cl}(E_0+f_1+f_i+e)={\rm cl}(E_0+f_1+e+\tilde{e})$ for $i\in \{2,3\}$.
As $E_0\cup F_v$ is independent, this in turn implies that ${\rm cl}(E_0+f_1+e+\tilde{e})={\rm cl}(E_0\cup F_v)$, and
contradicts our initial assumption that $e\notin [F_v]$.

Hence $e\in [F_v]$.  Claim~\ref{claim:2opposite1} now implies that $\tilde{e}\in [F_v]$.
\end{proof}

\begin{claim}\label{claim:Fv}
Suppose $E_0+f+e+e'$ is independent for some 
$f\in F_v$ and distinct $e, e'\in [F_v]\cap K$. Then 
${\rm cl}(E_0+f+e+e')={\rm cl}(E_0+K_v)$.
\end{claim}
\begin{proof}
Recall that, by Claim~\ref{claim:2mechanism1},  $F_v$ is a base of $K_v\setminus {\rm cl}(E_0)$ in ${\cal C}_{n-1}/{\rm cl}(E_0)$.
Since $E_0+f+e+e'$ is independent and $f,e,e'\in [F_v]$, $\{f,e,e'\}$ 
spans $F_v$ in ${\cal C}_{n-1}/{\rm cl}(E_0)$.
This implies that ${\rm cl}(E_0+f+e+e')={\rm cl}(E_0\cup F_v)={\rm cl}(E_0\cup K_v)$.
\end{proof}

We will assume henceforth that  $F
=\{f_1, f_2\}
\in {F_v \choose 2}$ is a good pair. This is justified by Claim \ref{claim:good}.
We next categorize $G_0+f_i$ according to the properties of ${\rm cl}(E_0+f_i)$. We say that $G_0+f_i$  is:
\begin{description}
\item[Type 0] 
if $K_u\subseteq {\rm cl}(E_0+f_i+e_1+e_2)$ for some two edges $e_1, e_2\in K_u$.

\item[Type A] 
if the graphs
$A_1^i :=G_0+f_i+u_1u_2+u_2u_3$, $A_2^i:=G_0+f_i+u_2u_3+u_3u_1$, and $A_3^i:=G_0+f_i+u_3u_1+u_1u_2$ have the following properties.
\begin{itemize}
\item[(1)] Each $A_j^i$ is $C^1_2$-independent and has one degree of freedom. 
\item[(2)] Each ${\rm cl}(A_j^i)\cap K_u$ is as shown in Figure~\ref{fig:A}.
\item[(3)] For some (or equivalently, every) non-trivial motion $\bq_j^i$ of $(A_j^i, \bp)$, $1\leq j\leq 3$, we have $Z(G_0+f_i, \bp)=Z_0(G_0+f_i,\bp)\oplus \langle \bq_1^i, \bq_2^i, \bq_3^i\rangle$. 
\end{itemize}

\item[Type B] 
if the graphs
$B_1^i:=G_0+f_i+u_1u_2+u_3u_4$,  $B_2^i:=G_0+f_i+u_2u_3+u_3u_1$, and $B_3^i:=G_0+f_i+u_3u_1+u_1u_2$ have the following properties.
\begin{itemize}
\item[(1)] Each $B_j^i$ is $C^1_2$-independent and has one degree of freedom.  
\item[(2)] Each ${\rm cl}(B_j^i)\cap K_u$ is as shown in Figure~\ref{fig:B}.
\item[(3)] For some (or equivalently, every) non-trivial motion $\bq_j^i$ of $(B_j^i, \bp)$, $1\leq j\leq 3$, we have $Z(G_0+f_i, \bp)=Z_0(G_0+f_i,\bp)\oplus \langle \bq_1^i, \bq_2^i, \bq_3^i\rangle$.
\item[(4)] $K_v\subset {\rm cl}(B_1^i)$ and $\{u_1u_2,u_3u_4\}\subset [F_v]$.
\end{itemize}

\item[Type C] 
if  the graphs 
$C_1^i:=G_0+f_i+u_1u_2+u_2u_3$, $C_2^i:=G_0+f_i+u_2u_3+u_3u_1$, and $C_3^i:=G_0+f_i+u_3u_1+u_1u_2$ have the following properties.
\begin{itemize}
\item[(1)] Each $C_j^i$ is $C^1_2$-independent and has one degree of freedom. 
\item[(2)] Each ${\rm cl}(C_j^i)\cap K_u$ is as  shown in Figure~\ref{fig:C}.
\item[(3)] For some (or equivalently, every) non-trivial motion $\bq_j^i$ of $(C_j^i, \bp)$, $1\leq j\leq 3$, we have $Z(G_0+f_i, \bp)=Z_0(G_0+f_i,\bp)\oplus \langle \bq_1^i, \bq_2^i, \bq_3^i\rangle$.
\item[(4)] $K_v\subset {\rm cl}(C_1^i)$.
\end{itemize}

\item[Type D] 
if  the graphs
$D_1^i:=G_0+f_i+u_1u_2+u_3u_4$,  $D_2^i:=G_0+f_i+u_1u_3+u_3u_4$, and $D_3^i:=G_0+f_i+u_3u_1+u_1u_2$ have the following properties.
\begin{itemize}
\item[(1)] Each $D_j^i$ is $C^1_2$-independent and has one degree of freedom. 
\item[(2)] Each ${\rm cl}(D_j^i)\cap K_u$ is as  shown in Figure~\ref{fig:D}.
\item[(3)] For some (or equivalently, every) non-trivial motion $\bq_j^i$ of $(D_j^i, \bp)$, $1\leq j\leq 3$, we have $Z(G_0+f_i, \bp)=Z_0(G_0+f_i,\bp)\oplus \langle \bq_1^i, \bq_2^i, \bq_3^i\rangle$.
\item[(4)] $K_v\subset {\rm cl}(D_1^i)$.
\end{itemize}

\end{description}
In addition, for X,Y in $\{$A,B,C,D$\}$, we say that $(G_0,F)$ is {\em type} XY if $G_0+f_1$ is type X and $G_0+f_2$ is type Y.

\begin{figure}[p]
\begin{minipage}{0.32\textwidth}
\centering
\includegraphics[scale=0.8]{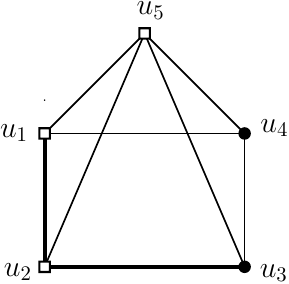}
\par
(i) ${\rm cl}(A_1^i)\cap K_u$
\end{minipage}
\begin{minipage}{0.32\textwidth}
\centering
\includegraphics[scale=0.8]{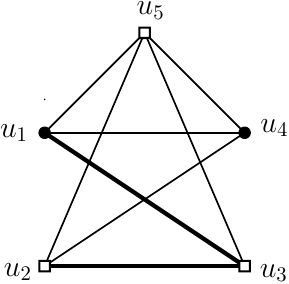}
\par
(ii) ${\rm cl}(A_2^i)\cap K_u$
\end{minipage}
\begin{minipage}{0.32\textwidth}
\centering
\includegraphics[scale=0.8]{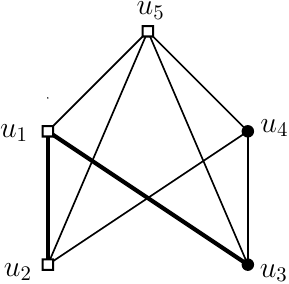}
\par
(iii) ${\rm cl}(A_3^i)\cap K_u$
\end{minipage}
\caption{${\rm cl}(A_j^i)\cap K_u$ for $1\leq j\leq 3$.}
\label{fig:A}
\end{figure}

\begin{figure}[p]
\begin{minipage}{0.32\textwidth}
\centering
\includegraphics[scale=0.8]{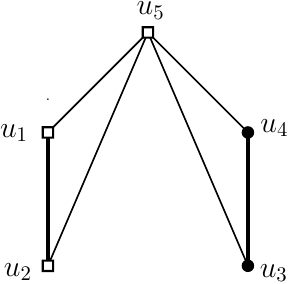}
\par
(i) ${\rm cl}(B_1^i)\cap K_u$
\end{minipage}
\begin{minipage}{0.32\textwidth}
\centering
\includegraphics[scale=0.8]{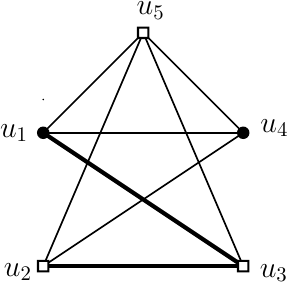}
\par
(ii) ${\rm cl}(B_2^i)\cap K_u$
\end{minipage}
\begin{minipage}{0.32\textwidth}
\centering
\includegraphics[scale=0.8]{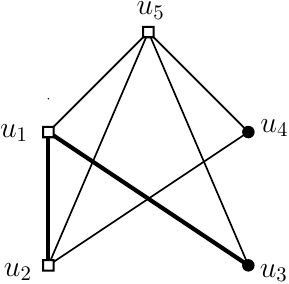}
\par
(iii) ${\rm cl}(B_3^i)\cap K_u$
\end{minipage}
\caption{${\rm cl}(B_j^i)\cap K_u$ for $1\leq j\leq 3$.}
\label{fig:B}
\end{figure}

\begin{figure}[p]
\begin{minipage}{0.32\textwidth}
\centering
\includegraphics[scale=0.8]{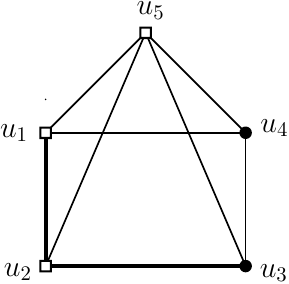}
\par
(i) ${\rm cl}(C_1^i)\cap K_u$
\end{minipage}
\begin{minipage}{0.32\textwidth}
\centering
\includegraphics[scale=0.8]{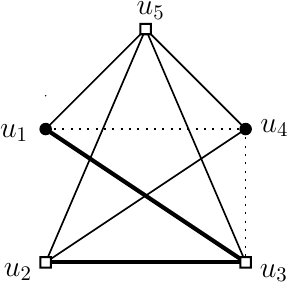}
\par
(ii) ${\rm cl}(C_2^i)\cap K_u$
\end{minipage}
\begin{minipage}{0.32\textwidth}
\centering
\includegraphics[scale=0.8]{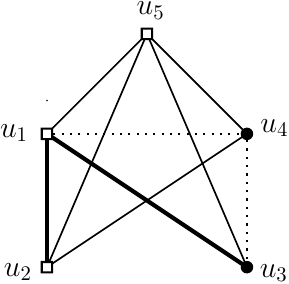}
\par
(iii) ${\rm cl}(C_3^i)\cap K_u$
\end{minipage}
\caption{${\rm cl}(C_j^i)\cap K_u$ for $1\leq j\leq 3$ (at most one dotted edge may exist in each graph).}
\label{fig:C}
\end{figure}

\begin{figure}[p]
\begin{minipage}{0.32\textwidth}
\centering
\includegraphics[scale=0.8]{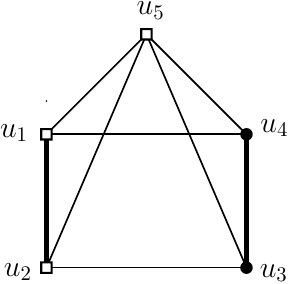}
\par
(i) ${\rm cl}(D_1^i)\cap K_u$
\end{minipage}
\begin{minipage}{0.32\textwidth}
\centering
\includegraphics[scale=0.8]{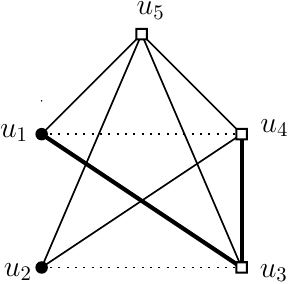}
\par
(ii) ${\rm cl}(D_2^i)\cap K_u$
\end{minipage}
\begin{minipage}{0.32\textwidth}
\centering
\includegraphics[scale=0.8]{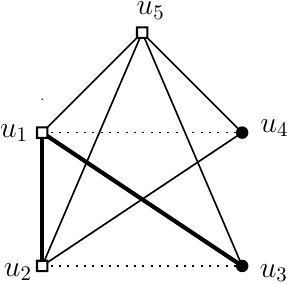}
\par
(iii) ${\rm cl}(D_3^i)\cap K_u$
\end{minipage}
\caption{${\rm cl}(D_j^i)\cap K_u$ for $1\leq j\leq 3$ (at most one dotted edge may exist in each graph).}
\label{fig:D}
\end{figure}

\begin{claim}\label{claim:types}
Suppose $F=\{f_1,f_2\}$ is a good pair and $f_i\in F$. Then, 
up to a possible relabeling of 
$u_1,u_2,u_3, u_4$, 
$G_0+f_i$ is type 0, A, B, C, or D, and, if neither $G_0+f_1$ nor $G_0+f_2$ is type 0, then $(G_0,F)$ is type AA, AB, BA, BB, CC or DD.
\end{claim}
\begin{proof}
Since the claim holds trivially if $G_0+f_i$ is type 0 for some $i\in \{1,2\}$, we may assume that this is not the case. 
Recall that $K$ is the edge set of the complete graph on $\{u_1,u_2,u_3, u_4\}$. 
By Claim~\ref{claim:2opposite1}, $|[F_v]\cap K|$ is even.
We will consider the different alternatives for $|[F_v]\cap K|$. 

Suppose that $|[F_v]\cap K|=6$.
Since $F$ is a good pair, Claim \ref{claim:2mechanism2} implies there is an edge $e\in K$ such that $(E_0\cup F)+e$ is independent. Since  $e\in K=[F_v]\cap K$, this and \eqref{eq:star} give  $K_u\subset {\rm cl}(E_0\cup F_v)={\rm cl}((E_0\cup F)+e)$. 
This contradicts Claim~\ref{claim:2mechanism5}.
Hence $|[F_v]\cap K|\in \{0, 2, 4\}$, and we split the proof accordingly.

\medskip

\noindent
\emph{Case 1:} $|[F_v]\cap K|=0$. 
We show $G_0+f_i$ is type A and hence $(G_0,F)$ is type AA.
Note that Claim~\ref{claim:2opposite2} and the fact that $[F_v]\cap K=\emptyset$ imply that 
$\tilde{e}\in [\{f_i, e\}]$ for all $e\in K$, so every $e\in K$ is $f_i$-coupled.

Let $H^i:=G_0+f_i+u_1u_2+u_2u_3+u_3u_1$.
If $H^i$ is $C_2^1$-dependent then, by symmetry, we may assume  that $u_3u_1\in {\rm cl}(E_0+f_i+u_1u_2+u_2u_3)$.
Since each $e\in K$ is $f_i$-coupled, this gives $K\subset {\rm cl}(E_0+f_i+u_1u_2+u_2u_3)$, and contradicts our initial assumption that  $G_0+f_i$ is not type 0.
Hence $H^i$ is $C_2^1$-independent.

Since each $A_j^i$ is a subgraph of $H^i$, each $A_j^i$ is a $C_2^1$-independent 1-dof graph.
The $C_2^1$-independence of $H^i$ and the fact that each $e\in K$ is $f_i$-coupled imply that
 ${\rm cl}(A_j^i)\cap K_u$ is as shown in Figure~\ref{fig:A}. This in turn implies that
$Z(G_0+f_i,\bp)=Z_0(G_0+f_i,\bp)\oplus \langle \bq_1^i,\bq_2^i,\bq_3^i\rangle$ for any non-trivial motions $\bq_j^i$ of $(A_j^i, \bp)$, $1\leq j\leq 3$.

\medskip
\noindent
\emph{Case 2:} $|[F_v]\cap K|=2$. By relabeling $u_1,u_2,u_3,u_4$ if necessary, we may assume that $[F_v]\cap K=\{u_1u_2, u_3u_4\}$.
Then by Claim~\ref{claim:2opposite2}, each of  the edges in $K\setminus [F_v]$ is $f_i$-coupled.
We will show $G_0+f_i$ is type A or type B, and hence $(G_0,F)$ is type AA, AB, BA or BB.
%

Let $H^i:=G_0+f_i+u_1u_2+u_2u_3+u_3u_1$.
Since $u_2u_3, u_3u_1$ are $f_i$-coupled and $G_0+f_i$ is not type 0, 
$H^i$ is $C_2^1$-independent.
If 
$u_1u_2$ is $f_i$-coupled, then we may use  the same proof as in  Case 1 to show that $G_0+f_i$ is type A.
Hence we may assume that  $u_3u_4\notin [\{f_i, u_1u_2\}]$. 

Since $B_2^i$ and $B_3^i$ are subgraphs of $H^i$, they are $C_2^1$-independent 1-dof graphs.
The assumption that $G_0+f_i$ is not type 0 now implies that  ${\rm cl}(B_2^i)\cap K_u$ and ${\rm cl}(B_3^i)\cap K_u$ are as shown in Figure~\ref{fig:B}(ii)(iii). 
(The edge $u_3u_4$ does not exist in $\cl(B^i_2)\cap K_u$ since $G_0+f_i$ is not type 0. The reason why the edge $u_3u_4$ does not exist in $\cl(B^i_3)\cap K_u$ will be clarified later.)
The graph $B_1^i$ is $C^1_2$-independent since $u_3u_4\notin [\{f_i, u_1u_2\}]$.
Also, by Claim~\ref{claim:Fv}, ${\rm cl}(B_1^i)={\rm cl}(E_0+f_i+u_1u_2+u_3u_4)={\rm cl}(E_0+K_v)$.
This gives $K_v\subset {\rm cl}(B_1^i)$. Since $[F_v]\cap K=\{u_1u_2, u_3u_4\}$, it also implies  that 
${\rm cl}(B_1^i)\cap K_u$ is as shown in Figure~\ref{fig:B}(i).
The fact that ${\rm cl}(B_1^i)\cap K_u$ is as shown in Figure~\ref{fig:B}(i) now tells us that
$u_3u_4$ does not exist in $\cl(B^i_3)\cap K_u$.

The assertion that 
$Z(G_0+f_i,\bp)=Z_0(G_0+f_i,\bp)\oplus \langle \bq_1^i,\bq_2^i,\bq_3^i\rangle$ for any non-trivial motions $\bq_j^i$ of $(B_j^i, \bp)$ easily follows from
the fact that ${\rm cl}(B_j^i)\cap K_u$ is as in Figure~\ref{fig:B}.

\medskip
\noindent
\emph{Case 3:} $|[F_v]\cap K|=4$. 
By relabeling $u_1,u_2,u_3,u_4$ if necessary, we may assume that $[F_v]\cap K=\{u_1u_2, u_2u_3, u_3u_4, u_4u_1\}$.
Then by Claim~\ref{claim:2opposite2}, $u_1u_3$ and $u_2u_4$  are $f_i$-coupled.
We will show that $G_0+f_i$ is type C or type D.
If $|[\{f_i, u_1u_2\}]\cap \{u_2u_3, u_3u_4, u_1u_4\}|\geq 2$ then,  since $u_1u_3$ is coupled, ${\rm cl}(E_0+f_i+u_1u_2+u_1u_3)$ would contain at least five edges from $K$.  This in turn would imply that $K_u\subset {\rm cl}(E_0+f_i+u_1u_2+u_1u_3)$
and contradict our initial assumption that $G_0+f_i$ is not type 0. Hence
\begin{equation}\label{eq:type_case3}
\left|[\{f_i, u_1u_2\}]\cap \{u_2u_3, u_3u_4, u_1u_4\}\right|\leq 1.
\end{equation}

%
%

Consider the following two subcases:

\medskip
\noindent
\emph{Subcase 3-1:} $[\{f_1, u_1u_2\}]\cap \{u_2u_3, u_1u_4\}=\{u_1u_4\}$ and $[\{f_2, u_1u_2\}]
\cap \{u_2u_3, u_1u_4\}=\{u_2u_3\}$.
We will show that  $(G_0,F)$ is type DD. 
We first show that  $J^i:=G_0+f_i+u_1u_2+u_1u_3+u_3u_4$ is $C_2^1$-independent for both $i=1,2$.
To see this, observe that $E_0+f_i+u_1u_2+u_3u_4$ is independent by (\ref{eq:type_case3}) and the assumption of subcase 3-1.
Hence, by Claim~\ref{claim:Fv}, 
\begin{equation}\label{eq:type_case3_1}
{\rm cl}(E_0+f_i+u_1u_2+u_3u_4)={\rm cl}(E_0\cup F_v)={\rm cl}(E_0\cup K_v).
\end{equation} 
Since $[F_v]\cap K=\{u_1u_2, u_2u_3, u_3u_4, u_4u_1\}$ we have $u_1u_3\notin {\rm cl}(E_0+f_i+u_1u_2+u_3u_4)$ and hence 
$J^i$ is $C_2^1$-independent.

Since each $D_j^i$ is a subgraph of $J^i$, each $D_j^i$ is a $C_2^1$-independent 1-dof graph.
We have $K_v\subset {\rm cl}(D_1^i)$ by (\ref{eq:type_case3_1}).
The assumption that $G_0+f_i$ is not type 0 and (\ref{eq:type_case3_1}) also imply that 
${\rm cl}(D_j^i)\cap K_u$ is  as shown in Figure~\ref{fig:D}. The assertion that
$Z(G_0+f_i,\bp)=Z_0(G_0+f_i,\bp)\oplus \langle \bq_1^i,\bq_2^i,\bq_3^i\rangle$ for any non-trivial motions $\bq_j^i$ of $(D_j^i, \bp)$, $1\leq i\leq 3$,  follows easily from
the fact that ${\rm cl}(D_j^i)\cap K_u$ is as in Figure~\ref{fig:D}.

\medskip
\noindent
\emph{Subcase 3-2:} Either $[\{f_i, u_1u_2\}]\cap \{u_2u_3, u_1u_4\}\subseteq \{u_1u_4\}$ for both $i=1,2$, 
or\\ $[\{f_i, u_1u_2\}]\cap \{u_2u_3, u_1u_4\}\subseteq \{u_2u_3\}$ for both $i=1,2$.\\
Relabeling if necessary, we may assume that the first alternative holds.
We will show that  $(G_0,F)$ is type CC. 
We first show that  $H^i:=G_0+u_1u_2+u_2u_3+u_3u_1$  is $C_2^1$-independent (for both $i=1,2$).
To see this, observe that $E_0+f_i+u_1u_2+u_2u_3$ is independent since  $[\{f_i, u_1u_2\}]\cap \{u_2u_3, u_1u_4\}\subseteq \{u_1u_4\}$.
Hence, by Claim~\ref{claim:Fv}, 
\begin{equation}\label{eq:type_case3_2}
{\rm cl}(E_0+f_i+u_1u_2+u_2u_3)={\rm cl}(E_0\cup F_v)={\rm cl}(E_0\cup K_v).
\end{equation} 
Since  $[F_v]\cap K=\{u_1u_2, u_2u_3, u_3u_4, u_4u_1\}$, we have $u_1u_3\notin {\rm cl}(E_0+f_i+u_1u_2+u_2u_3)$ and hence 
$H^i$ is $C_2^1$-independent.

Since each $C_j^i$ is a subgraph of $H^i$, each $C_j^i$ is a $C_2^1$-independent  1-dof graph.
We have $K_v\subset {\rm cl}(C_1^i)$  by (\ref{eq:type_case3_2}).
The assumption that $G_0+f_i$ is not type 0 and  (\ref{eq:type_case3_2}) imply that 
${\rm cl}(C_j^i)\cap K_u$ is  as shown in Figure~\ref{fig:C}.
The assertion that 
$Z(G_0+f_i,\bp)=Z_0(G_0+f_i,\bp)\oplus \langle \bq_1^i,\bq_2^i,\bq_3^i\rangle$ for any non-trivial motions $\bq_j^i$ of $(C_j^i, \bp)$  follows easily from
the fact that ${\rm cl}(C_j^i)\cap K_u$ is as in Figure~\ref{fig:C}.
%
\end{proof}

\subsubsection{The motion spaces of $(G_0+f_i,p|_{V_0})$ and $(G+f_i,p)$}
\label{subsubsec:5}

We will obtain expressions for a non-trivial motion of each of the 1-dof frameworks $(A^i_j,\bp|_{V_0})$,
$(B^i_j,\bp|_{V_0})$, $(C^i_j,\bp|_{V_0})$ and  $(D^i_j,\bp|_{V_0})$. 
We first apply a projective transformation to $(G,\bp)$ to ensure that
\begin{equation}\label{eq:quasi_generic}
\bp(u_1)=(1,0),\ \bp(u_2)=(0, 0),\ \bp(u_3)=(0,1),\ \bp(u_4)=(1,1),
\end{equation}
and all other points in $\bp(V_0)$ are generically placed.
Note that this transformation does not change the underlying $C_2^1$-cofactor matroid of $(G,\bp)$ by the projective invariance of $C_2^1$-rigidity~(see~\cite{Wsurvey} or Appendix~\ref{sec:projective}), and does not change the fact that $(G,\bp)$ has a bad motion at $v_0$ by Lemma \ref{lem:projective_bad}. 

For simplicity we use the notation 
$$D_{ij}=D(\bp(u_i),\bp(u_j)) \quad \text{ and } \quad \Delta_{ijk}=\Delta(\bp(u_i),\bp(u_j),\bp(u_k))$$ 
throughout the remainder of this section. 

Let $F=\{f_1,f_2\}$ be the fixed good pair, and suppose that $G_0+f_i$ is type A for some $i\in \{1,2\}$.
Since $(G,\bp)$ is 2-dof, $(G+f_i,\bp)$ is 1-dof.
Let $\bq^i$ be a non-trivial motion of  $(G+f_i,\bp)$. Then $\bq^i|_{V_0}$ is a motion of $(G_0+f_i,\bp|_{V_0})$ and 
property (3) in the definition of type A implies that $\bq^i|_{V_0}$ is a linear combination of $\bq_{1}^i,\bq_{2}^i,\bq_{3}^i$ and a trivial motion, where $\bq_j^i$ is a non-trivial motion of $(A_j^i, \bp|_{V_0})$.   By adding a suitable trivial motion to $\bq^i$, we may assume that  we have  
\begin{equation}\label{eq:uncoupled2}
\bq^i|_{V_0}=\sum_{j=1}^3\alpha_j^i \bq_j^i
\end{equation}
for some scalars $\alpha_j^i\in \mathbb{R}$.
\begin{claim}\label{claim:motionA}
Suppose that $G_0+f_i$ is type A. Then
for all $1\leq j\leq 3$, $(A_j^i,\bp|_{V_0})$ 
has a non-trivial motion $\bq_j^i$ satisfying $\bq_j^i(w)\in\mathbb{Q}(\bp(V_0))$ for all $w\in V_0$,
 and 
\begin{itemize}
\item $\bq_1^i(u_5)=\bq_1^i(u_1)=\bq_1^i(u_2)=0, \bq_1^i(u_3)=D_{3,2} \times D_{3,5}$,  $\bq_1^i(u_4)=t_1 D_{4,1} \times D_{4,5}$;
\item $\bq_2^i(u_5)=\bq_2^i(u_2)=\bq_2^i(u_3)=0, \bq_2^i(u_1)=D_{1,3} \times D_{1,5}$, $\bq_2^i(u_4)=t_2 D_{4,2} \times D_{4,5}$;
\item $\bq_3^i(u_5)=\bq_3^i(u_1)=\bq_3^i(u_2)=0, \bq_3^i(u_3)=D_{3,1} \times D_{3,5}$,  $\bq_3^i(u_4)=t_3 D_{4,2} \times D_{4,5}$;
\end{itemize}
where 
$t_1=\Delta_{523}/\Delta_{514}$, $t_2=-\Delta_{513}/\Delta_{524}$ and  $t_3=\Delta_{513}/\Delta_{524}$.

In addition, $(G+f_i,\bp)$ has  a non-trivial motion $\bq^i$ such that 
$\bq^i|_{V_0}=\sum_{j=1}^3\alpha_j^i \bq_j^i$ where:
\begin{align*}
\alpha_1^i&=-\Delta_{012}\Delta_{013}\Delta_{024}\Delta_{135}\Delta_{245}t_2 - \Delta_{013}^2\Delta_{024}\Delta_{135}^2 + \Delta_{013}\Delta_{023}\Delta_{024}\Delta_{135}\Delta_{245}t_3; \\
\alpha_2^i&=\Delta_{012}\Delta_{013}\Delta_{014}\Delta_{135}\Delta_{145}t_1 - \Delta_{012}\Delta_{023}\Delta_{024}\Delta_{235}\Delta_{245}t_3;\\
\alpha_3^i&=\Delta_{012}\Delta_{023}\Delta_{024}\Delta_{235}\Delta_{245}t_2 - \Delta_{013}\Delta_{014}\Delta_{023}\Delta_{135}\Delta_{145}t_1 + \Delta_{013}\Delta_{023}\Delta_{024}\Delta_{135}\Delta_{235}.
\end{align*}
\end{claim}
\begin{proof}
We will suppress the superscript $i$ in $A^i_j,\bq^i_j,\bq^i,\alpha^i_j$ as it remains constant throughout the proof.

Consider $(A_1,\bp|_{V_0})$. Since $\cl(A_1)$ contains the triangle on $\{u_1, u_2, u_5\}$ and $u_1u_3\not\in \cl(A_1)$,  
$(A_1,\bp|_{V_0})$ has a non-trivial motion $\bq_1$ such that $\bq_1(u_1)=\bq_1(u_2)=\bq_1(u_5)=0$ and $\bq_1(u_3)\neq 0$. Since ${\rm cl}(A_1)$ contains $u_2u_3$ and $u_3u_5$, $\bq_1(u_3)=s_1 D_{2,3}\times D_{3,5}$ for some scalar $s_1$. Since $\bq_1(u_3)\neq 0$ we can scale $\bq_1$, so that $s_1=1$.
Since ${\rm cl}(A_1)$ contains $u_1u_4$ and $u_4u_5$, $\bq_1(v_4)=t_1 D_{1,4}\times D_{4,5}$ for some scalar $t_1$.
Since $u_3u_4\in \cl(A_1)$, we have 
$$0=D_{34}\cdot (\bq_1(u_3)-\bq_1(u_4))=\left|\begin{array}{c} D_{34} \\ D_{32} \\ D_{35} \end{array}\right|- \left|\begin{array}{c} D_{43} \\ D_{41} \\ D_{45} \end{array}\right|t_1 \mbox{ so } t_1=\left|\begin{array}{c} D_{34} \\ D_{32} \\ D_{35} \end{array}\right|/ \left|\begin{array}{c} D_{43} \\ D_{41} \\ D_{45} \end{array}\right|.$$ 
This can be simplified to $t_1=\Delta_{523}/\Delta_{514}$ by using (\ref{eq:Vandermonde}) and (\ref{eq:quasi_generic}).
Furthermore, since $(A_1,\bp|_{V_0})$ is a 1-dof framework and  $\bq_1(u_3)$ takes fixed non-zero values in $\mathbb{Q}(\bp(V_0))$, $\bq_1$ is  uniquely determined by Cramer's rule and will satisfy $\bq_1(w)\in\mathbb{Q}(\bp(V_0))$ for all $w\in V_0$.

We may apply the same argument, to deduce that  $(A_2,\bp|_{V_0})$  and $(A_3,\bp|_{V_0})$ have the non-trivial motions $\bq_2$ and $\bq_3$ given in the claim.

\medskip
We next prove the second part of the claim.
The existence of a non-trivial motion $\bq$ such that 
$\bq|_{V_0}=\sum_{j=1}^3 \alpha_j\bq_j$ is established in (\ref{eq:uncoupled2}).
Since $u_0u_k$ is an edge of $G$ for all $1\leq k\leq 5$, we have
$D_{0,k}\cdot [\bq(u_0)-\bq(u_k)]=0$ for all $1\leq k\leq 5$. 
Hence we obtain the system of equations
$$ D_{0,k}\cdot [\bq(u_0)-\alpha_1 \bq_1(u_k)-\alpha_2 \bq_2(u_k)-\alpha_3 \bq_3(u_k)]=0 \mbox{ for $1\leq k\leq 5$},$$
or in matrix form,
\begin{equation}
\label{eq:coupled1}
\left(
\begin{array}{cccc}
D_{0,1}\cdot \bq_1(u_1) & D_{0,1}\cdot \bq_2(u_1) &  D_{0,1}\cdot \bq_3(u_1) & D_{0,1} \\
D_{0,2}\cdot \bq_1(u_2) & D_{0,2}\cdot \bq_2(u_2) &  D_{0,2}\cdot \bq_3(u_2) & D_{0,2} \\
D_{0,3}\cdot \bq_1(u_3) & D_{0,3}\cdot \bq_2(u_3) &  D_{0,3}\cdot \bq_3(u_3) & D_{0,3} \\
D_{0,4}\cdot \bq_1(u_4) & D_{0,4}\cdot \bq_2(u_4) &  D_{0,4}\cdot \bq_3(u_4) & D_{0,4} \\
D_{0,5}\cdot \bq_1(u_5) & D_{0,5}\cdot \bq_2(u_5) &  D_{0,5}\cdot \bq_3(u_5) & D_{0,5} \\
\end{array}
\right) \left(
\begin{array}{c}
\alpha_1 \\ \alpha_2 \\ \alpha_3 \\ -a_0\\ -b_0\\ -c_0
\end{array}
\right) = 0,
\end{equation}
where  $(a_0,b_0,c_0)=\bq(u_0)$.
Substituting the zero values for $\bq_j(u_k)$ given in the first  part of the claim, we may rewrite this system  as: 
\begin{equation}
\label{eq:coupled2}
\left(
\begin{array}{cccc}
0 & D_{0,1}\cdot \bq_2(u_1) &  0 & D_{0,1} \\
0 & 0 &  0 & D_{0,2} \\
D_{0,3}\cdot \bq_1(u_3) & 0 &  D_{0,3}\cdot \bq_3(u_3) & D_{0,3} \\
D_{0,4}\cdot \bq_1(u_4) & D_{0,4}\cdot \bq_2(u_4) &  D_{0,4}\cdot \bq_3(u_4) & D_{0,4} \\
0 & 0 &  0 & D_{0,5} \\
\end{array}
\right) \left(
\begin{array}{c}
\alpha_1 \\ \alpha_2 \\ \alpha_3\\ -a_0\\ -b_0\\ -c_0 
\end{array}
\right) = 0.
\end{equation}
By scaling $\bq$, we may suppose that $\alpha_1+\alpha_2+\alpha_3=\mu$, 
where $\mu$ is a non-zero number which will be chosen later. 
Let $M'$ be the matrix of coefficients in (\ref{eq:coupled2}) and  $M$ denote the square matrix $\left(\begin{array}{c}b\\M'\end{array}\right)$ where $b=(1,1,1,0,0,0)$. Then we may rewrite (\ref{eq:coupled2}) as $M(\alpha_1,\alpha_2,\alpha_3,-a_0,-b_0,-c_0)^\top=(\mu,0,0,0,0,0)^\top$ and Cramer's rule gives
\begin{align*}
\alpha_1&=\frac{\mu}{{\rm det} M}\left|
\begin{array}{cccc}
 D_{0,1}\cdot \bq_2(u_1) & 0  & D_{0,1} \\
 0 &  0 & D_{0,2} \\
 0 &  D_{0,3}\cdot \bq_3(u_3) & D_{0,3} \\
 D_{0,4}\cdot \bq_2(u_4) &  D_{0,4}\cdot \bq_3(u_4) & D_{0,4} \\
 0 &  0 & D_{0,5} \\
\end{array}
\right| \\
\alpha_2&=-\frac{\mu}{{\rm det} M}\left|
\begin{array}{cccc}
0  &  0 & D_{0,1} \\
0  &  0 & D_{0,2} \\
D_{0,3}\cdot \bq_1(u_3)  &  D_{0,3}\cdot \bq_3(u_3) & D_{0,3} \\
D_{0,4}\cdot \bq_1(u_4)  &  D_{0,4}\cdot \bq_3(u_4) & D_{0,4} \\
0 &  0 & D_{0,5} \\
\end{array}
\right| \\
\alpha_3&=\frac{\mu}{{\rm det} M}\left|
\begin{array}{cccc}
0 & D_{0,1}\cdot \bq_2(u_1) &   D_{0,1} \\
0 & 0 &   D_{0,2} \\
D_{0,3}\cdot \bq_1(u_3) & 0  & D_{0,3} \\
D_{0,4}\cdot \bq_1(u_4) & D_{0,4}\cdot \bq_2(u_4)  & D_{0,4} \\
0 & 0  & D_{0,5} \\
\end{array}
\right|
\end{align*}
Note  that each entry of the form $D_{0,k}\cdot \bq_j(u_k)$ in the above determinants can be written as 
a product of areas of three triangles.
For example, we have
$D_{01}\cdot \bq_2(u_1)=\left|\begin{array}{c} D_{10}\\ D_{13} \\ D_{15}  \end{array}\right|
=-\Delta_{103}\Delta_{135}\Delta_{150}$, 
where the first equation follows from the first part of the claim and the second equation follows from the Vadermonde identity (\ref{eq:Vandermonde}).

We may  expand the determinant in the formula for $\alpha_1$ to obtain
\begin{align*}
\frac{{\rm det} M}{\mu}\alpha_1 
&=  D_{0,1}\cdot \bq_2(u_1)\left(
-D_{0,3}\cdot \bq_3(u_3) 
\left|
\begin{array}{c}
D_{0,2} \\ D_{0,4} \\ D_{0,5}
\end{array}\right|
+
D_{0,4}\cdot \bq_3(u_4) 
\left|
\begin{array}{c}
D_{0,2} \\ D_{0,3} \\ D_{0,5}
\end{array}\right|\right) \\
&\quad -D_{0,4} \cdot \bq_2(u_4) D_{0,3}\cdot \bq_3(u_3) 
\left|
\begin{array}{c}
D_{0,1} \\ D_{0,2} \\ D_{0,5}
\end{array}\right| \\ 
&= \left|
\begin{array}{c}
D_{1,0} \\ D_{1,3} \\ D_{1,5}
\end{array}\right|
\left(
-
\left|
\begin{array}{c}
D_{3,0} \\ D_{3,1} \\ D_{3,5}
\end{array}\right|
\left|
\begin{array}{c}
D_{0,2} \\ D_{0,4} \\ D_{0,5}
\end{array}\right|
+t_3\left|
\begin{array}{c}
D_{4,0} \\ D_{4,2} \\ D_{4,5}
\end{array}\right|
\left|
\begin{array}{c}
D_{0,2} \\ D_{0,3} \\ D_{0,5}
\end{array}\right|
\right)
-t_2
\left|
\begin{array}{c}
D_{4,0} \\ D_{4,2} \\ D_{4,5}
\end{array}\right|
\left|
\begin{array}{c}
D_{3,0} \\ D_{3,1} \\ D_{3,5}
\end{array}\right|
\left|
\begin{array}{c}
D_{0,1} \\ D_{0,2} \\ D_{0,5}
\end{array}\right|\\
&=
(\Delta_{013}\Delta_{135}\Delta_{015}\Delta_{035}\Delta_{024}\Delta_{045}\Delta_{025})
(-\Delta_{013}\Delta_{135} +t_3\Delta_{245}\Delta_{023} -t_2\Delta_{245}\Delta_{012}).
\end{align*}
The required expression for $\alpha_1$ follows by putting  $\mu={\rm det} M/ \Delta_{015}\Delta_{025}\Delta_{035}\Delta_{045}$. 

Similar calculations give the required formulae for $\alpha_2$ and $\alpha_3$ (using the same constant $\mu$).
\end{proof}

We next obtain an analogous claim for the case when $(G+f_i)$ is type B.

\begin{claim}\label{claim:motionB}
Suppose that $G_0+f_i$ is type B.
Then
for all $1\leq j\leq 3$, $(B_j^i,\bp|_{V_0})$ 
has a non-trivial motion $\bq_j^i$ satisfying $\bq_j^i(w)\in \rat(\bp(V_0))$ for all $w\in V_0$, 
 and 
\begin{itemize}
\item $\bq_1^i(u_5)=\bq_1^i(u_1)=\bq_1^i(u_2)=0$, \\
$\bq_1^i(u_3)=D_{3,1} \times D_{3,5}+s_{12} D_{3,2}\times D_{3,5}$,  \\
$\bq_1^i(u_4)=t_{11}D_{4,1} \times D_{4,5}+ t_{12} D_{4,2}\times D_{4,5}$;
\item $\bq_2^i(u_5)=\bq_2^i(u_2)=\bq_2^i(u_3)=0, \bq_2^i(u_1)=D_{1,3} \times D_{1,5}$,  $\bq_2^i(u_4)=t_2 D_{4,2} \times D_{4,5}$;
\item $\bq_3^i(u_5)=\bq_3^i(u_1)=\bq_3^i(u_2)=0, \bq_3^i(u_3)=D_{1,3} \times D_{3,5}$, $\bq_3^i(u_4)=t_3^i D_{4,2} \times D_{4,5}$;
\end{itemize}
for $t_2=-\Delta_{513}/\Delta_{524}$ and some
non-zero constants $s_{12},t_{11},t_{12},t_3^i \in\mathbb{Q}(\bp(V_0))$, and we may take  $s_{12},t_{11},t_{12}$ to be the same constants for both $\bq_1^1$ and $\bq_1^2$ when $(G_0,F)$ is type BB.

In addition, $(G+f_i,\bp)$ has  a non-trivial motion $\bq^i$ such that 
$\bq^i|_{V_0}=\sum_{j=1}^3\beta_j^i \bq_j^i$ where: 
\begin{align*}
\beta_1^i&=-\Delta_{012}\Delta_{013}\Delta_{024}\Delta_{135}\Delta_{245}t_2
- \Delta_{013}^2 \Delta_{024} \Delta_{135}^2
+ \Delta_{013}\Delta_{023} \Delta_{024} \Delta_{135} \Delta_{245} t_3^i; \\
\beta_2^i&=\Delta_{012}\Delta_{013}\Delta_{014}\Delta_{135}\Delta_{145}t_{11}
-\Delta_{012}\Delta_{013}\Delta_{024}\Delta_{135}\Delta_{245}t_3^i \\
&\hspace{1em}+ \Delta_{012}\Delta_{013}\Delta_{024}\Delta_{135}\Delta_{245}t_{12} 
- \Delta_{012}\Delta_{023}\Delta_{024}\Delta_{235}\Delta_{245}s_{12}t_3^i; \\
\beta_3^i&=\Delta_{012}\Delta_{013}\Delta_{024}\Delta_{135}\Delta_{245}t_2  
+ \Delta_{012}\Delta_{023}\Delta_{024}\Delta_{235}\Delta_{245}s_{12}t_2 \\
&\hspace{1em}+ \Delta_{013}^2\Delta_{024}\Delta_{135}^2 
- \Delta_{013}\Delta_{014}\Delta_{023}\Delta_{135}\Delta_{145}t_{11}\\
&\hspace{1em}+ \Delta_{013}\Delta_{023}\Delta_{024}\Delta_{135}\Delta_{235}s_{12} 
- \Delta_{013}\Delta_{023}\Delta_{024}\Delta_{135}\Delta_{245}t_{12}.
\end{align*}
\end{claim}
\begin{proof}
We will suppress the superscript $i$ in $B^i_j,\bq^i_j,\beta^i_j, t_3^i$ as it remains constant throughout the proof.

We first consider $(B_2,\bp|_{V_0})$. Since $K(u_2,u_3,u_5)\subset \cl(B_2)$ and $u_1u_2\not\in \cl(B_2)$,  
$(B_2,\bp|_{V_0})$ has a non-trivial motion $\bq_2$ such that $\bq_2(u_2)=\bq_2(u_3)=\bq_2(u_5)=0$ and $\bq_2(u_1)\neq 0$. Since ${\rm cl}(B_2)$ contains $u_1u_3$ and $u_1u_5$, $\bq_2(u_1)=s_2 D_{1,3}\times D_{1,5}$ for some scalar $s_2$. Since $\bq_2(u_1)\neq 0$ we can scale $\bq_2$, so that $s_2=1$.
Since ${\rm cl}(B_2)$ contains $u_4u_2$ and $u_4u_5$, $\bq_2(v_4)=t_2 D_{4,2}\times D_{4,5}$ for some scalar $t_2$.
Since $u_1u_4\in \cl(B_2)$, we have 
$0=D_{14}\cdot (\bq_2(u_1)-\bq_2(u_4))$ and we can now use 
(\ref{eq:Vandermonde}) and (\ref{eq:quasi_generic}) to deduce that 
$t_2= -\Delta_{513}/\Delta_{524}$.
Furthermore, since $(B_2,\bp|_{V_0})$ is a 1-dof framework and $\bq_2(u_1)$ takes fixed non-zero values in $\mathbb{Q}(\bp(V_0))$, $\bq_2$ is  uniquely determined by Cramer's rule and will satisfy $\bq_2(w)\in\mathbb{Q}(\bp(V_0))$ for all $w\in V_0$.

%

We can apply the same argument to $(B_3, \bp|_{V_0})$ using the fact that $K(u_1,u_2,u_5)\subset B_3$ to obtain the formula for $\bq_3$. Note that $\bq_3(u_3)\neq 0\neq \bq_3(u_4)$ since $u_1u_4,u_2u_3\not\in \cl(B_3)$ and that we cannot obtain an explicit formula for $t_3$ because $u_3u_4\not\in \cl(B_3)$.


We next consider $(B_1, \bp|_{V_0})$. Since $B_1-u_3u_4+u_3u_1=B_3$,
$\bq_3$ is a motion of $(B_1-u_3u_4,\bp|_{V_0})$ satisfying $\bq_3(u_5)=\bq_3(u_1)=\bq_3(u_2)=0, \bq_3(u_3)=D_{1,3} \times D_{3,5}$, $\bq_3(u_4)=t_3 D_{4,2} \times D_{4,5}$ and  $D_{3,4}\times (\bq_3(u_3)-\bq_3(u_4))\neq 0$. We may use the symmetry between $u_3$ and $u_4$ in type $B$, see Figure \ref{fig:B}, to deduce that $(B_1-u_3u_4,\bp|_{V_0})$ also has a nontrivial motion $\bar \bq_3$ satisfying $\bar\bq_3(u_5)=\bar\bq_3(u_1)=\bar\bq_3(u_2)=0, \bar\bq_3(u_3)=D_{2,3} \times D_{3,5}$, $\bar\bq_3(u_4)=\bar t_3 D_{4,1} \times D_{4,5}$ and
$D_{3,4}\times (\bar\bq_3(u_3)-\bar\bq_3(u_4))\neq 0$. (The motion $\bar \bq_3$ is a non-trivial motion of $\bar B_3:=B_1-u_3u_4+u_4u_1$ and we have ${\rm cl}(\bar B_3)\cap K_u={(\rm cl}(B_3)\cap K_u)-u_1u_3-u_2u_4+u_1u_4+u_2u_3$.)

We can now obtain the  non-trivial motion $\bq_1$ of $(B_1,\bp|_{V_0})$ satisfying $\bq_1(u_5)=\bq_1(u_1)=\bq_1(u_2)=0$, 
$\bq_1(u_3)=s_{11}D_{3,1} \times D_{3,5}+s_{12} D_{3,2}\times D_{3,5}$ and 
$\bq_1(u_4)=t_{11}D_{4,1} \times D_{4,5}+ t_{12} D_{4,2}\times D_{4,5}$
by choosing a non-trivial linear combination of $\bq_3$ and $\bar \bq_3$ which will satisfy the constraint given by the edge $u_3u_4$. The constants $s_{11},s_{12},t_{11},t_{12}$ will all be non-zero since $u_3u_2,u_3u_1,u_4u_2,u_4u_1\not\in {\rm cl}(B_1)$, and so we can scale $\bq_1$ to ensure that $s_{11}=1$.

It remains to show that we can choose the same values for $s_{12}^i,t_{11}^i,t_{12}^i$ when $i=1,2$ and $G_0+f_1$ and $G_0+f_2$ are both type B.
Property (4) in the definition of type B gives $\{u_1u_2,u_3u_4\}\subset [F_v]$.
Claim~\ref{claim:Fv} and the independence of $B_1$ now imply that ${\rm cl}(B_1^1)={\rm cl}(E_0+K_v)={\rm cl}(B_1^2)$.
This in turn implies that  $Z(B_1^1, \bp|_{V_0})=Z(B_1^2, \bp|_{V_0})$, so we can take
a common non-trivial motion that has the form stated in the claim.
This completes the proof of the first part of the claim.


\medskip

We next prove the second part of the claim. We can construct a non-trivial motion $\bq^i$ of $(G+f_i,\bp)$ with $\bq^i|_{V_0}=\sum_{j=1}^3 \beta^i_j \bq^i_j$ for some $\beta^i_j\in \R$ by adding a suitable trivial motion to an arbitrary non-trivial motion of $(G+f_i,\bp)$, as in the derivation of (\ref{eq:uncoupled2}). We need to show we can choose the $\beta^i_j$ to take the values stated in the claim.  
The proof is identical to that of the second part of Claim~\ref{claim:motionA} so we only give a sketch. 
As $i$ is fixed, we once more suppress the superscripts on $\beta^i, \bq^i$ and $\bq^i_j$.

Since $u_0u_k$ is an edge of $G$ for all $1\leq k\leq 5$, we have
$D_{0,k}\cdot [\bq(u_0)-\bq(u_k)]=0$ for all $1\leq k\leq 5$. 
Hence we obtain the system of equations
$$ D_{0,k}\cdot [\bq(u_0)-\beta_1 \bq_1(u_k)-\beta_2 \bq_2(u_k)-\beta_3 \bq_3(u_k)]=0 \mbox{ for $1\leq k\leq 5$}.$$
Putting $(a_0,b_0,c_0)=\bq(u_0)$ and
substituting the zero values for $\bq_j(u_k)$ given in the first  part of the claim, we may rewrite this system  as: 
\begin{equation}
\label{eq:beta1}
\left(
\begin{array}{cccc}
0 & D_{0,1}\cdot \bq_2(u_1) &  0 & D_{0,1} \\
0 & 0 &  0 & D_{0,2} \\
D_{0,3}\cdot \bq_1(u_3) & 0 &  D_{0,3}\cdot \bq_3(u_3) & D_{0,3} \\
D_{0,4}\cdot \bq_1(u_4) & D_{0,4}\cdot \bq_2(u_4) &  D_{0,4}\cdot \bq_3(u_4) & D_{0,4} \\
0 & 0 &  0 & D_{0,5} \\
\end{array}
\right) \left(
\begin{array}{c}
\beta_1 \\ \beta_2 \\ \beta_3\\ -a_0\\ -b_0\\ -c_0 
\end{array}
\right) = 0.
\end{equation}
Each entry of the form $D_{0,k}\cdot \bq_j(u_k)$ in the above matrix can be expressed as 
a product of areas of three triangles by using the non-zero values for $\bq_j(u_k)$ given in the first  part of the claim.
By scaling $\bq$, we may suppose that $\beta_1+\beta_2+\beta_3=\mu$ for some constant $\mu$. 
We can now proceed as in the proof of Claim~\ref{claim:motionA} and obtain the stated formula for $\beta_j$ by setting $\mu={\rm det} M/ \Delta_{015}\Delta_{025}\Delta_{035}\Delta_{045}$, where $M=\left(
\begin{array}{c}
b\\M'
\end{array}
\right)$, $M'$ is the matrix of coefficients in (\ref{eq:beta1}) and $b=(1,1,1,0,0,0)$.
%
%
\end{proof}

We next state the analogous claims for the cases when $(G+f_i)$ is type C or D without proof as their proofs are identical to the proof of Claim~\ref{claim:motionA}.
\begin{claim}\label{claim:motionC}
Suppose that $G_0+f_i$ is type C.
Then
for all $1\leq j\leq 3$, $(C_j^i,\bp|_{V_0})$ 
has a non-trivial motion $\bq_j^i$ satisfying $\bq_j^i(w)\in \rat(\bp(V_0))$ for all $w\in V_0$,  and 
\begin{itemize}
\item $\bq_1^i(u_5)=\bq_1^i(u_1)=\bq_1^i(u_2)=0, \bq_1^i(u_3)= D_{3,2} \times D_{3,5}$, $\bq_1^i(u_4)=t_1 D_{4,1} \times D_{4,5}$;
\item $\bq_2^i(u_5)=\bq_2^i(u_2)=\bq_2^i(u_3)=0, \bq_2^i(u_1)= D_{1,3} \times D_{1,5}$, $\bq_2^i(u_4)=t_2^i D_{4,2} \times D_{4,5}$;
\item $\bq_3^i(u_5)=\bq_3^i(u_1)=\bq_3^i(u_2)=0, \bq_3^i(u_3)= D_{3,1} \times D_{3,5}$, $\bq_3^i(u_4)=t_3^i D_{4,2} \times D_{4,5}$;
\end{itemize}
for $t_1=\Delta_{523}/\Delta_{514}$ and some $t_2^i, t_3^i\in\mathbb{Q}(\bp(V_0))$. 

In addition, $(G+f_i,\bp)$ has  a non-trivial motion $\bq^i$ such that 
$\bq^i|_{V_0}=\sum_{j=1}^3\gamma_j^i \bq_j^i$ where:
\begin{align*}
\gamma_2^i&=\Delta_{012}\Delta_{013}\Delta_{014}\Delta_{135}\Delta_{145}t_1 - \Delta_{012}\Delta_{023}\Delta_{024}\Delta_{235}\Delta_{245}t_3^i;\\
\gamma_3^i&=\Delta_{012}\Delta_{023}\Delta_{024}\Delta_{235}\Delta_{245}t_2^i - \Delta_{013}\Delta_{014}\Delta_{023}\Delta_{135}\Delta_{145}t_1 + \Delta_{013}\Delta_{023}\Delta_{024}\Delta_{135}\Delta_{235}.
\end{align*}
\end{claim}

\begin{claim}\label{claim:motionD}
Suppose that $G_0+f_i$ is type D.
Then
for all $1\leq j\leq 3$, $(D_j^i,\bp|_{V_0})$ 
has
a non-trivial motion $\bq_j^i$ satisfying $\bq_j^i(w)\in \rat(\bp(V_0))$ for all $w\in V_0$, 
 and 
\begin{itemize}
\item $\bq_1^i(u_5)=\bq_1^i(u_1)=\bq_1^i(u_2)=0, \bq_1^i(u_3)=D_{3,2} \times D_{3,5}$,  $\bq_1^i(u_4)=t_1 D_{4,1} \times D_{4,5}$;
\item $\bq_2^i(u_5)=\bq_2^i(u_3)=\bq_2^i(u_4)=0, \bq_2^i(u_1)=s_2^iD_{1,3} \times D_{1,5}$,  $\bq_2^i(u_2)=t_2^i D_{2,4} \times D_{2,5}$;
\item $\bq_3^i(u_5)=\bq_3^i(u_1)=\bq_3^i(u_2)=0, \bq_3^i(u_3)=s_3^iD_{3,1} \times D_{3,5}$, $\bq_3^i(u_4)=t_3^i D_{4,2} \times D_{4,5}$;
\end{itemize}
for $t_1=\Delta_{523}/\Delta_{514}$ and some $s_2^i, t_2^i, s_3^i, t_3^i\in\mathbb{Q}(\bp(V_0))$ 
with $\Delta_{513} s_j^i -\Delta_{524} t_j^i\neq 0$ for $j=2,3$.\footnote{The expression for $q_2^i$ contains two parameters $s^i_2,t^i_2$ because one of the edges $u_1u_4, u_2u_3$ may exist in $\cl(D^i_2)$, and hence we could have $q^i_2(u_1)=0$ or $q^i_2(u_2)=0$. We can deduce that $\Delta_{513}s^i_2-\Delta_{524}t^i_2\neq 0$ from the fact that the edge $u_1u_2$ is not in $\clo(D^i_2)$. A similar remark holds for $q_3^i$.}

In addition, $(G+f_i,\bp)$ has  a non-trivial motion $\bq^i$ such that 
$\bq^i|_{V_0}=\sum_{j=1}^3\delta_j^i \bq_j^i$ where:
\begin{align*}
\delta_2^i=&\Delta_{012} \Delta_{013} \Delta_{014} \Delta_{135} \Delta_{145} s_3^it_1 - \Delta_{012} \Delta_{023} \Delta_{024} \Delta_{235} \Delta_{245}  t_3^i; \\
\delta_3^i=&-\Delta_{013} \Delta_{014} \Delta_{023} \Delta_{135} \Delta_{145} s_2^it_1 + \Delta_{013} \Delta_{014} \Delta_{024} \Delta_{145} \Delta_{245} t_1 t_2^i \\ 
&+\Delta_{013} \Delta_{023} \Delta_{024} \Delta_{135} \Delta_{235} s_2^i - \Delta_{014} \Delta_{023} \Delta_{024} \Delta_{235} \Delta_{245}  t_2^i.
\end{align*}
\end{claim}

We will show that each of the above expressions for a motion of $(G+f_i,p)$  
lead to a contradiction. We first show that the motion given by Claim \ref{claim:motionA} cannot occur.

\begin{claim}\label{claim:typeA}
$G_0+f_i$ cannot be of type A.
\end{claim}
\begin{proof}
Suppose, for a contradiction, that $G_0+f_i$ is type $A$. Let $\bq_j^i$ be the nontrivial motion of $(A^i_j,\bp|_{V_0})$ given by Claim \ref{claim:motionA}. 
We will omit the superscript $i$ from $\bq_j^i$ throughout the proof as it remains fixed. Let $\bp(u_0)=(x_0, y_0)$ and $\bp(u_5)=(x_5, y_5)$, and recall that the values of $\bp(u_k)$, $1\leq k\leq 4$, are given by (\ref{eq:quasi_generic}).
We will contradict the choice of $G$ by showing that $(G+f_i,\bp)$ is a 1-dof framework with a bad motion at $u_0$.

By Claim~\ref{claim:motionA}, $(G+f_i, \bp)$ has  a non-trivial motion $\bq^i$ such that
$\bq^i|_{V_0}=\sum_{j=1}^3 \alpha_j \bq_j$.
The formulae for $\alpha_j$, $j=1,2,3$,  and $\bq_j(u_k)$, $1\leq k\leq 5$, given in Claim~\ref{claim:motionB} imply that 
each coordinate of $\bq^i(u_k)$, $1\leq k\leq 5$, can be expressed as a  polynomial in $x_0, y_0, x_5, y_5$ over $\mathbb{Q}$.  
We  show that the same is true for $\bq^i(u_0)$.

Let $\bq^i(u_0)=(a_0,b_0,c_0)$. 
We recall the the following linear system in the proof of  Claim~\ref{claim:motionA}:

\begin{equation}
\label{eq:Aq^i}
\left(
\begin{array}{cccc}
1 & 1 & 1 & 0  \\
0 & D_{0,1}\cdot \bq_2(u_1) &  0 & D_{0,1} \\
0 & 0 &  0 & D_{0,2} \\
D_{0,3}\cdot \bq_1(u_3) & 0 &  D_{0,3}\cdot \bq_3(u_3) & D_{0,3} \\
D_{0,4}\cdot \bq_1(u_4) & D_{0,4}\cdot \bq_2(u_4) &  D_{0,4}\cdot \bq_3(u_4) & D_{0,4} \\
0 & 0 &  0 & D_{0,5} \\
\end{array}
\right) \left(
\begin{array}{c}
\alpha_1 \\ \alpha_2 \\ \alpha_3\\ -a_0\\ -b_0\\ -c_0 
\end{array}
\right) = 
\left(
\begin{array}{c}
\mu \\ 0 \\ 0 \\ 0\\ 0\\ 0 
\end{array}
\right)
\end{equation}
where $\mu=\alpha_1+\alpha_2+\alpha_3\in \rat(\bp(\hat N_G(u_0)))$ takes the value given at the end of the proof of Claim~\ref{claim:motionA} so can be considered as a rational function of $x_0, y_0, x_5, y_5$ over $\rat$. 
Since each entry of the matrix of coefficients in (\ref{eq:Aq^i}) is a polynomial in  $x_0, y_0, x_5, y_5$, each of 
$a_0,b_0,c_0$ can be expressed as a rational function of $x_0, y_0, x_5, y_5$ over $\rat$.
Hence, by 
a further 
scaling 
$\bq^i$, we can suppose that each component of $\bq^i(u_k)$, $0\leq k\leq 5$, can be expressed as a polynomial in  $x_0, y_0, x_5, y_5$. 
This gives us a map $b:\hat{N}_G(u_0)\rightarrow \mathbb{Q}[X_0,Y_0,X_5, Y_5]^3$ such that 
$\bq^i$ is a $b$-motion on $\hat{N}_G(u_0)$.

We next show that $b$ satisfies (\ref{eq:bad}) i.e.
the polynomial identity given by an edge $u_ku_\ell$, $1\leq k < \ell\leq 5$, 
is satisfied if and only if $\ell=5$. Since $x_0,y_0,x_5,y_5$ are generic, this is equivalent to showing that the motion $\bq^i$ satisfies the constraint  corresponding to the edge $u_ku_\ell$, 
if and only if $\ell=5$. 
Since the current motion $\bq^i$, is a non-zero scalar multiple of our original motion  
it will suffice to show this for the original $\bq^i$. 
(More specifically, it will suffice to show $D_{k,\ell}\cdot (\bq^i(u_k)-\bq^i(u_{\ell}))\neq 0$ for $1\leq k<\ell\leq 4$ as $\bq^i$ is a motion of $(G+f_i,\bp)$ and ${\rm cl}(G+f_i)$ contains the star on $N_G(u_0)$ centered at $u_5$.) 
This follows since $\bq^i|_{V_0}=\sum_{j=1}^3\alpha_j \bq_j$ and each $u_ku_\ell$ is an edge of $\cl(A^i_j)$ for exactly two values $1\leq j\leq 3$.
%
%
Thus $(G+f_i, \bp)$ has a bad motion at $u_0$.

Although $(G+f_i, \bp)$ is not generic as it satisfies (\ref{eq:quasi_generic}), Lemma~\ref{lem:projective_bad} ensures that we may use a projective transformation to construct a generic 1-dof framework $(G+f_i,\bp')$ which has a bad motion at $u_0$. This contradicts our  initial choice of $(G,\bp)$ to minimize $|V|$ and then $k$ since $k=2$ and $G+f_i$ is $1$-dof.
\end{proof}

Our proof that the motions given by Claim \ref{claim:motionB}, \ref{claim:motionC} and \ref{claim:motionD}  cannot occur is more complicated because the expressions for the motions $\bq^i_j$ in these claims contain undetermined parameters. 
%
We first give a preliminary claim which shows that 
specific values of these parameters cannot occur when  
$G_0+f_i$ is type B. 

\begin{claim}\label{claim:bad_case}
Suppose that $G_0+f_i$ is type B.  Let $\bq^i$ and $\bq_j^i$ be the non-trivial motion of  $(G+f_i,\bp)$ and $(B_j^i,\bp|_{V_0})$, respectively,  given in Claim~\ref{claim:motionB}.
Then 
\[
 (s_{12},  t_{11}, t_{12}, t_2, t_3^i)\neq
\left( -\frac{\Delta_{513}}{\Delta_{523}}, \frac{\Delta_{513}}{\Delta_{514}}, 
-\frac{\Delta_{513}}{\Delta_{524}}, -\frac{\Delta_{513}}{\Delta_{524}}, -\frac{\Delta_{513}}{\Delta_{524}}\right).
\]
%
\end{claim}
\begin{proof}
Suppose that $
 (s_{12},  t_{11}, t_{12}, t_2, t_3^i)=
\left( -\frac{\Delta_{513}}{\Delta_{523}}, \frac{\Delta_{513}}{\Delta_{514}}, 
-\frac{\Delta_{513}}{\Delta_{524}}, -\frac{\Delta_{513}}{\Delta_{524}}, -\frac{\Delta_{513}}{\Delta_{524}}\right)
$.
Then we can use the same argument as in the proof of Claim \ref{claim:typeA} to show that some scalar multiple of $\bq^i$ is a
  $b$-motion of $(G+f_i,\bp)$ on $\hat{N}_G(u_0)$ for some $b:\hat{N}_G(u_0)\rightarrow \mathbb{Q}[X_0,Y_0,X_5, Y_5]^3$. In order to
 verify that $b$ satisfies (\ref{eq:bad}) it suffices to show that 
$b$ satisfies
  $D_{k,\ell}\cdot (\bq^i(u_k)-\bq^i(u_{\ell}))\neq 0$ for all
  $u_ku_\ell\in K(u_1,u_2,u_3,u_4)$.  
This inequality holds for the edges $u_1u_2,u_1u_3,u_2u_4$ since they are  edges of $\cl(A^i_j)$ for exactly two values $1\leq j\leq 3$. We can verify the inequality holds for the remaining three edges by a direct computation using the expressions for the motions $\bq^i$ and $\bq^i_j$ given in Claim~\ref{claim:motionB}. For example, the facts that $u_1u_4$ is in $\cl(A^i_2)$, $q^i_1(u_1)=0=q^i_3(u_1)$ and $t_{12}=t_3^i$ give 
\begin{eqnarray*}
D_{1,4}\cdot (\bq^i(u_1)-\bq^i(u_{4}))&=&-D_{1,4}\cdot (\bq^i_1(u_4)+\bq^i_3(u_{4}))=t_{12}(\beta_1+\beta_3)
\left|
\begin{array}{c}
D_{4,1} \\ D_{4,2} \\ D_{4,3}
\end{array}\right|\\
&=&t_{12}\Delta_{513}^2 x_0(2x_0-1)(x_0-1)
\left|
\begin{array}{c}
D_{4,1} \\ D_{4,2} \\ D_{4,3}
\end{array}\right|
\neq 0.
\end{eqnarray*}
%
%
Thus $(G+f_i, \bp)$ has a bad motion at $u_0$.

We can now use Lemma~\ref{lem:projective_bad}  to construct a generic 1-dof framework $(G+f_i,\bp')$ which has a bad motion at $u_0$. This contradicts our  initial choice of $(G,\bp)$ to minimize $|V|$ and then $k$ since $k=2$ and $G+f_i$ is $1$-dof.
\end{proof}

We can use Claim \ref{claim:bad_case} to obtain the following key result for handling the case when $G_0+f_i$ is type B.

\begin{claim}\label{claim:linear3}
Suppose that $G_0+f_i$ is type B, and, for $1\leq j\leq 3$, let $\bq_j^i$ be the non-trivial motion of  $(B_j^i,\bp|_{V_0})$ given in Claim~\ref{claim:motionB}.
If $\lambda_2\beta_2+\lambda_3\beta_3=0$ for some $\lambda_2,\lambda_3\in \mathbb{Q}(\bp(V_0))$, then $\lambda_2=\lambda_3=0$.
\end{claim}
\begin{proof}
 We will again omit the superscript $i$ from the proof since it stays fixed. We assume that  $\lambda_2\beta_2+\lambda_3\beta_3=0$ for some $\lambda_2,\lambda_3\in \mathbb{Q}(\bp(V_0))$ which are not both zero, and show that this contradicts Claim \ref{claim:bad_case}. 
%

Let  $\bp(u_i)=(x_i, y_i)$ for $0\leq i\leq5$. 
By Claim~\ref{claim:motionB} we have an explicit formula for each $\beta_i$ in terms of $x_0,\dots, x_5, y_0, \dots, y_5$.
Since $V_0=V\setminus \{u_0\}$, 
 we may regard each $\beta_i$ as a polynomial in $x_0, y_0$ with coefficients in $\mathbb{Q}(\bp(V_0))$. 
If $\lambda_2\beta_2+\lambda_3\beta_3=0$ with $\lambda_i\in \mathbb{Q}(\bp(V_0))$, then $\lambda_2\beta_2+\lambda_3\beta_3$ is identically zero as a polynomial in $x_0, y_0$. 
In particular, the coefficient of each monomial is zero. 
%

Using Claim~\ref{claim:motionB} and (\ref{eq:quasi_generic}),
a computation of the list of coefficients of the polynomial $\lambda_2\beta_2+\lambda_3\beta_3$ in $x_0, y_0$ over $\mathbb{Q}(\bp(V_0))$ gives:
\begin{itemize}
\item[$x_0^3$:] $-\Delta_{513}^2 \lambda_3 + \Delta_{513} \Delta_{514} \lambda_3 t_{11} - \Delta_{513} \Delta_{523} \lambda_3 s_{12} + \Delta_{513} \Delta_{524} \lambda_3 t_{12}$,
\item[$x_0^2y_0$:] $-\Delta_{513} \Delta_{514} \lambda_2 t_{11} + \Delta_{513} \Delta_{514} \lambda_3 t_{11} + \Delta_{513} \Delta_{523} \lambda_3 s_{12} - \Delta_{513} \Delta_{524} \lambda_2 t_{12}$ 

 $+ \Delta_{513} \Delta_{524} \lambda_2 t_3 +  \Delta_{523} \Delta_{524} \lambda_2 s_{12} t_3$,
\item[$x_0^2$:] $2 \Delta_{513}^2 \lambda_3 - 2 \Delta_{513} \Delta_{514} \lambda_3 t_{11} + \Delta_{513} \Delta_{523} \lambda_3 s_{12} - \Delta_{513} \Delta_{524} \lambda_3 t_{12}$,
\item[$x_0 y_0^2$:] $\Delta_{513}^2 \lambda_3 - \Delta_{513} \Delta_{514} \lambda_2 t_{11} - \Delta_{513} \Delta_{524} \lambda_3 t_{12} - \Delta_{523} \Delta_{524} \lambda_2 s_{12} t_3$,
\item[$x_0y_0$:] $-\Delta_{513}^2 \lambda_3 + 2 \Delta_{513} \Delta_{514} \lambda_2 t_{11} - \Delta_{513} \Delta_{514} \lambda_3 t_{11} - \Delta_{513} \Delta_{523} \lambda_3 s_{12} + \Delta_{513} \Delta_{524} \lambda_2 t_{12} - \Delta_{513} \Delta_{524} \lambda_2 t_3 +  \Delta_{513} \Delta_{524} \lambda_3 t_{12}$,
\item[$x_0$:] $-\Delta_{513}^2 \lambda_3 + \Delta_{513} \Delta_{514} \lambda_3 t_{11}$,
\item[$y_0^3$:] $\Delta_{513} \Delta_{524} \lambda_2 t_{12} - \Delta_{513} \Delta_{524} \lambda_2 t_3$,
\item[$y_0^2$:] $-\Delta_{513}^2 \lambda_3 + \Delta_{513} \Delta_{514} \lambda_2 t_{11} - \Delta_{513} \Delta_{524} \lambda_2 t_{12} + \Delta_{513} \Delta_{524} \lambda_2 t_3$,
\item[$y_0$:] $\Delta_{513}^2 \lambda_3 - \Delta_{513} \Delta_{514} \lambda_2 t_{11}$
\end{itemize}
where  $s_{12}, t_{11}, t_{12}, t_3$ are all  non-zero by Claim~\ref{claim:motionB}.
Since $\lambda_2\beta_2+\lambda_3\beta_3=0$, all coefficients are zero. The hypothesis that $\lambda_2,\lambda_3$ are not both zero and the expression for the coefficient of $y_0$ give $\lambda_2\neq 0$ and $\lambda_3\neq 0$.
The coefficient of $x_0$ now gives $t_{11}=\Delta_{513}/\Delta_{514}$, and
the coefficient of $y_0$ now gives $\lambda_2=\lambda_3$.
Then the coefficients of $x_0^3$ and $y_0^2$, give $s_{12}=\Delta_{524} t_{12}/ \Delta_{523}$ and $t_3=t_{12}$.
Finally, the coefficient of $x_0^2y_0$ gives $t_{12}=-\Delta_{513}/\Delta_{524}$.
Hence 
$$ (s_{12},  t_{11}, t_{12}, t_2, t_3)=
\left( -\frac{\Delta_{513}}{\Delta_{523}}, \frac{\Delta_{513}}{\Delta_{514}}, 
-\frac{\Delta_{513}}{\Delta_{524}}, -\frac{\Delta_{513}}{\Delta_{524}}, -\frac{\Delta_{513}}{\Delta_{524}}\right).$$
This contradicts Claim~\ref{claim:bad_case}.
\end{proof}

Our next result is key to handling the cases when $G_0+f_i$ is type C or D.

\begin{claim}\label{claim:linear1} Suppose  
$\gamma_j^i,\delta_j^i, $ are as in Claims~\ref{claim:motionC} and \ref{claim:motionD}, respectively.\\
(a) If $\lambda_2\gamma_2^i+\lambda_3\gamma_3^i=0$ for some $\lambda_2,\lambda_3\in \mathbb{Q}(\bp(V_0))$, then $\lambda_2=\lambda_3=0$.
\\
(b) If $\lambda_2\delta_2^i+\lambda_3\delta_3^i=0$ for some $\lambda_2,\lambda_3\in \mathbb{Q}(\bp(V_0))$, then $\lambda_2=\lambda_3=0$.
\end{claim}
\begin{proof}
We first prove (a).
Let  $\bp(u_i)=(x_i, y_i)$ for $0\leq i\leq5$. 
As in the proof  of Claim~\ref{claim:linear3}, we regard each $\gamma_j^i$ as a polynomial in $x_0, y_0$ with coefficients in $\mathbb{Q}(\bp(V_0))$. 
Suppose $\sum_{2\leq j\leq 3} \lambda_j \gamma_j^i=0$ with $\lambda_j^i\in \mathbb{Q}(\bp(V_0))$. Then $P:=\sum_{2\leq j\leq 3} \lambda_j \gamma_j^i$ is identically zero as a polynomial in $x_0, y_0$. 
Using Claim~\ref{claim:motionC} and (\ref{eq:quasi_generic}), we find that the coefficient of $y_0^2$ in $P$ is $\Delta_{513}\Delta_{523} \lambda_2$,
and hence $\lambda_2=0$. This implies that the coefficient of $x_0y_0$ in $P$ is $-2\Delta_{513}\Delta_{523} \lambda_3$ and hence $\lambda_3=0$.

We next prove (b). 
If $\sum_{2\leq j\leq 3} \lambda_j \delta_j^i=0$ with $\lambda_j^i\in \mathbb{Q}(\bp(V_0))$, then $Q:=\sum_{2\leq j\leq 3} \lambda_j \delta_j^i$ is identically zero as a polynomial in $x_0, y_0$. 
Using Claim~\ref{claim:motionD} and (\ref{eq:quasi_generic}), 
we find that the coefficient of $x_0^2$ in $Q$ is $(-\Delta_{513} s_2^i + \Delta_{524}t_2^i)\Delta_{523} \lambda_3$.
Since  $\Delta_{513} s_2^i - \Delta_{524}t_2^i\neq 0$ by Claim~\ref{claim:motionD}, 
 $\lambda_3=0$. This implies that the coefficient of $x_0^2y_0$ in $Q$ is $(-\Delta_{513} s_3^i + \Delta_{524}t_3^i)\Delta_{523} \lambda_2$.
Since  $\Delta_{513} s_3^i - \Delta_{524}t_3^i\neq 0$ by Claim~\ref{claim:motionD}, 
 $\lambda_2=0$.
\end{proof}

We next 
obtain quadratic versions of Claims~\ref{claim:linear3} and \ref{claim:linear1}.
\begin{claim}\label{claim:quadratic}
Let 
$\beta_j^i, 
\gamma_j^i, \delta_j^i$ be as defined in Claims~\ref{claim:motionB}, 
\ref{claim:motionC} and \ref{claim:motionD}, and  $\lambda_{h,j}\in \mathbb{Q}(\bp(V_0))$ for all $1\leq h,j\leq 3$. 
\begin{description}
%
%
\item[BB:] If $\sum_{2\leq h\leq 3}\sum_{2\leq j\leq 3} \lambda_{h,j} \beta_h^1\beta_j^2=0$, then $\lambda_{k,k}=0$ for some $k\in \{2,3\}$.

\item[CC:] If $\sum_{2\leq h\leq 3}\sum_{2\leq j\leq 3} \lambda_{h,j} \gamma_h^1\gamma_j^2=0$, then $\lambda_{k,k}=0$ for some $k\in \{2,3\}$.

\item[DD:] If $\sum_{2\leq h\leq 3}\sum_{2\leq j\leq 3} \lambda_{h,j} \delta_h^1\delta_j^2=0$, then $\lambda_{k,k}=0$ for some $k\in \{2,3\}$.
\end{description}
\end{claim}
\begin{proof}
We first check case (CC).
Suppose $\sum_{2\leq i\leq 3}\sum_{2\leq j\leq 3} \lambda_{i,j} \gamma_i^1\gamma_j^2=0$. 
We may regard this as a polynomial identity with indeterminates  $x_0, y_0$ and hence the coefficient of each monomial must be zero.
If we compute the formula for $\sum_{2\leq i\leq 3}\sum_{2\leq j\leq 3} \lambda_{i,j} \gamma_i^1\gamma_j^2$ explicitly using Claim~\ref{claim:motionC} and (\ref{eq:quasi_generic}), we find that
the coefficient of $y_0^4$  is 
$\Delta_{513}^2\Delta_{523}^2 \lambda_{2,2}$.
Since $\Delta_{513}\neq 0$ and $\Delta_{523}\neq 0$, we obtain $\lambda_{2,2}=0$.

We next check case (DD).
Suppose $\sum_{2\leq i\leq 3}\sum_{2\leq j\leq 3} \lambda_{i,j} \delta_i^1\delta_j^2=0$. 
If we compute the formula for $\sum_{2\leq i\leq 3}\sum_{2\leq j\leq 3} \lambda_{i,j} \delta_i^1\delta_j^2$ explicitly using Claim~\ref{claim:motionD} and (\ref{eq:quasi_generic}), we find that
the coefficient of $x_0^4$  is $\Delta_{523}^2 \lambda_{3,3}(\Delta_{524}t_2^1-T_{513}s_2^1)(\Delta_{524}t_2^2-\Delta_{513}s_2^2)$.
Since $\Delta_{524}t_2^i-\Delta_{513}s_2^i\neq 0$ by Claim~\ref{claim:motionD}, we obtain $\lambda_{3,3}=0$.

Case (BB) can be verified in  the same way by using (\ref{eq:quasi_generic}), Claim~\ref{claim:motionB},  and Claim~\ref{claim:bad_case}. The main steps are as follows. We suppose that $P:=\sum_{2\leq h\leq 3}\sum_{2\leq j\leq 3} \lambda_{h,j} \beta_h^1\beta_j^2=0$ and $\lambda_{k,k}\neq 0$ for $k\in\{2,3\}$. 
The coefficient of $x_0^2$ in $P$  implies that $t_{11}=\Delta_{513}/\Delta_{514}$, and the coefficient of $y_0^6$ implies that either $t_{12}=t_{3}^1$ or $t_{12}=t_{3}^2$. By symmetry we may assume that $t_{12}=t_{3}^1$. Then 
the coefficient of $x_0^4$ gives $s_{12}=t_{3}^1\Delta_{524}/\Delta_{523}$, and the coefficient of $y_0^2$ gives $\lambda_{3,2}=\lambda_{2,2}+\lambda_{3,3}-\lambda_{2,3}$. 
The coefficient of $x_0y_0^2$ now implies that either $\lambda_{3,3}=\lambda_{2,3}$ or $t_3^1=t_3^2$. 
If $t_3^1\neq t_3^2$ then we have $\lambda_{3,3}=\lambda_{2,3}$, the coefficient of $x_0y^5$ now implies that $t_{3}^1=-\Delta_{513}/\Delta_{524}$   and the resulting values of $t_{12},s_{12}$ contradict Claim \ref{claim:bad_case}. Hence $t_3^1=t_3^2$. The coefficient of $x_0y_0^3$ and the fact that $t_{3}^1\neq -\Delta_{513}/\Delta_{524}$ by Claim \ref{claim:bad_case} then give $\lambda_{2,2}=\lambda_{3,3}$. The coefficient of $x_0^2y_0^4$ now gives $t_{3}^1=-\Delta_{513}/\Delta_{524}$ and this again contradicts Claim \ref{claim:bad_case}.
 \end{proof}

%

We  close this section with an observation on the type 0 case.
\begin{claim}\label{claim:motion0}
Suppose that $G_0+f_i$ is type 0. Then 
$(G+f_i, \bp)$ has a non-trivial motion $\bq^i$ such that 
$\bq^i|_{V_0}$ is a non-trivial motion of $(G_0+f_i+K_u,\bp|_{V_0})$, $\bq^i(v_0)=\bq^i(v_1)=\bq^i(v_5)=0$, and
$\bq^i(w)\in \mathbb{Q}(\bp(V_0))$ for all $w\in V_0$.
\end{claim}
\begin{proof}
Since $G_0+f_i$ is type 0, there are two edges $e_1, e_2\in K$ such that 
$K_u\subseteq {\rm cl}(E_0+f_i+e_1+e_2)$.
Since $G+f_i$ is $C_2^1$-independent, we have ${\rm cl}(G+f_i+e_1+e_2-u_0u_4-u_0u_5)={\rm cl}(G+f_i)$ and hence $K_u\subseteq {\rm cl}(G+f_i)$.
This implies that  every non-trivial motion of $(G_0+f_i+e_1+e_2, \bp|_{V_0})$ is extendable to a non-trivial motion $\bq^i$ of $(G+f_i,\bp)$. Since $K(v_0,v_1,v_5)\subset G$ by (\ref{eq:Kv}), we may combine $\bq^i$ with a suitable trivial motion to ensure that $\bq^i(v_0)=\bq^i(v_1)=\bq^i(v_5)=0$. By scaling, we may also assume that some component of $\bq_i|_{V_0}$ is equal to one. Since $(G_0+f_i+e_1+e_2, \bp|_{V_0})$ is a 1-dof framework, we may now use Cramer's rule to deduce that $\bq^i(w)\in \mathbb{Q}(\bp(V_0))$ for all $w\in V_0$.
\end{proof}

\subsubsection{Completing the proof of Theorem \ref{thm:well-behaved}}\label{subsubsec:6}

We are now ready to complete the proof of Theorem \ref{thm:well-behaved}. 
Recall that $(G,\bp)$ is a minimal counterexample to Theorem \ref{thm:well-behaved}, and  has a motion $\bq_{\rm bad}$ which is bad at $v_0$, and $F=\{f_1,f_2\}\in {{F_v}\choose{2}}$ is a good pair.
Since  $K(v_0, v_1, v_5)\subset G$ by (\ref{eq:Kv}),
we may combine $\bq_{\rm bad}$ with a suitable trivial motion $\bt$ so that the resulting motion 
\begin{equation}\label{eq:q+t}
\hat{\bq}_{\rm bad}:=\bq_{\rm bad}+\bt
\end{equation}
 satisfies 
$\hat{\bq}_{\rm bad}(v_0)=\hat{\bq}_{\rm bad}(v_1)=\hat{\bq}_{\rm bad}(v_5)=0$.
Similarly, if $G_0+f_i$ is type X for some X in $\{$B,C,D$\}$,
we may combine the motion $\bq_j^i$ of $(X_j^i,\bp|_{V_0})$ given in Claims~\ref{claim:motionA}-\ref{claim:motionD} with a suitable trivial motion $\bt_j^i$ so that the resulting motion 
\begin{equation}\label{eq:q_j^i+t_j^i}
\hat{\bq}_j^i:=\bq_j^i+\bt_j^i
\end{equation}
 satisfies  
$\hat{\bq}_j^i(v_0)=\hat{\bq}_j^i(v_1)=\hat{\bq}_j^i(v_5)=0$.
Since all entries of ${\bq}_{\rm bad}(v_k)$ and ${\bq}_j^i(v_k)$ are contained in $\mathbb{Q}(\bp(V_0))$ for each $v_k\in \hat{N}_G(v_0)$, we may deduce that all entries of 
$\hat{\bq}_{\rm bad}(v_k)$ and $\hat{\bq}_j^i(v_k)$ are contained in $\mathbb{Q}(\bp(V_0))$ for each $v_k\in \hat{N}_G(v_0)$.

%

We will use the following parameterization of the motion space of $(G,\bp)$ over $\hat{N}_G(v_0)$. Let $G_v$ be the subgraph of $G$ induced by $\hat N_G(v_0)$. Then (\ref{eq:Kv}) implies that $G_v$ is isomorphic to the graph obtained from a complete bipartite graph $K_{2,4}$ by adding an edge between the two vertices in the 2-set $\{v_0,v_5\}$ of the bipartition. See Figure~\ref{fig:K24plus}.
\begin{figure}[h]
\centering
\includegraphics[scale=0.6]{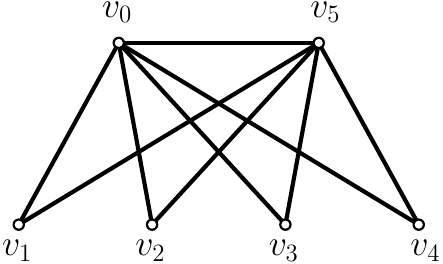}
\caption{$G_v$.}
\label{fig:K24plus}
\end{figure}
Let 
$Z'$ be the space of all motions $\bq'$ of $(G_v,\bp|_{\hat{N}(v_0)})$ with $\bq'(v_0)=\bq'(v_1)=\bq'(v_5)=0$.
As $G_v$ is a 3-dof graph, $Z'$ is a 3-dimensional linear space.
Since $u_0\not\in V(G_v)$ we may choose a base $B$ for $Z'$ in which the coordinates of each vector lie in $\rat(\bp(V_0))$, and let $\psi:Z'\rightarrow \mathbb{R}^3$ be the linear bijection which maps $B$ onto the standard base of $\R^3$. Note that $\psi$ can be represented by a matrix over $\mathbb{Q}(\bp(V_0))$.
%

Using  Claims \ref{claim:types} and \ref{claim:typeA}, and relabeling $u_1,u_2,u_3,u_4$ and $f_1,f_2$ if necessary, we may assume that each $G_0+f_i$ is type 0, B, C or D, and that $(G_0,F)$ is type  BB, CC or DD if neither $G_0+f_1$ nor $G_0+f_2$ is type 0.
The proof is completed by considering the following three cases, depending on the type of each $G+f_i$.  

\medskip
\noindent \emph{Case 1:} Neither $G_0+f_1$ nor $G_0+f_2$ is type 0. \\
%
Let $\bq^i$ be the non-trivial motion of $(G+f_i,\bp)$ given in Claims~\ref{claim:motionB}-\ref{claim:motionD}. Then $\bq^i|_{V_0}=\sum_{j=1}^3\chi_j^i\bq_j^i$ for some appropriate $\chi\in \{\beta,\gamma,\delta\}$  and we put 
\begin{equation}\label{eq:58.1}
\hat \bq^i=\bq^i+\sum_{j=1}^3\chi_j^i \bt_j^i
\end{equation}
where $\bt_j^i$ is the trivial motion given in (\ref{eq:q_j^i+t_j^i}) 
(and we are abusing notation by using the same symbol $\bt_j^i$ for a trivial motion of $(G+f_i,\bp)$ and its restriction to $V_0$). 
Then, by  $\bq^i|_{V_0}=\sum_{j=1}^3\chi_j^i\bq_j^i$  and (\ref{eq:q_j^i+t_j^i}), 
$$\hat \bq^i|_{V_0}=\bq^i|_{V_0}+\sum_{j=1}^3\chi_j^i \bt_j^i=\sum_{j=1}^3\chi_j^i \bq_j^i+\sum_{j=1}^3\chi_j^i \bt_j^i=\sum_{j=1}^3\chi_j^i (\bq_j^i+\bt_j^i)=\sum_{j=1}^3\chi_j^i \hat{\bq}_j^i.$$ 
%
%
As $(f_1, f_2)$ is a good pair, $Z(G,\bp)=Z_0(G,\bp)\oplus \langle \hat{\bq}^1, \hat{\bq}^2\rangle$ holds.
Hence, by Lemma~\ref{lem:canonical}, 
\begin{equation}\label{eq:58.5}
\hat{\bq}_{\rm bad}\in \langle \hat{\bq}^1, \hat{\bq}^2\rangle.
\end{equation}
Suppose $(G_0,F)$  is type BB. Since $\hat{\bq}_{\rm bad}\in \langle \hat{\bq}^1, \hat{\bq}^2\rangle$ by (\ref{eq:58.1}), we have 
\begin{equation}\label{eq:BB1}
\hat{\bq}_{\rm bad}|_{V_0}\in \left\langle \sum_{j=1}^3\beta_j^1 \,\hat{\bq}_j^1|_{V_0}, \sum_{j=1}^3\beta_j^2 \,\hat{\bq}_j^2|_{V_0}\right\rangle.
\end{equation}
We also have $K_v\subseteq {\rm cl}(B_1^i)$ by Property (4) of the definition of type B, so
$\hat{\bq}_1^i|_{\hat{N}(v_0)}=0$ for both $i=1,2$.
Hence the restriction of (\ref{eq:BB1}) to $\hat{N}(v_0)$ gives
\begin{equation}\label{eq:BB2}
\hat{\bq}_{\rm bad}|_{\hat{N}(v_0)}\in \left\langle \sum_{j=2}^3\beta_j^1 \,\hat{\bq}_j^1|_{\hat{N}(v_0)}, \sum_{j=2}^3\beta_j^2 \,\hat{\bq}_j^2|_{\hat{N}(v_0)}\right\rangle.
\end{equation}

Let $k$ and $k_j^i$ denote the images of $\hat{\bq}_{\rm bad}|_{\hat{N}(v_0)}$ and $\hat{\bq}_j^i|_{\hat{N}(v_0)}$,  respectively, under $\psi$. Then (\ref{eq:BB2}) implies that 
\[
{\rm det} 
\left(\begin{array}{c|c|c}
 & & \\
 k& \sum_{j=2}^3\beta_j^1 k_j^1 & \sum_{j=2}^3\beta_j^2 k_j^2 \\
 & & 
 \end{array}\right)=0,
\]
and the bilinearity of the determinant gives
\[
\sum_{i=2}^3\sum_{j=2}^3{\rm det} 
\left(\begin{array}{c|c|c}
 & & \\
 k &  k_i^1 &  k_j^2 \\
 & & 
 \end{array}\right)\beta_i^1\beta_j^2=0.
\]
%
%
%
%
%
Claim~\ref{claim:quadratic} now implies that
\[
{\rm det} 
\left(\begin{array}{c|c|c}
 & & \\
 k &  k_j^1 &  k_j^2 \\
 & & 
 \end{array}\right)=0
\]
for some $j\in \{2,3\}$. We may assume without loss of generality that $j=2$.
Then
$k$ is a linear combination of $k_2^1$ and $k_2^2$ so $\hat{\bq}_{\rm bad}|_{\hat{N}(v_0)}=a\,\hat{\bq}_2^1|_{\hat{N}(v_0)}+b\,\hat{\bq}_2^2|_{\hat{N}(v_0)}$
for some $a,b\in \R$. 
Let 
\[
\bar \bq :=a\,\hat{\bq}_2^1+b\,\hat{\bq}_2^2.
\]
Since $\hat{\bq}_2^i$ is  a non-trivial motion of $(B_2^i, \bp|_{V_0})$ and $B_2^i=G_0+f_i+u_1u_3+u_2u_3$,
$\bar \bq$ is a motion of $(G_0+u_1u_3+u_2u_3,\bp|_{V_0})$. 
Let 
\[
\tilde \bq:=\bar \bq- \bt,
\] 
where $\bt$ is the trivial motion given in (\ref{eq:q+t}).
Then $\tilde\bq$ is a motion of 
the 2-dof framework $(G_0+u_1u_3+u_2u_3,\bp|_{V_0})$ and 
$$\tilde \bq|_{\hat{N}(v_0)}=\bar \bq|_{\hat{N}(v_0)}- \bt|_{\hat{N}(v_0)}=\hat \bq_{\rm bad}|_{\hat{N}(v_0)}- \bt|_{\hat{N}(v_0)}=\bq_{\rm bad}|_{\hat{N}(v_0)}.$$
Hence $\tilde \bq$ is bad a bad motion at $v_0$ for the framework $(G_0+u_1u_3+u_2u_3,\bp|_{V_0})$.
This contradicts the fact that $(G,\bp)$ is a minimal counterexample.

%

The remaining cases when $(G_0,F)$ is type  CC or DD can be solved similarly.

\medskip
\noindent \emph{Case 2:} $G_0+f_1$ is type 0 and $G_0+f_2$ is not. \\
We only give a proof for the case when $G_0+f_2$ is type B since the other cases can be solved in an identical manner.

By Claim~\ref{claim:motion0}, $(G+f_1,\bp)$ has a non-trivial motion $\bq^1$ such that $\bq^1|_{V_0}$ is a motion of $(G_0+f_1+K_u, \bp|_{V_0})$, $\bq^1(v_0)=\bq^1(v_1)=\bq^1(v_5)=0$ and $\bq^1(w)\in \mathbb{Q}(\bp(V_0))$ for all $w\in V_0$. 

Let $\hat{\bq}^2=\bq^2+\sum_{j=1}^3\beta^2_j \bt^2_j$ be as defined in (\ref{eq:58.1}). 
Since  $(f_1, f_2)$ is a good pair, $Z(G,\bp)=Z_0(G,\bp)\oplus \langle {\bq}^1, \hat{\bq}^2\rangle$ holds.
Hence, by Lemma~\ref{lem:canonical},   $\hat{\bq}_{\rm bad}\in \langle {\bq}^1, \hat{\bq}^2\rangle$.
Since $K_v\subseteq {\rm cl}(B_1^2)$ by Property (4) of the definition of type B, this gives
\begin{equation}\label{eq:0B1}
\hat{\bq}_{\rm bad}|_{\hat{N}(v_0)}\in \left\langle \bq^1|_{\hat{N}(v_0)}, \sum_{j=2}^3\beta_j^2 \,\hat{\bq}_j^2|_{\hat{N}(v_0)}\right\rangle.
\end{equation}
As in the previous case, this implies that 
\[
{\rm det} 
\left(\begin{array}{c|c|c}
 & & \\
 k & k^1 & \sum_{j=2}^3\beta_j^2 k_j^2 \\
 & & 
 \end{array}\right)=0,
\]
where $k$, $ k^1$, $k_j^2$ are the image of $\hat \bq_{\rm bad}|_{\hat{N}(v_0)}$, $\bq^1|_{\hat{N}(v_0)}$, $\hat{\bq}_j^2|_{\hat{N}(v_0)}$, respectively, under $\psi$.
The bilinearity of the determinant now gives 
\[
\sum_{2\leq j\leq 3}{\rm det} 
\left(\begin{array}{c|c|c}
 & & \\
 k &  k^1 &  k_j^2 \\
 & & 
 \end{array}\right)\beta_j^2=0,
\]
and Claim~\ref{claim:linear3} implies that 
\[
{\rm det} 
\left(\begin{array}{c|c|c}
 & & \\
 k &  k^1 &  k_j^2 \\
 & & 
 \end{array}\right)=0
\]
for each $j\in \{2,3\}$. In particular, we have that
$k$ is a linear combination of $k^1$ and $k_2^2$ so $\hat{\bq}_{\rm bad}|_{\hat{N}(v_0)}=a{\bq}^1|_{\hat{N}(v_0)}+b\hat{\bq}_2^2|_{\hat{N}(v_0)}$
for some $a,b\in \R$. Let $\bar \bq :=a{\bq}^1|_{V_0}+b\hat{\bq}_2^2$.
Since
$\bq^1|_{V_0}$ is a motion of $(G_0+f_1+K_u, \bp|_{V_0})$ and
$\hat{\bq}_2^2$ is  a  motion of $(B_2^2, \bp|_{V_0})$ where $B_2^2=G_0+f_2+u_2u_3+u_3u_1$,
$\bar \bq$ is a motion of $(G_0+u_2u_3+u_3u_1,\bp|_{V_0})$. 
Let $\tilde \bq:=\bar \bq- \bt$, where $\bt$ is the trivial motion given in (\ref{eq:q+t}).
Then $\tilde \bq$ is a motion of the 2-dof framework $(G_0+u_2u_3+u_3u_1,\bp|_{V_0})$ and 
$$\tilde \bq|_{\hat{N}(v_0)}=\bar \bq|_{\hat{N}(v_0)}- \bt|_{\hat{N}(v_0)}=\hat \bq_{\rm bad}|_{\hat{N}(v_0)}- \bt|_{\hat{N}(v_0)}=\bq_{\rm bad}|_{\hat{N}(v_0)}.$$
Hence $\tilde \bq$ is bad at $v_0$.
This contradicts the fact that $(G,\bp)$ is a minimal counterexample.

%
%

\medskip
\noindent \emph{Case 3:} Both $G_0+f_1$ and $G_0+f_2$ are type 0. \\ 
By Claim~\ref{claim:motion0}, $(G+f_i, \bp)$ has a motion $\bq^i$ such that $\bq^i|_{V_0}$ is a motion of $(G_0+f_i+K_u, \bp|_{V_0})$, $\bq^i(v_0)=\bq^i(v_1)=\bq^i(v_5)=0$ and $\bq^i(w)\in \mathbb{Q}(\bp(V_0))$ for all $w\in V_0$. 

As $(f_1, f_2)$ is a good pair, $Z(G,\bp)=Z_0(G,\bp)\oplus \langle \bq^1, \bq^2\rangle$ holds.
Hence, by Lemma~\ref{lem:canonical},  $\hat{\bq}_{\rm bad}\in \langle \bq^1, \bq^2\rangle$.
This gives
\begin{equation}\label{eq:001}
\hat{\bq}_{\rm bad}|_{\hat{N}(v_0)}\in \left\langle \bq^1|_{\hat{N}(v_0)}, \bq^2|_{\hat{N}(v_0)}\right\rangle.
\end{equation}
and hence
\[
{\rm det} 
\left(\begin{array}{c|c|c}
 & & \\
 k & k^1 &  k^2 \\
 & & 
 \end{array}\right)=0,
\]
where $k$, $k^1$, $k^2$ are the image of $\hat \bq_{\rm bad}|_{\hat{N}(v_0)}$, $\bq^1|_{\hat{N}(v_0)}$, $\bq^2|_{\hat{N}(v_0)}$respectively, under $\psi$.
This in turn implies that  $\hat \bq_{\rm bad}|_{\hat{N}(v_0)}$ is a linear combination of  $\bq^1|_{\hat{N}(v_0)}$ and $\bq^2|_{\hat{N}(v_0)}$.
We can now use a similar argument to the previous case to deduce that 
$(G_0+K_u,\bp|_{V_0})$ has a bad motion at $v_0$. Since $(G_0+K_u,\bp|_{V_0})$ is a $k$-dof framework for some $k=1,2$, this  contradicts the minimality of $(G,\bp)$.

This completes the proof of Theorem~\ref{thm:well-behaved}.
%
%
%

%
%
%

\appendix 
\section*{Appendix}
\section{Proof of  Lemma~\ref{lem:bad_pinning} and statement (\ref{eq:special})}\label{sec:bad_pinning}
We deduce both Lemma~\ref{lem:bad_pinning}  and statement (\ref{eq:special}) from Lemma \ref{lem:abstract_bad_pinning} below, which is a general result on matrices with polynomial entries. 

Given a matrix $M$ with entries in $\mathbb{R}[X_1,X_2,\ldots,X_m]$ and $x=(x_1,x_2,\ldots,x_m)\in \R^m$ we use $M(x)$ to denote the real matrix obtained by substituting the values $X_i=x_i$ into $M$. 
We assume that the rows and columns of $M$ are indexed by two disjoint sets $R$ and $C$ respectively. For $D\subseteq R$ and $U\subseteq C$, we use $M_D$ to denote the submatrix of $M$ with rows indexed by $D$, and $M_{D,U}$ to denote the submatrix  of $M_D$ with columns indexed by $U$. We also set $\overline{D}=R\setminus D$ and $\overline{U}=C\setminus U$.  For $x\in \R^m$, we consider the components of each vector  $z(x)\in \ker M(x)$ to be indexed by $C$ and use $z_i(x)$ to denote the $i$'th component of $z(x)$ for $i\in C$.

%
%

\begin{lemma}\label{lem:abstract_bad_pinning}
Let $\F$ be a subfield of $\R$, 
$M$ be a square matrix with entries in $\F[X_1,\ldots,X_m]$, and 
$
T=\left\{x\in \R^m: 
 \text{$M(x)$ is non-singular}
 \right\}.
$
Let
$D\subseteq R$ and $U\subseteq C$ such that every entry of $M_{D,\overline{U}}$ is zero, and let
$b:U\to \F[X_1,X_2,\ldots,X_m]$.
Suppose that,  for some $\tilde x\in T$ which is generic over $\F$, 
there is a vector $z(\tilde x)\in \ker M_{\overline{D}}(\tilde x)$ such that $z_i(\tilde x)=b_i(\tilde x)$ for all $i\in U$.
Then, for  all $x\in T$, there is a vector $z(x)\in \ker M_{\overline{D}}(x)$ such that $z_i(x)=b_i(x)$
for all $i\in U$.
\end{lemma}
\begin{proof}
 Choose any $x\in T$.
As $M(x)$ is non-singular, it is row independent and hence $M_{R-d}(x)$ is  row independent for all $d\in D$.
 Since each entry of $M_{R-d}(x)$ is a polynomial function of the coordinates of $x$, 
Cramer's rule implies that we can choose a non-zero $z_d\in \ker M_{R-d}(x)$ in such a way that 
each of its coordinates is a polynomial function of the coordinates of $x$.
If  for some $d_1\in D$, the vector $z_{d_1}$ is spanned by $\{z_d\,:\,d\in D-d_1\}$ then we would have $\ker M_{\overline{D}}(x)= \ker M_{\overline{D-d_1}}(x)$ and this would contradict the row independence of $M(x)$. Thus  
 $\{z_d:d\in D\}$ is linearly independent, and so $\dim \ker M_{\overline{D}}(x)\geq|D|$. The independence of $M(x)$ implies this holds with equality. Hence $\{z_d:d\in D\}$ is  a base for $\ker M_{\overline{D}}(x)$ and each $z\in \ker M_{\overline{D}}(x)$ can be expressed as $z=\sum_{d\in D} \lambda_d z_d$ for some $\lambda_d\in \R$.

Let  $z_d|_U$ denote the restriction of $z_d$ to $U$ for each $d\in D$. We will show that 
\begin{equation}\label{eq:abstract_bad_pinning2}
\text{
$\{ z_d|_U\,:\,d\in D\}$ is linearly independent.}
\end{equation}
Suppose to the contrary that, for some  
 some $d_1\in D$, $z_{d_1}|_U$ is spanned by $\{z_d|_U:d\in D-d_1\}$. Then  the hypothesis that every entry of $M_{D,\overline{U}}$ is zero implies that we again have $\ker M_{\overline{D}}(x)= \ker M_{\overline{D-d_1}}(x)$ and contradict the row independence of $M(x)$.
Hence (\ref{eq:abstract_bad_pinning2}) holds.

We next consider the matrix equation $A_x\,\lambda=b_x$ with variable $\lambda$, where $A_x$ is the $|U|\times |D|$ matrix with columns $z_d|_U$ for $d\in D$, 
$\lambda=(\lambda_1,\ldots,\lambda_{|D|})^\top$, and   $b_x\in \R^{|U|}$ is the column vector with entries given by the coordinates of
$b_{i}(x)$ for $i\in U$. 
By (\ref{eq:abstract_bad_pinning2}), $\rank A_x=|D|$ for all $x\in T$.

Since, for some generic $\tilde x\in T$,   
there is a vector $z(\tilde x)\in \ker M_{\overline{D}}(\tilde x)$ such that $z_i(\tilde x)=b_i(\tilde x)$ for all $i\in U$, the equation $A_{\tilde x}\,\lambda=b_{\tilde x}$ has a solution, and hence $\rank (A_{\tilde x},b_{\tilde x})=\rank A_{\tilde x}=|D|$ for this generic $\tilde x$.
Since each entry in $(A_x,b_x)$ is a polynomial function of the coordinates of $x$, this implies that $\rank (A_x,b_x)\leq |D|$ for all $x\in T$.
On the other hand we have seen that $\rank A_x=|D|$ for all $x\in T$.
Hence $\rank (A_x,b_x)=\rank A_x$ holds for all $x\in T$, and as a result, the equation $A_x\,\lambda=b_x$ has a solution for all $x\in T$.
This solution will give us a vector $z(x)\in \ker M_{\overline{D}}(x)$ such that $z_i(x)=b_i(x)$ for all $i\in U$ and all $x\in T$.
\end{proof}

We are now ready to prove Lemma~\ref{lem:bad_pinning}. For convenience we repeat the statement.

\newtheorem*{slemma}{Lemma~\ref{lem:bad_pinning}}
\begin{slemma}
Let $G=(V,E)$ be a $C_2^1$-independent graph with $V=\{v_1,\dots,v_n\}$, 
$U
\subseteq V$, $F
\subseteq K(U)$, 
and
$
S=\{\bp:\mbox{$(G+F,\bp)$ is minimally $C_2^1$-rigid and $\bp$ is non-degenerate on $U$}\}.
$
Suppose that $b:U\to \rat[X_1,Y_1,\ldots,X_n,Y_n]^3$ and $(G,\bp)$ has a $b$-motion for some generic $\bp$. Then $(G, \bp)$ has a $b$-motion for all $ \bp\in S$.
\end{slemma}
\begin{proof}
Choose $ \bp_1\in S$ and let $\bp_1(v_i)=(x_i,y_i)$ for all $v_i\in V$. Since $\bp_1$ is non-degenerate on $U$, we can choose three distinct vertices $v_a,v_b,v_c\in U$ which have distinct `$y$-coordinates' in $(G, \bp_1)$. Let $\tilde{C}(G+F, \bp_1)$ be the extended $C_2^1$-cofactor matrix with respect to $v_a,v_b,v_c$ and let $M$ be the matrix obtained from $\tilde{C}(G+F, \bp_1)$ by replacing each of the coordinates $x_i,y_i$ in the definition of $\tilde{C}(G+F, \bp_1)$ by indeterminates $X_i,Y_i$ for all $v_i\in V$. Then each entry in $M$ belongs to $\rat[X_1,Y_1,\ldots,X_n,Y_n]$.
We will deduce that $(G, \bp_1)$ has a $b$-motion  at $v_0$ by  applying Lemma 
\ref{lem:abstract_bad_pinning} to 
$M$.

Let
$
T=\{ \bp: 
 \text{$M( \bp)$ is non-singular}
 \}$ and let $S_1$ be the set of all $\bp\in S$ such that $\bp(v_a)$, $\bp(v_b)$ and $\bp(v_c)$ have distinct $y$-coordinates. 
 Lemma~\ref{lem:1dof_motion1} implies that $\tilde{C}(G+F, \bp)$ is non-singular for all $ \bp\in S_1$ so $S_1\subseteq T$.
Let $D$ be the set of rows of $M$ indexed by  
$$F\cup \{\be_{s,t}:(s,t)\in \{(a,1), (a, 2), (a,3), (b,1), (b,2), (c,1)\},$$
using the row labelling defined in Section \ref{subsec:pinning}. 
Then every entry of $M_{D,\overline{U}}$ is zero, and  $M_{\overline{D}}(\bp)=C(G,\bp)$ for all realisations $\bp$. Since $(G,\bp)$ has a $b$-motion for some  generic $\bp$,  Lemma  
\ref{lem:abstract_bad_pinning} tells us that 
$(G,\bp)$ has a $b$-motion for all $\bp\in T$. In particular $(G,\bp_1)$ has a $b$-motion. 
\end{proof}

\paragraph{Proof of statement (\ref{eq:special}).}
Since $q_{\rm bad}$ is a bad motion of $(G,\bp)$ at $v_0$, there exists $b:\hat N(v_0)\to \rat[X_0,Y_0,\ldots,X_5,Y_5]^3$ which satisfies (\ref{eq:bad}) and has $q_{\rm bad}(v)=b(v)$ for all $v\in \hat N_G(v_0)$. 

Consider the family of realisations $\bp_t$ of $G$ parametrized by $t\in \mathbb{R}$ defined by 
 $\bp_t(w)=\bp(w)$ for $w\in V_0$
and $\bp_t(u_0)=t\bp(u_0)+(1-t)\bp(u_1)$. (This family was used in the proof of \cite[Theorem 10.2.7]{Wsurvey} to show that 
vertex-splitting preserves $C_2^1$-independence.) 
Since $\bp$ is generic and $\bp_t|_{V_0}=\bp|_{V_0}$, 
we can choose three distinct vertices $v_a,v_b,v_c\in \hat N_G(v_0)$ which have distinct `$y$-coordinates' in $(G,\bp_t)$ for all $t\in \mathbb R$.  
Let $\tilde{C}(G+F,\bp_t)$ be the extended $C_2^1$-cofactor matrix for $(G+F,\bp_t)$ with respect to  $v_a,v_b,v_c$, and  let ${C}^*(G+F,\bp_t)$ be obtained from $\tilde{C}(G+F,\bp_t)$
by replacing the row indexed by the edge $u_0u_1$ with a row of the form 
$$
\kbordermatrix{
 & & u_0 & & u_1 & \\
 e=u_0u_1 & 0\dots 0 & D(\bp(u_0),\bp(u_1)) & 0 \dots 0 & -D(\bp(u_0),\bp(u_1)) & 0\dots 0 
}.
$$
Since  $D(\bp_t(u_0),\bp_t(u_1))=t^2 D(\bp(u_0), \bp_t(u_1))$,
$\ker  \tilde{C}(G,\bp_t)=\ker {C}^*(G,\bp_t)$ for $t\neq 0$.
Moreover, when $t=0$,  the  argument for vertex splitting  in 
\cite[Theorem 10.2.7]{Wsurvey}
can be applied directly to show that ${C}^*(G+F,\bp_0)$ is non-singular.\footnote{
Since $(G_0+F+e_1+e_2,\bp|_{V_0})$ is minimally $C_2^1$-rigid, Lemma~\ref{lem:1dof_motion1}  implies that  $\tilde{C}(G_0+F+e_1+e_2,\bp|_{V_0})$ is row independent. 
We may now use elementary matrix manipulations and the fact that $\bp_0(u_0)=\bp(u_1)$ 
to deduce that 
${C}^*(G+F,\bp_0)$ is row independent and hence non-singular, see the proof 
of  \cite[Theorem 10.2.7]{Wsurvey}
for more details.}

Observe that $\bp_t$ is generic over $\rat$ whenever $t$ is taken to be generic over $\mathbb{Q}(\bp(V))$.
Together with Lemma~\ref{lem:bad_pinning} and the facts that $\bp$ is generic over $\rat$ and $(G,\bp)$ has a $b$-motion, this implies that  $(G, \bp_t)$ has a $b$-motion for any  $t$ which is generic over $\mathbb{Q}(\bp(V))$. 

We next consider $M:={C}^*(G+F,\bp_t)$ to be a matrix with entries in $\F[t]$ where $\F=\rat(p(V))$ and $t$ is an indeterminate. 
Let $T=\{t\in \mathbb{R} : \mbox{$M(t)$ is non-singular}\}$.
Then $0\in T$ and $t\in T$ for any generic $t$ over $\mathbb{Q}(\bp(V))$.
Let $U=\hat N_G(v_0)$ and $D$ be the set of rows of $M$ indexed by  
$F\cup \{\be_{h,k}:(h,k)\in \{(a,1), (a, 2), (a,3), (b,1), (b,2), (c,1)\}$,
using the row labelling defined in Section \ref{subsec:pinning}. 
Then every entry of $M_{D,\overline{U}}$ is zero, and  $M_{\overline{D}}(t)=C^*(G,\bp_t)$ for all  $t\in \R$. Since $(G,\bp_t)$ has a $b$-motion and $\bp_t\in T$ for any  $t$ which is generic over $\mathbb{Q}(\bp(V))$, and $0\in T$,  Lemma  
\ref{lem:abstract_bad_pinning} implies that there exists a $z\in \ker M(0)$ such that $z(v)=b(v)$ for all $v\in U=\hat N_G(v_0)$. The fact that 
$M_{\overline{D}}(0)=C^*(G,\bp_0)$ and the defininition of $\bp_0$ now imply that
 $z|_{V_0}$ is a bad motion of $(G_0+e_1+e_2,\bp|_{V_0})$  so (\ref{eq:special}) holds.
\qed

\section{Calculation in the proof of Lemma~\ref{lem:bad_motion}}\label{app:bad}
We use the same notation for the polynomials $D_{i,j}$ and $\Delta_{i,j,k}$, and the polynomial map $b$ as in the proof of Lemma~\ref{lem:bad_motion}. 
We need to show that:
$$\text{the graph on $N_G(v_0)$ with edge set 
$\{v_iv_j: 
D_{v_i,v_j}\cdot (b(v_i)-b(v_j))=0 \}$
 is a star.}
$$
Since $\bp$ is generic   and the $C^1_2$-motion $\bq$ of $(G,p)$ defined in the conclusion of  Lemma~\ref{lem:motion} satisfies $\bq(v_i)=b(v_i)$ for all $v_i\in N_G(v_0$), it will suffice to show that 
$$\text{the graph on $N_G(v_0)$ with edge set 
$\{v_iv_j: 
D_{v_i,v_j}\cdot (\bq(v_i)-\bq(v_j))=0 \}$
 is a star,}
$$
where we are abusing notation by continuing to use  $D_{ij}$ for the real number obtained from the polynomial $D(i,j)$ by the  substitution $(X_i,Y_i)\to (x_i,y_i)=\bp(v_i)$ for all $v_i\in \hat N_G(v_0)$. 
The hypotheses that $\bq$ is a $C^1_2$-motion of $(G,\bp)$ and $v_0$ is a type $(\star)$ vertex in $G$ imply that
  $D_{v_i,v_5}\cdot (\bq(v_i)-\bq(v_5))=0$ for all $1\leq i\leq 4$. Hence it only remains to check that $\delta_{ij}:=D_{v_i,v_j}\cdot (\bq(v_i)-\bq(v_j))\neq 0$ for all $1\leq i<j\leq 4$.  We can do this for each $\delta_{i,j}$ as follows, where we use $c_{ij}$ and $c_{ij}'$ to denote non-zero constants.
\begin{align*}
\delta_{1,2}&=D_{1,2}\cdot (\bq(v_1)-\bq(v_2))=c_{1,2} D_{1,2}\cdot (D_{2,4}\times D_{2,5}) 
=c_{1,2}\left|
\begin{array}{c}
D_{1,2} \\ D_{2,4} \\ D_{2,5}
\end{array}\right| \neq 0. \\
\delta_{1,3}&\neq 0 \quad \text{(by a symmetric argument)}. \\
\delta_{1,4}&=D_{1,4}\cdot (\bq(v_1)-\bq(v_4))=-\alpha\Delta_{1,2,4} D_{1,4}\cdot (D_{3,4}\times D_{4,5})-\beta\Delta_{1,3,4} D_{1,4}\cdot (D_{2,4}\times D_{4,5}) \\
&=-\alpha \Delta_{1,2,4} \left|
\begin{array}{c}
D_{1,4} \\ D_{3,4} \\ D_{4,5}
\end{array}\right| 
-
\beta
\Delta_{1,3,4} \left|
\begin{array}{c}
D_{1,4} \\ D_{2,4} \\ D_{4,5}
\end{array}\right| \\
&=-\Delta_{2,5,0}\Delta_{0,1,5}\Delta_{0,5,3}\Delta_{4,5,1}\Delta_{2,4,5}\Delta_{1,2,4}\Delta_{4,3,5}\Delta_{4,1,3}(\Delta_{2,0,4}\Delta_{0,3,1}-\Delta_{3,0,4}\Delta_{0,2,1})\\
&\neq 0,
\end{align*}
where the third equation follows from the Vandermonde identity and the last relation holds since
$\Delta_{2,0,4}\Delta_{0,3,1}\neq \Delta_{3,0,4}\Delta_{0,2,1}$. 
(To see this, consider the special position where $\bp(v_3)$ is on the line through $\bp(v_0)$ and $\bp(v_1)$. 
Then the left term is zero while the right term is non-zero when the other points are in a generic position.) 
\begin{align*}
\delta_{2,3}&=D_{2,3}\cdot (\bq(v_2)-\bq(v_3)) =D_{2,3}\cdot (\beta \Delta_{1,3,2} D_{2,4}\times D_{2,5} - \alpha \Delta_{1,2,3}D_{3,4} \times D_{3,5}) \\
&=-\Delta_{1,2,3}\left( 
\beta\left|
\begin{array}{c}
D_{2,3} \\ D_{2,4} \\ D_{2,5}
\end{array}\right|
+\alpha\left|
\begin{array}{c}
D_{3,2} \\ D_{3,4} \\ D_{3,5}
\end{array}\right|
\right) \\
&=\Delta_{1,2,3} \Delta_{2,3,4}\Delta_{2,4,5}\Delta_{2,5,3}\Delta_{3,4,5}\Delta_{2,5,0}\Delta_{0,1,5}\Delta_{0,5,3}(\Delta_{3,0,4} \Delta_{0,2,1} -\Delta_{2,0,4} \Delta_{0,3,1}) \neq 0.\\
\delta_{2,4}&=D_{2,4}\cdot (\bq(v_2)-\bq(v_4))\\
&=D_{2,4} \cdot (c_{24} D_{2,4}\times D_{2,5} -\alpha \Delta_{1,2,4} D_{3,4} \times D_{4,5} - c_{24}' D_{2,4} \times D_{4,5}) \\
&=-\alpha \Delta_{1,2,4}  \left|
\begin{array}{c}
D_{2,4} \\ D_{3,4} \\  D_{4,5}
\end{array}\right| \neq 0. \\
\delta_{3,4}&\neq 0 \quad \text{(by a symmetric argument)}. 
\end{align*}  

\section{Projective Transformations}\label{sec:projective}

We give a new analysis of the projective invariance of $C_2^1$-rigidity from the viewpoint of $C_2^1$-motions, and then use this to prove Lemma~\ref{lem:projective_bad}.

Let $\mathbb{S}^{3\times 3}$ be the set of symmetric $3\times 3$ matrices. We will use the fact  that ${\rm Trace}(AB)=\sum_{1\leq i, j\leq 3} a_{ij}b_{ij}$ holds for any $A=(a_{ij}), B=(b_{ij})\in \mathbb{S}^{3\times 3}$.

\subsection{\boldmath $C_2^1$-motions in the projective setting}
Recall that each point $p_i=(x_i,y_i)^\top\in \R^2$ is associated with the point $p_i^{\uparrow}=[x_i,y_i,1]^\top$ in 2-dimensional real projective space $\bP^2$. This allows us to associate a framework $(G,\bp^{\uparrow})$ in $\bP^2$ with any framework $(G,\bp)$ in $\R^2$.   
We will consider $(G,\bp^{\uparrow})$ to be a framework in $\R^3$ in which the `$z$-component' of every vertex is non-zero and  define a projective  $C_2^1$-motion as a new kind of motion of such a framework.
  
Let $G=(V,E)$ be a graph and $\tilde{\bp}:V \to \mathbb{R}^3$ such that $\tilde \bp(v_i)=(x_i, y_i, z_i)^{\top}$ and
$z_i\neq 0$ for all $v_i\in V$. 
We say that $\bQ:V\rightarrow \mathbb{S}^{3\times 3}$ is a {\em projective  $C_2^1$-motion} of $(G, \tilde{\bp})$ if 
\begin{equation}\label{eq:projective}
{\rm Trace}\left((\tilde{\bp}(v_i)\times \tilde{\bp}(v_j))(\tilde{\bp}(v_i)\times \tilde{\bp}(v_j))^{\top} (\bQ(v_i)-\bQ(v_j))\right)=0 
\end{equation}
for all $v_iv_j\in E$, where $(\tilde{\bp}(v_i)\times \tilde{\bp}(v_j))$ denotes the cross product of $\tilde{\bp}(v_i)$ and $\tilde{\bp}(v_j)$.
A projective $C^1_2$-motion $\bQ$ is \emph{trivial} if \eqref{eq:projective} holds for all $v_i,v_j\in V$. 
These definitions are inspired by the following observation (and Lemma~\ref{lem:projective2} below).

\begin{lemma}\label{lem:projective1}
For a framework $(G,\bp)$ in $\R^2$, 
define $\bp^{\uparrow}:V\rightarrow \mathbb{R}^3$ by 
$\bp^{\uparrow}(v_i)=(\bp(v_i), 1)^\top$.
For $\bq:V\rightarrow \mathbb{R}^3$, define $\bQ:V\rightarrow \mathbb{S}^{3\times 3}$ by 
\begin{equation}\label{eq:lifting}
\bq(v_i)=\begin{pmatrix} 
q_1(v_i) \\ q_2(v_i)\\ q_3(v_i)
\end{pmatrix}
\mapsto
\bQ(v_i)=\begin{pmatrix}
2q_3(v_i) & -q_2(v_i) & 0 \\
-q_2(v_i) & 2q_1(v_i) & 0 \\
0 & 0 & 0
\end{pmatrix} \mbox{for all $v_i\in V$.}
\end{equation}
Then $\bq$ is a $C^1_2$-motion of $(G,\bp)$ if and only if $\bQ$ is a projective  $C_2^1$-motion of $(G,\bp^{\uparrow})$.
\end{lemma}
\begin{proof}
Denote $\bp(v_i)=(x_i ,y_i )^\top$ for all $v_i\in V$. Then
\[
(\bp^{\uparrow}(v_i)\times \bp^{\uparrow}(v_j))(\bp^{\uparrow}(v_i)\times \bp^{\uparrow}(v_j))^{\top}
=\begin{pmatrix}
(y_i-y_j)^2 & -(x_i-x_j)(y_i-y_j) & \ast \\
-(x_i-x_j)(y_i-y_j) & (x_i-x_j)^2 & \ast \\
\ast & \ast & \ast
\end{pmatrix}.
\]
It is now straightforward to check that  
$${\rm Trace}\left((\bp^{\uparrow}(v_i)\times \bp^{\uparrow}(v_i))(\bp^{\uparrow}(v_i)\times \bp^{\uparrow}(v_i))^{\top} (\bQ(v_i)-\bQ(v_j))\right)=2D(\bp(v_i), \bp(v_j))\cdot (\bq(v_i)-\bq(v_j))$$
for all $v_i,v_j\in V$.
\end{proof}

Lemma~\ref{lem:projective1} gives a linear isomorphism between the $C_2^1$-motions $\bq$ of $(G,\bp)$ and the projective $C_2^1$-motions $\bQ$ of $(G,\bp^{\uparrow})$
with the property  that the right column and bottom row of $\bQ$ are both zero.

\subsection{\boldmath Projective $C_2^1$-rigidity}
We show that every framework $(G,\tilde \bp)$ in $\R^3$ has at least $3|V|+6$ linearly independent projective $C_2^1$-motions.
We use this to define projective $C_2^1$-rigidity and then show that a given framework $(G,\bp)$ in $\R^2$ is $C_2^1$-rigid if and only if the corresponding framework
$(G,\bp^\uparrow)$ in $\R^3$ is projective $C_2^1$-rigid.

%
We defined six linearly independent (trivial) $C_2^1$-motions $\bq_i^*$ for an arbitrary 2-dimensional framework in (\ref{eq:rigid_motions1}) and (\ref{eq:rigid_motions2}).
We may apply Lemma \ref{lem:projective1} to each of these to obtain the following 
six linearly independent  projective $C_2^1$-motions for any framework $(G,\tilde \bp)$ in $\R^3$.
%
\begin{equation}\label{eq:projective_trivial}
\begin{gathered}
\bQ_1^*(v_i)=\begin{pmatrix}
0 & 0 & 0 \\
0 & 2 & 0 \\
0 & 0 & 0 
\end{pmatrix}, 
\bQ_2^*(v_i)=\begin{pmatrix}
0 & -1 & 0 \\
-1 & 0 & 0 \\
0 & 0 & 0 
\end{pmatrix}, 
\bQ_3^*(v_i)=\begin{pmatrix}
2 & 0 & 0 \\
0 & 0 & 0 \\
0 & 0 & 0 
\end{pmatrix}, \\
\bQ_4^*(v_i)=\begin{pmatrix}
0 & x_i/z_i & 0 \\
x_i/z_i & 2 y_i/z_i & 0 \\
0 & 0 & 0 
\end{pmatrix}, 
\bQ_5^*(v_i)=\begin{pmatrix}
2x_i/z_i & y_i/z_i & 0 \\
y_i/z_i & 0 & 0 \\
0 & 0 & 0 
\end{pmatrix}, \\
\bQ_6^*(v_i)=\begin{pmatrix}
2x_i^2/z_i^2 & 2x_iy_i/z_i^2 & 0 \\
2x_iy_i/z_i^2 & 2y_i^2/z_i^2 & 0 \\
0 & 0 & 0 
\end{pmatrix}.
\end{gathered}
\end{equation}
It is straightforward to check that each $\bQ_k^*$, $1\leq k\leq 6$, satisfies (\ref{eq:projective}) for every pair of vertices $v_i, v_j\in V$, and hence is a trivial projective $C_2^1$-motion of $(G,\tilde{\bp})$.
We may define a further $3|V|$ linearly independent projective $C_2^1$-motions for any $(G,\tilde \bp)$, three for each vertex, as follows.
For each $v_i\in V$, let $\bQ_{i,1}^*, \bQ_{i,2}^*, \bQ_{i,3}^*: V\rightarrow \mathbb{S}^{3\times 3}$ by: 
\begin{equation}\label{eq:trivial_vertex}
\begin{gathered}
\bQ_{i,1}^*(v_i)= \begin{pmatrix} 0 & x_i & 0 \\ x_i & 2y_i & z_i \\ 0 & z_i & 0 \end{pmatrix}, 
\text{ and } \bQ_{i,1}^*(v_j)=\begin{pmatrix} 0 & 0 & 0 \\ 0 & 0 & 0 \\ 0 & 0 & 0 \end{pmatrix} \text{ for all } v_j\neq v_i; \\
\bQ_{i,2}^*(v_i)= \begin{pmatrix} 2x_i & y_i & z_i \\ y_i & 0 & 0 \\ z_i & 0 & 0 \end{pmatrix}, \text{ and } 
\bQ_{i,2}^*(v_j)=\begin{pmatrix} 0 & 0 & 0 \\ 0 & 0 & 0 \\ 0 & 0 & 0 \end{pmatrix} \text{ for all } v_j\neq v_i; \\
\bQ_{i,3}^*(v_i)= \begin{pmatrix} -x_i^2 & -x_iy_i & 0 \\ -x_iy_i & -y_i^2 & 0 \\ 0 & 0 & z_i^2 \end{pmatrix}, \text{ and } 
\bQ_{i,3}^*(v_j)=\begin{pmatrix} 0 & 0 & 0 \\ 0 & 0 & 0 \\ 0 & 0 & 0 \end{pmatrix} \text{ for all  } v_j\neq v_i. 
\end{gathered}
\end{equation}
It is straightforward to check that each  $\bQ_{i,j}^*$ satisfies (\ref{eq:projective}) for all pairs of vertices $v_i, v_j\in V$, and hence is a trivial motion of $(G,\tilde{\bp})$.
We will refer to the space of projective $C_2^1$-motions generated by these $3|V|+6$ motions as the {\em space of trivial motions}, and say that a projective $C_2^1$-motion of $(G,\tilde \bp)$ is {\em non-trivial} if it does not belong to this space.
We define $(G,\tilde{\bp})$ to be {\em projectively $C_2^1$-rigid} if every projective $C_2^1$-motion of $(G,\tilde{\bp})$ is a trivial motion, or equivalently, if its space of projective $C_2^1$-motions has dimension $3|V|+6$.

The following lemma follows from Lemma~\ref{lem:projective1} (and the formulae for $\bQ_{i,j}^*$ given above).
\begin{lemma}\label{lem:projective2}
Let $(G,\bp)$ be a two-dimensional framework, and let $(G,\bp^{\uparrow})$ be as defined in Lemma~\ref{lem:projective1}.
Then $(G,\bp)$ is $C_2^1$-rigid if and only if $(G,\bp^{\uparrow})$ is projectively $C_2^1$-rigid.
\end{lemma}

\subsection{Projective invariance} \label{subsec:projective_motion}
In this subsection we show that 
the projective $C_2^1$-rigidity of a framework $(G,\tilde \bp)$ is invariant under a pointwise scaling and a non-singular linear transformation of $\mathbb{R}^3$.
These two facts combined with Lemma~\ref{lem:projective2} 
will imply the projective invariance of $C_2^1$-rigidity for 2-dimensional frameworks. 
 
We first show that the projective $C_2^1$-rigidity of $(G,\tilde \bp)$ is invariant by the scaling of each point.
Choose a non-zero scalar
$\lambda_i\in \R$ for each $v_i\in V$, and let $\lambda \tilde{\bp}$ be the point configuration in $\mathbb{R}^3$ defined by $(\lambda \tilde{\bp})(v_i)=\lambda_i \tilde{\bp}(v_i)$ for each $v_i\in V$.
Then $\bQ:V\rightarrow \mathbb{S}^{3\times 3}$ is a (non-trivial) motion of $(G,\tilde{\bp})$ if and only if 
$\bQ$ is a (non-trivial) motion of $(G,\lambda \tilde{\bp})$ since
\begin{align*}
&{\rm Trace}\left(((\lambda\tilde{\bp})(v_i)\times (\lambda\tilde{\bp})(v_j))((\lambda\tilde{\bp})(v_i)\times (\lambda\tilde{\bp})(v_j))^{\top} (\bQ(v_i)-\bQ(v_j))\right) \\
&=\lambda_i^2\lambda_j^2{\rm Trace}\left((\tilde{\bp}(v_i)\times \tilde{\bp}(v_j))(\tilde{\bp}(v_i)\times \tilde{\bp}(v_j))^{\top} (\bQ(v_i)-\bQ(v_j))\right).
\end{align*}
Thus projective $C_2^1$-rigidity is invariant by the scaling of each point.

To see the invariance under non-singular linear transformations, choose any non-singular matrix $A\in \mathbb{R}^{3\times 3}$ and let $C_A$ be its cofactor matrix.
It is well-known that $(Ax)\times (Ay)=C_A(x\times y)$ for any $x,y\in \mathbb{R}^3$.
Moreover, $C_A$ is non-singular if and only if $A$ is non-singular.
Given $\bQ:V\to  \mathbb{S}^{3\times 3}$
we let $\bQ_A:V\rightarrow \mathbb{S}^{3\times 3}$ by $\bQ_A(v_i)=C_A^{-\top} \bQ(v_i) C_A^{-1}$, where $C_A^{-\top}=(C_A^{-1})^{\top}$.
Then, for any $v_i, v_j\in V$,
\begin{align*} 
&{\rm Trace}\left((A\tilde{\bp}(v_i)\times A\tilde{\bp}(v_j))(A\tilde{\bp}(v_i)\times A\tilde{\bp}(v_j))^{\top} (\bQ_A(v_i)-\bQ_A(v_j))\right) \\
&= {\rm Trace}\left(C_A(\tilde{\bp}(v_i)\times \tilde{\bp}(v_j))(\tilde{\bp}(v_i)\times \tilde{\bp}(v_j))^{\top} C_A^{\top} C_A^{-\top}(\bQ(v_i)-\bQ(v_j)) C_A^{-1}\right) \\
&= {\rm Trace}\left(C_A\left((\tilde{\bp}(v_i)\times \tilde{\bp}(v_j))(\tilde{\bp}(v_i)\times \tilde{\bp}(v_j))^{\top} (\bQ(v_i)-\bQ(v_j)) \right) C_A^{-1} \right)\\
&= {\rm Trace}\left((\tilde{\bp}(v_i)\times \tilde{\bp}(v_j))(\tilde{\bp}(v_i)\times \tilde{\bp}(v_j))^{\top} (\bQ(v_i)-\bQ(v_j)) \right).
\end{align*} 
Thus $\bQ$ is a (non-trivial) projective $C_2^1$-motion of $(G, \tilde{\bp})$ if and only if $\bQ_A$ is a (non-trivial) projective $C_2^1$-motion of $(G, A\tilde{\bp})$.

Hence projective $C_2^1$-rigidity is invariant under both pointwise scaling and  non-singular linear transformations. This implies that the $C_2^1$-rigidity of 2-dimensional frameworks is projectively invariant.
 
\subsection{Proof of Lemma~\ref{lem:projective_bad}}


\newtheorem*{tlemma}{Lemma~\ref{lem:projective_bad}}

\begin{tlemma}
Let $(G,\bp)$ be a generic framework, $v_0$ be a vertex of degree five with $N_G(v_0)=\{v_1,\dots, v_5\}$, 
and $(G, \bp')$ be a projective image of $(G,\bp)$ such that 
$\bp'(v_1)=(1, 0),\ \bp'(v_2)=(0, 0),\ \bp'(v_3)=(0,1), \text{and } \bp'(v_4)=(1,1)$.
If $(G, \bp')$ has a bad motion at $v_0$, then $(G, \bp)$ has a bad motion at $v_0$.
\end{tlemma}
\begin{proof}
Since $\bp$ is generic,  the set of coordinates of the other vertices of $(G,\bp')$ is algebraically independent over $\rat$.
Suppose that $(G,\bp')$ has a bad motion $\bq'$ at $v_0$, i.e., there is a map $b':\hat N_G(v_0)\rightarrow \rat[X_0,\dots, Y_5]^3$ satisfying (\ref{eq:bad}) and such that $b'_{i}(\bp')=\bq'(v_i)$ for all $v_i\in \hat N_G(v_0)$, where $b'_{i}(\bp')$ denotes the evaluation of $b'(v_i)$ at $p'(\hat N_G(v_0))$.
Suppose also that the inverse of the projective map from $(G,\bp)$ to $(G,\bp')$ is represented by the matrix $A\in \mathbb{R}^{3\times 3}$. Note that each entry of $A$ can be expressed  as a rational function of $\bp(v_1),\dots, \bp(v_4)$ over $\mathbb{Q}$.
We consider the following procedure for converting the motion $\bq'$ to a motion $\bq$ of $(G,\bp)$:
\begin{align*}
\text{$C_2^1$-motion $\bq'$ of $(G,\bp')$} &\overset{\raise0.2ex\hbox{\textcircled{\scriptsize{1}}}}{\longrightarrow} \text{projective $C_2^1$-motion $\bQ'$ of $(G,(\bp')^{\uparrow})$} \\
&\overset{\raise0.2ex\hbox{\textcircled{\scriptsize{2}}}}{\longrightarrow} \text{projective $C_2^1$-motion $\bQ'_A$ of $(G,A(\bp')^{\uparrow})$} \\
&\overset{\raise0.2ex\hbox{\textcircled{\scriptsize{3}}}}{\longrightarrow} \text{projective $C_2^1$-motion $\bQ'_A$ of $(G,\bp^{\uparrow})$} \\
&\overset{\raise0.2ex\hbox{\textcircled{\scriptsize{4}}}}{\longrightarrow} \text{projective $C_2^1$-motion $\bQ$ of $(G,\bp^{\uparrow})$} \\
&\overset{\raise0.2ex\hbox{\textcircled{\scriptsize{5}}}}{\longrightarrow} \text{$C_2^1$-motion $\bq$ of $(G,\bp)$}.
\end{align*}
An explanation of each step in this procedure is given below.
\begin{itemize}
\item[\raise0.2ex\hbox{\textcircled{\scriptsize{1}}}] We construct the projective $C_2^1$-motion $\bQ'$ of $(G,(\bp')^{\uparrow})$ from $\bq'$ as in (\ref{eq:lifting}).
\item[\raise0.2ex\hbox{\textcircled{\scriptsize{2}}}] $\bQ'_A$ is a projective $C_2^1$-motion of $(G,A(\bp')^{\uparrow})$
as explained in Subsection~\ref{subsec:projective_motion}.
\item[\raise0.2ex\hbox{\textcircled{\scriptsize{3}}}] Since $A$ represents the projective transformation from $(G,\bp')$ to $(G,\bp)$, we can scale each point of $(G,A(\bp')^{\uparrow})$ to obtain $(G, \bp^{\uparrow})$. As explained in Subsection~\ref{subsec:projective_motion}, $\bQ'_A$ remains a projective $C_2^1$-motion after this pointwise scaling.
\item[\raise0.2ex\hbox{\textcircled{\scriptsize{4}}}] We eliminate non-zero entries in the right column and the bottom row in $\bQ'_A(v_i)$ by adding (scaled) trivial motions $\bQ_{i,1}^*, \bQ_{i,2}^*, \bQ_{i,3}^*$ for each $v_i\in V$.
This is always possible since $\bQ_{i,1}^*, \bQ_{i,2}^*, \bQ_{i,3}^*$ are of the form, 
$
\begin{pmatrix}
* & * & 0 \\ * & * & 1 \\ 0 & 1 & 0
\end{pmatrix},
\begin{pmatrix}
* & * & 1 \\ * & * & 0 \\ 1 & 0 & 0
\end{pmatrix},
\begin{pmatrix}
* & * & 0 \\ * & * & 0 \\ 0 & 0 & 1
\end{pmatrix}
$
respectively, 
by (\ref{eq:trivial_vertex}). (Note that $z_i=1$ in $\bp^{\uparrow}(v_i)$.)
Since we only add  trivial motions, the resulting $\bQ:V\rightarrow \mathbb{S}^{3\times 3}$ is still a projective $C_2^1$-motion of $(G,\bp^{\uparrow})$.
\item[\raise0.2ex\hbox{\textcircled{\scriptsize{5}}}] We construct the $C_2^1$-motion $\bq$ of $(G,\bp)$ from $\bQ$ as in  (\ref{eq:lifting}).
\end{itemize}
We can compute each $\bq(v_i)$ from $\bq'(v_i)$ by following the above procedure.
For each $v_i\in \hat N_G(v_0)$, the procedure can be performed at a symbolic level starting from $\bq'(v_i)=b_{i}'(\bp')$,
and  returning a map $b:\hat N_G(v_0)\rightarrow \rat(X_0,Y_0,X_1,\dots, Y_5)^3$ with $\bq(v_i)=b_{i}(\bp)$ for each $v_i\in N_G(v_0)$.
By scaling appropriately, we may suppose that $b(v_i)$ is a polynomial map for each $v_i\in N_G(v_0)$. 
Then $\bq$ is a $b$-motion of $(G,\bp)$.
In addition, for all $v_i, v_j\in V$, 
$D(\bp'(v_i),\bp'(v_j))\cdot (b_{i}'(\bp')-b_{j}'(\bp'))=0$ if and only if 
$D(\bp(v_i),\bp(v_j))\cdot (b_{i}(\bp)-b_{j}(\bp))=0$.
This can be verified by checking the corresponding formulae in each step of the above procedure.
Since $p$ is generic, this implies that $b$ satisfies (\ref{eq:bad}), and hence $\bq$ is a bad motion of $(G,\bp)$ at $v_0$.
\end{proof}

\section*{Acknowledgments} 
This work was supported by JST CREST Grant Number JPMJCR14D2, JSPS KAKENHI Grant Number 18K11155 and EPSRC overseas travel grant EP/T030461/1.

\bibliographystyle{amsplain}

\begin{thebibliography}{99}
%
%
%
\bibitem{AR78}
L. Asimow and B. Roth. \newblock The rigidity of graphs. \newblock {\em Trans.~Am.~Math.~Soc.},
245: 279--289, 1978. 

\bibitem{B88}
L.~J.~Billera. \newblock Homology of smooth splines: generic triangulations and a conjecture of Strang. \newblock {\em Trans.~Am.~Math.~Soc.},
310: 325--340, 1988. 

\bibitem{CJT1}
K.~Clinch, B.~Jackson, and S.~Tanigawa. \newblock Abstract 3-Rigidity and Bivariate $C_2^1$-Splines II: Combinatorial Characterization. \newblock {\em Discrete Analysis}, to appear.  ArXiv 1911.00205




\bibitem{C14}
J.~Cruickshank. \newblock On spaces of infinitesimal motions and three dimensional Henneberg extensions. \newblock {\em Disc.~Comp.~Geom.}, 51: 702--721, 2014.


\bibitem{Glu}
{ H. Gluck}.
\newblock Almost all simply connected closed surfaces are rigid. \newblock In
{\em Geometric Topology} (Proc. Conf., Park City, Utah, 1974).
{Lecture Notes in Math. 438,
Springer, Berlin}: 225--239, 1975.

\bibitem{G91} J.~E.~Graver. \newblock Rigidity matroids. \newblock {\em SIAM Journal on Discrete Mathematics}, 4:  355--368, 1991.

\bibitem{GSS93} J. E. Graver, B. Servatius, and H. Servatius. \newblock  {\em Combinatorial rigidity}.  \newblock {Amer.~Math.~Soc.}, 1993.


\bibitem{JJ05} B.~Jackson and T.~Jord{\'a}n.
\newblock The Dress conjectures on rank in the 3-dimensional rigidity matroid.  \newblock {\em Advances in Applied Mathematics},  35: 355--367, 2005.

\bibitem{JT}
B.~Jackson and S.~Tanigawa.
\newblock Maximal matroids in weak order posets. 2021. arXiv:2102.09901.

\bibitem{Lam}
G.~Laman. \newblock On graphs and the rigidity of skeletal structures. \newblock {\em J. Eng. Math.}, 
4: 331--340, 1970.

\bibitem{LY82}  
L.~Lov{\'a}sz and Y.~Yemini. \newblock On generic rigidity in the plane. \newblock {\em SIAM J.~Algebraic Discrete Methods}, 3: 91--98, 1982.

\bibitem{Max}  
J. C. Maxwell. \newblock On the calculation of the equilibrium and stiffness of frames. \newblock  {\em Philos. Mag.}, 27:  294--299, 1864.

\bibitem{N10} V.~H.~Nguyen.  \newblock  On abstract rigidity matroids. 
\newblock  {\em SIAM Journal on Discrete Mathematics}, 24: 363--369, 2010.

\bibitem{PAP} G.~Pap. A note on maximum matroids of graphs. 2020. The EGRES Quick-Proofs series, QP-2020-02.



\bibitem{P-G} H. Pollaczek-Geiringer. \newblock  \"{U}ber die Gliederung ebener Fachwerke. \newblock   {\em Zeitschrift f\"{u}r Angewandte Mathematik und Mechanik (ZAMM)}, 7: 58--72, 1927.


\bibitem{SW} B.~Schulze and W.~Whiteley. \newblock  Rigidity and scene analysis. \newblock  In {\em Handbook of Discrete and Computational Geometry}, 
Third Edition, J.~E.~Goodman, J.~O'Rourke, and C.~D.~T{\'o}th (Eds.),
CRC Press LLC, 2017.



\bibitem{SV} S.~Sitharam and A.~Vince.The maximum matroid of a graph. 2019, arXiv:1910.05390.



\bibitem{TW85} T.-S.~Tay and W.~Whiteley. \newblock  Generating  isostatic frameworks. \newblock  {\em Structural Topology}, 11: 20--69, 1985.


\bibitem{W90}
W.~Whiteley. Geometry of bivariate splines. 1990, technical report.
 


\bibitem{W91} 
W.~Whiteley. \newblock  Combinatorics of bivariate splines.\newblock   
{\em Applied Geometry and Discrete Mathematics --- the Victor Klee Festschrift}. DIMACS, vol. 4, AMS, 1991: 587--608.


\bibitem{Wsurvey} {W.~Whiteley}. \newblock  
Some matroids from discrete applied geometry. \newblock  In {\em Matroid Theory} (Seattle, WA, 1995),
Contemp. Math., 197, Amer. Math. Soc., Providence, RI, 1996: 171--311.


\end{thebibliography}


\begin{dajauthors}
\begin{authorinfo}[katie]
Katie Clinch\\Department of Electrical and Electronic Engineering, University of Melbourne\\ Parkville, Victoria, 3010, Australia.\\ 
 katie\imagedot{}clinch\imageat{}unimelb\imagedot{}edu\imagedot{}au
\end{authorinfo}
\begin{authorinfo}[bill]
Bill Jackson\\
School
of Mathematical Sciences, Queen Mary University of London\\
Mile End Road, London E1 4NS, England.\\
b\imagedot{}jackson\imageat{}qmul\imagedot{}ac\imagedot{}uk
\end{authorinfo}
\begin{authorinfo}[shin]
Shin-ichi Tanigawa\\
Department of Mathematical Informatics\\ 
Graduate School of Information Science and Technology, University of Tokyo\\ 
7-3-1 Hongo, Bunkyo-ku, 113-8656,  Tokyo, Japan. \\ 
tanigawa\imageat{}mist\imagedot{}i\imagedot{}u-tokyo\imagedot{}jp
  \end{authorinfo}
\end{dajauthors}

\end{document}